\documentclass{amsart}
\usepackage{amssymb,mathtools}
\usepackage[british]{babel}
\usepackage{enumitem}
\usepackage[latin1]{inputenc}
\usepackage{url}
\usepackage{hyperref}
\usepackage{tikz-cd}

\newtheorem{theorem}{Theorem}
\numberwithin{theorem}{section}
\newtheorem{corollary}[theorem]{Corollary}
\newtheorem{lemma}[theorem]{Lemma}
\newtheorem{proposition}[theorem]{Proposition}
\theoremstyle{definition}
\newtheorem{definition}[theorem]{Definition}

\newtheorem{example}[theorem]{Example}

\newtheorem{summary}[theorem]{Summary}

\newcommand{\rca}{\mathbf{RCA}}
\newcommand{\aca}{\mathbf{ACA}}
\newcommand{\supp}{\operatorname{supp}}
\newcommand{\otp}{\operatorname{otp}}

\newcommand{\id}{\operatorname{Id}}
\newcommand{\lef}{<^{\operatorname{fin}}}
\newcommand{\leqf}{\leq^{\operatorname{fin}}}

\newcommand{\rng}{\operatorname{rng}}

\newcommand{\en}{\operatorname{en}}
\newcommand{\lkb}{<_{\operatorname{KB}}}
\newcommand{\len}{\operatorname{len}}
\newcommand{\hth}{\operatorname{ht}}

\title[Derivatives of normal functions in reverse mathematics]{Derivatives of normal functions\\ in reverse mathematics}
\author{Anton Freund and Michael Rathjen}

\address{Anton Freund, Department of Mathematics, Technical University of Darmstadt, Schloss\-garten\-str.~7, 64289~Darmstadt, Germany}
\email{freund@mathematik.tu-darmstadt.de}

\address{Michael Rathjen, Department of Pure Mathematics, University of Leeds, Leeds LS2\,9JT, United Kingdom}
\email{rathjen@maths.leeds.ac.uk}

\begin{document}

\begin{abstract}
Consider a normal function $f$ on the ordinals (i.\,e.~a function~$f$ that is strictly increasing and continuous at limit stages). By enumerating the fixed points of~$f$ we obtain a faster normal function $f'$, called the derivative of~$f$. The present paper investigates this important construction from the viewpoint of reverse mathematics. Within this framework we must restrict our attention to normal functions $f:\aleph_1\rightarrow\aleph_1$ that are represented by dilators (i.\,e.~particularly uniform endofunctors on the category of well-orders, as introduced by J.-Y.~Girard). Due to a categorical construction of P.~Aczel, each normal dilator $T$ has a derivative~$\partial T$. We will give a new construction of the derivative, which shows that the existence and fundamental properties of $\partial T$ can already be established in the theory $\rca_0$. The latter does not prove, however, that $\partial T$ preserves well-foundedness. Our main result shows that the statement ``for every normal dilator $T$, its derivative~$\partial T$ preserves well-foundedness'' is $\aca_0$-provably equivalent to $\Pi^1_1$-bar induction (and hence to $\Sigma^1_1$-dependent choice and to \mbox{$\Pi^1_2$-reflection} for $\omega$-models).
\end{abstract}


\maketitle

{\let\thefootnote\relax\footnotetext{\copyright~2020. This manuscript version is made available under the CC-BY-NC-ND 4.0 license \url{http://creativecommons.org/licenses/by-nc-nd/4.0/}. It is the accepted version of a paper published in the Annals of Pure and Applied Logic 172(2) 2021, article no. 102890, 49 pp., \href{https://doi.org/10.1016/j.apal.2020.102890}{doi:10.1016/j.apal.2020.102890}.}}

\section{Introduction}

For the purpose of this paper, a normal function is a function $f:\aleph_1\rightarrow\aleph_1$ that is strictly increasing and continuous at limit stages, i.\,e.~we demand that
\begin{enumerate}[label=(\roman*)]
 \item $\alpha<\beta$ implies $f(\alpha)<f(\beta)$ and that
 \item $f(\lambda)=\sup_{\alpha<\lambda}f(\alpha)$ holds for any limit ordinal $\lambda$.
\end{enumerate}
Equivalently, $f$ is the unique strictly increasing enumeration of a closed and unbounded (club) subset of~$\aleph_1$. It is easy to see that the fixed points of any normal function~$f$ do again form an $\aleph_1$-club. The normal function that enumerates these fixed points is called the derivative of $f$ and is denoted by $f'$. Let us agree to call a normal function $g$ an upper derivative of $f$ if $f(g(\alpha))=g(\alpha)$ holds for any ordinal~$\alpha<\aleph_1$. Note that such a function~$g$ must majorize the derivative $f'$ of~$f$. As an example we consider the function $f(\alpha)=\omega^\alpha$ from ordinal arithmetic. In this case $f'(\alpha)=\varepsilon_\alpha$ is the $\alpha$-th $\varepsilon$-number. The notion of normal function plays an important role in proof theory (see~e.\,g.~\cite[Chapter~V]{schuette77}) and has interesting computability-theoretic properties (due to~\cite{marcone-montalban}). More generally, one can consider normal functions on the class of ordinals. A fundamental example from set theory is the function $f(\alpha)=\aleph_\alpha$ that enumerates the infinite cardinals. However, such normal functions are beyond the scope of the present paper.

The above construction of derivatives via clubs uses the fact that~$\aleph_1$ is a regular cardinal. One can also build derivatives by transfinite recursion, which relies on collection or replacement. In the present paper we construct derivatives in a much weaker setting. In very informal terms, we show that the following statements are equivalent over a suitable base theory:
\begin{enumerate}[label=(\arabic*$^0$)]
 \item Every normal function has a derivative.
 \item Every normal function has an upper derivative.
 \item Transfinite induction holds for any $\Pi^1_1$-formula.
\end{enumerate}
To establish this equivalence we will give precise sense to the following argument: To see that~(1$^0$) implies~(2$^0$) it suffices to observe that any derivative is an upper derivative. To prove the direction from~(2$^0$) to~(3$^0$) we must establish induction for a $\Pi^1_1$-formula $\varphi(\gamma)$ up to an arbitrary ordinal~$\alpha$. Using the Kleene normal form theorem one obtains countable trees $\mathcal T_\gamma$ with
\begin{equation*}
\varphi(\gamma)\leftrightarrow\text{``$\mathcal T_\gamma$ is well-founded"}.
\end{equation*}
The assumption that $\varphi$ is progressive along the ordinals can then be expressed as
\begin{equation*}
\forall_{\beta<\gamma}\text{``$\mathcal T_\beta$ is well-founded"}\rightarrow\text{``$\mathcal T_\gamma$ is well-founded"}.
\end{equation*}
Assume that this statement is witnessed by a binary function $h$ on the ordinals, in the sense that we have
\begin{equation*}
\forall_{\beta<\gamma}\otp(\mathcal T_\beta)\leq\delta\rightarrow\otp(\mathcal T_\gamma)\leq h(\gamma,\delta)
\end{equation*}
for any ordinals $\gamma$ and $\delta$, where $\otp(\mathcal T)$ denotes the order type of~$\mathcal T$. To avoid the dependency on~$\gamma$ we set $h_0(\delta)=\sup_{\gamma<\alpha}h(\gamma,\delta)$ (alternatively one could set $h_0(\delta)=\sup_{\gamma\leq\delta} h(\gamma,\delta)$, avoiding the reference to the fixed bound $\alpha$). Now form a normal function~$f$ with $h_0(\delta)\leq f(\delta+1)$, e.\,g.~by setting $f(\delta)=\sum_{\gamma<\delta}1+h_0(\gamma)$. Then statement~(2$^0$) allows us to consider an upper derivative~$g$ of $f$. Let us show that we have
\begin{equation*}
\otp(\mathcal T_\gamma)\leq g(\gamma+1)
\end{equation*}
for all $\gamma<\alpha$. For $\beta<\gamma$ we may inductively assume $\otp(\mathcal T_\beta)\leq g(\beta+1)\leq g(\gamma)$. By the above we obtain
\begin{equation*}
\otp(\mathcal T_\gamma)\leq h(\gamma,g(\gamma))\leq h_0(g(\gamma))\leq f(g(\gamma)+1)\leq f(g(\gamma+1))=g(\gamma+1).
\end{equation*}
So $g$ witnesses that $\mathcal T_\gamma$ is well-founded for any $\gamma<\alpha$. This yields $\forall_{\gamma<\alpha}\varphi(\gamma)$, which is the conclusion of transfinite induction. To see that~(3$^0$) implies~(1$^0$) we will construct notation systems for the values~$f'(\alpha)$, relative to a given normal function~$f$. The crucial fact that the notation system for $f'(\alpha)$ is well-founded (and hence represents an ordinal) will be established by transfinite induction on~$\alpha$.

In order to make the result from the previous paragraph precise we will use the framework of reverse mathematics. This research program uncovers equivalences between different mathematical and foundational statements in the language of second order arithmetic (see~\cite{simpson09} for an introduction). As the base theory for our investigation we take~$\aca_0$. In second order arithmetic the above statement~(3$^0$) corresponds to the following assertion:
\begin{enumerate}
\item[(3)] Induction for $\Pi^1_1$-formulas is available along any countable well-order.
\end{enumerate}
We will refer to this assertion as $\Pi^1_1$-bar induction, in order to distinguish it from the principle of transfinite induction along a specific (class- or set-sized) well-order. Let us recall that $\Pi^1_1$-bar induction is well-established in reverse mathematics: Simpson~\cite{simpson-bar-induction} has shown that it is equivalent to $\Sigma^1_1$-dependent choice and to $\Pi^1_2$-reflection for $\omega$-models, also over~$\aca_0$.

To formalize statements (1$^0$) and (2$^0$) in second-order arithmetic we will rely on \mbox{J.-Y.}~Girard's notion of dilator~\cite{girard-pi2,girard-intro}. For the purpose of the present paper, a (coded) prae-dilator is a particularly uniform functor $n\mapsto T_n$ from natural numbers to linear orders (full details can be found in Section~\ref{sect:normal-dils-so} below). Girard has observed that the uniformity allows to extend $T$ beyond the natural numbers. In~\cite{freund-thesis,freund-computable} the first author has given a detailed description of linearly ordered notation systems $D^T_X$ that are computable in $T$ and the linear order~$X$. This yields an endofunctor $X\mapsto D^T_X$ of linear orders, which one may call a class-sized prae-dilator. If $D^T_X$ is well-founded for every well-order~$X$, then $T$ is called a (coded) dilator. In this case $\alpha\mapsto\otp(D^T_\alpha)$ defines a function on the ordinals. A condition under which this function is normal has been identified by P.~Aczel~\cite{aczel-phd,aczel-normal-functors} (even before Girard had introduced dilators in the full sense). This leads to a notion of normal prae-dilator, which will be defined in Section~\ref{sect:normal-dils-so}. In the same section we will characterize (upper) derivatives on the level of normal prae-dilators. Once all this is made precise, we can take the following as our formalization of statement~(2$^0$) in second order arithmetic:
\begin{enumerate}
\item[(2)] Any normal dilator $T$ has an upper derivative $S$ such that $X\mapsto D^S_X$ preserves well-foundedness (so that $S$ is again a normal dilator).
\end{enumerate}
The advantage of this principle is that it is relatively easy to state and does not depend on a specific construction of derivatives. Its disadvantage is that it confounds the following two questions: How strong is the assertion that any normal dilator has an upper derivative? And how much strength is added by the demand that the upper derivative preserves well-foundedness? We want to disentangle these questions in our formalization of statement~(1$^0$). To see how this works, let us recall that Aczel~\cite{aczel-phd,aczel-normal-functors} has explicitly constructed a derivative~$\partial T$ of a given normal prae-dilator~$T$. In Section~\ref{sect:constructin-derivative} we will show that $\partial T_n$ can be represented by a term system. In view of this representation $\rca_0$ proves that $\partial T$ exists and is a derivative of~$T$. What $\rca_0$ cannot show is that $D^{\partial T}$ preserves well-foundedness whenever $D^T$ does. This suggests to formalize statement~(1$^0$) as the following assertion:
\begin{enumerate}
\item[(1)] If $T$ is a normal dilator, then $D^{\partial T}_X$ is well-founded for any well-order~$X$.
\end{enumerate}
As $\rca_0$ proves that any derivative is an upper derivative it will be immediate that (1) implies~(2). This means that the entire strength of these two principles is concentrated in the preservation of well-foundedness, which answers the questions that we have raised after the formulation of principle~(2).

Let us summarize the content of the following sections: As explained above, Section~\ref{sect:normal-dils-so} introduces (upper) derivatives on the level of normal prae-dilators and gives a precise formalization of statement~(2). In Section~\ref{sect:upper-deriv-bi} we prove that~(2) implies~(3), by giving precise meaning to the argument from the beginning of this introduction. Section~\ref{sect:constructin-derivative} contains the construction of $\partial T$ in $\rca_0$, which yields the implication from~(1) to~(2). In Section~\ref{sect:bi-deriv-wf} we prove that (3) implies~(1), using $\Pi^1_1$-induction along $X$ to establish that $D^{\partial T}_X$ is well-founded. At the end of the paper we will thus have shown that (1), (2) and (3) are equivalent over $\aca_0$ (see Theorem~\ref{thm:main-result} for the official statement of this result). In a separate paper by the first author~\cite{freund-derivatives-rca} it is shown that the base theory can be lowered to $\rca_0$, since the existence of derivatives implies arithmetical comprehension. The first author has also shown that $\Pi^1_1$-induction along the natural numbers is equivalent to the weaker principle that every normal dilator has at least one well-founded fixed point~\cite{freund-single-fixed-point}.

The length of our paper is due to the fact that we wanted to be precise about the formalization of dilators in second order arithmetic. The text includes several informal summaries, so that the reader can follow the main lines of the argument without reading all technical verifications. To get a good first idea, we recommend to read Section~\ref{sect:normal-dils-so} up to (and including) Summary~\ref{sum:class-sized-dilators}, skip the rest of Section~\ref{sect:normal-dils-so}, and then read Section~\ref{sect:upper-deriv-bi} up to Summary~\ref{sum:deduce-Pi11-bi}.

In the rest of this introduction we put our result into context. Let us first discuss implications for the predicative foundation of mathematics: The predicative stance originated with H.~Weyl's {\em Das Kontinuum} \cite{weyl-continuum} from 1908, and  may be characterized by the imposition of a constraint on set formation that countenances only that which is implicit in accepting the natural number structure as a completed totality. Based on a proposal due to G.~Kreisel, the modern logical analysis of predicativity (given the natural numbers) was carried out by S.~Feferman~\cite{feferman64} and K.~Sch\"utte~\cite{schuette64} in 1964. It is couched in terms of provability in an autonomous transfinite progression of ramified theories of sets which are based on classical logic and assume the existence of the set of natural numbers. The existence of further sets is regimented by a hierarchy of levels to be generated in an autonomous way. At each level, sets are asserted to exist only via definitions in which quantification over sets must be restricted to lower levels. The further condition of autonomy requires that one may ascend to a level $\alpha$ only if the existence of a well-ordering of order type $\alpha$ has been established at some level $\beta<\alpha$. Feferman and Sch\"utte independently showed that the least non-autonomous ordinal for this progression of theories is the recursive ordinal $\Gamma_0$. Set-theoretically, the constructible sets up to $\Gamma_0$ form the minimal model of the aforementioned progression. A connection with our result arises because derivatives of normal functions (and transfinite hierarchies of derivatives) provide the intuition behind the usual notation system for~$\Gamma_0$ (see e.\,g.~\cite{schuette77}). This does not imply, however, that the abstract notion of derivative (relative to an arbitrary normal function) is predicatively acceptable. Indeed our result shows that it is not: We prove that the existence of normal functions is equivalent to $\Pi^1_1$-bar induction and hence to $\Sigma^1_1$-dependent choice (all over $\aca_0$). Now the least ordinal~$\tau$ such that the constructible sets up to $\tau$ form a model of $\Sigma^1_1$-dependent choice is the first \mbox{non-recursive} ordinal $\omega_1^{\operatorname{CK}}$ (see \cite{kreisel62}), which is much larger than~$\Gamma_0$. As a consequence, the principle of $\Sigma^1_1$-dependent choice does not possess a prima facie predicative justification. By the result of our paper the same applies to the principle that the derivative of every dilator preserves well-foundedness. On the other hand, all sufficiently concrete consequences of these principles hold in predicative mathematics: The extension of $\aca_0$ by $\Sigma^1_1$-dependent choice is $\Pi^1_2$-conservative over the theory 
$\mathbf{(\Pi^1_0\textbf{-}CA)_{\omega^{\omega}}}$, which allows for $\omega^\omega$ iterations of arithmetical comprehension (due to A.~Cantini 
\cite{cantini86}). The latter is a predicative theory in its entirety.

Let us also compare our result to a theorem of T.~Arai~\cite{arai-derivatives}. Roughly speaking, this theorem states that the following are equivalent over~$\aca_0$:
\begin{itemize}
 \item The order $D^{\partial T}_X$ is well-founded for every well-order $X$.
 \item Any set is contained in a countable coded $\omega$-model of the statement that ``$D^T_X$~is well-founded for every well-order~$X$''.
\end{itemize}
This formulation of Arai's result should be read with quite some reservation: Arai does not represent normal functions by dilators. Instead his result relies on the assumption that we are given formulas that define term systems for~$D^T_X$ and $D^{\partial T}_X$, which must satisfy certain conditions. In particular this approach does not allow to quantify over dilators, as required for our result. On an informal level Arai's result can be read as a pointwise version of ours: Recall that $\Pi^1_1$-bar induction is equivalent to $\omega$-model reflection for $\Pi^1_2$-formulas. Assume that we want to establish this reflection principle for a formula~$\varphi$. Girard has shown that the notion of dilator is $\Pi^1_2$-complete (see D.~Norman's proof in~\cite[Annex~8.E]{girard-book-part2}, which will also play an important role in Section~\ref{sect:constructin-derivative} below). Thus one may hope to construct a normal prae-dilator $T$ such that $\varphi$ is equivalent to the statement that ``$D^T_X$ is well-founded for every well-order~$X$''. Using our principle~(1) one could conclude that ``$D^{\partial T}_X$ is well-founded for every well-order~$X$''. By Arai's result this would yield the desired $\omega$-models of $\varphi$. When we started working on the present paper we planned to derive the equivalence between (1), (2) and (3) from Arai's result, by making the given argument precise. However, this has met with so many technical obstacles that it turned out easier to give a completely new proof.

To conclude this introduction we compare our result to a theorem of the first author~\cite{freund-thesis,freund-equivalence,freund-categorical,freund-computable}, which says that the following are equivalent over $\rca_0$:
\begin{itemize}
 \item Every dilator has a well-founded Bachmann-Howard fixed point.
 \item The principle of $\Pi^1_1$-comprehension holds.
\end{itemize}
To explain what this means we point out that the first principle quantifies over arbitrary dilators~$T$, rather than just over normal ones. This includes cases where we have $\otp(D^T_\alpha)>\alpha$ for any ordinal~$\alpha$, so that $D^T$ cannot have a well-founded fixed point. The best we can hope for is a function $\vartheta:D^T_\alpha\rightarrow\alpha$ that is ``almost" order-preserving (see~\cite{freund-equivalence} for a precise definition). If such a function exists, then $\alpha$ is called a Bachmann-Howard fixed point of $T$. This name has been chosen since the conditions on~$\vartheta$ are inspired by properties of the collapsing function used to define the Bachmann-Howard ordinal (cf.~in particular~\cite{rathjen-weiermann-kruskal}). It is worth noting that the notion of Bachmann-Howard fixed point is most interesting for dilators that are not normal (see the proof of~\cite[Proposition~3.3]{freund-computable} for an instructive example). As is well-known, $\Pi^1_1$-comprehension is much stronger than $\Pi^1_1$-bar induction. Thus the results of~\cite{freund-thesis} and the present paper help to explain why collapsing functions, rather than derivatives, are the crucial feature of strong ordinal notation systems.

\textbf{Acknowledgements.} We would like to thank the anonymous referee for their detailed comments, which have been very helpful in making our paper more readable.

\section{Normal dilators in second order arithmetic}\label{sect:normal-dils-so}

In the present section we define and investigate (prae-) dilators in the setting of reverse mathematics. Our approach is based on the work of Girard~\cite{girard-pi2} and on details worked out by the first author~\cite{freund-thesis,freund-computable}. We will also characterize normal prae-dilators and their (upper) derivatives.

Let us fix some category-theoretic terminology: To turn the class of (countable) linear orders into a category we take the order embeddings (strictly increasing functions) as morphisms. The forgetful functor to the underlying set of an order will be left implicit. Conversely, a subset of an order will often be considered as a suborder. The finite subset functor $[\cdot]^{<\omega}$ on the category of sets is given by
\begin{align*}
 [X]^{<\omega}&=\text{``the set of finite subsets of $X$''},\\
 [f]^{<\omega}(a)&=\{f(x)\,|\,x\in a\}.
\end{align*}
We will often write the arguments of a functor $T$ as subscripts, so that a morphism $f:X\rightarrow Y$ is transformed into $T_f:T_X\rightarrow T_Y$. When we want to avoid iterated subscripts we revert to the notation $T(f):T(X)\rightarrow T(Y)$.

Since the formalization of dilators in second order arithmetic is somewhat technical, we first give an informal summary. Guided by this summary, the reader can skim through the rest of the section without considering all technical details.

\begin{summary}\label{sum:class-sized-dilators}
In a sufficiently expressive meta-theory, a class-sized prae-dilator can be defined as a pair $(T,\supp^T)$ of
\begin{itemize}
\item a functor $T$ from linear orders to linear orders and
\item a natural transformation $\supp^T:T\Rightarrow[\cdot]^{<\omega}$ that satisfies the following support condition: for any order~$X$, each element $\sigma\in T_X$ lies in the range of the morphism $T_{\iota_\sigma}$, where $\iota_\sigma:\supp^T_X(\sigma)\hookrightarrow X$ is the inclusion.
\end{itemize}
If $T_X$ is well-founded for every well-order~$X$, then $T$ is called a class-sized dilator. The specification ``class-sized" will often be omitted. It is instructive to consider Example~\ref{ex:omega-dilator} below. In~\cite[Remark~2.2.2]{freund-thesis} it has been verified that the given definition of dilator coincides with the original one by Girard~\cite{girard-pi2}. However, our prae-dilators are not quite equivalent to Girard's pre-dilators, as~the latter must satisfy an additional monotonicity condition that is automatic in the well-founded case. Any dilator~$T$ induces a function $\alpha\mapsto\otp(T_\alpha)$ on the \mbox{ordinals}. To~ensure that this function is normal one demands that $T$ preserves initial segments. More precisely, we write
\begin{equation*}
X\!\restriction\!x=\{x'\in X\,|\,x'<_X x\}
\end{equation*}
for an order~$(X,<_X)$ and $x\in X$. A normal (prae-) dilator consists of
\begin{itemize}
\item a (prae-) dilator $(T,\supp^T)$ and
\item a natural family of functions $\mu^T_X:X\to T_X$ with the following property: for any inclusion $\iota_x:X\!\restriction\!x\hookrightarrow X$, the morphism $T_{\iota_x}$ has range $T_X\!\restriction\!\mu^T_X(x)$.
\end{itemize}
If $T$ is a normal dilator, then $\alpha\mapsto\otp(T_\alpha)$ is a normal function on the ordinals, as shown in Proposition~\ref{prop:normal-dil-fct} below. Recall that a normal function~$g$ is an upper derivative of $f$ if we have $f\circ g(\alpha)=g(\alpha)$ for every ordinal~$\alpha$. To ensure this equality it suffices to have an embedding of $f\circ g(\alpha)$ into $g(\alpha)$. In the categorical setting it is natural to demand compatible embeddings: A morphism between normal prae-dilators $(T,\mu^T)$ and $(S,\mu^S)$ is given by a natural transformation $\nu:T\Rightarrow S$ such that we have $\nu_X\circ\mu^T_X=\mu^S_X$ for any order~$X$. Now an upper derivative of a normal prae-dilator~$T$ can be defined as a pair $(S,\xi)$ of
\begin{itemize}
\item a normal prae-dilator~$S$ and
\item a morphism $\xi:T\circ S\Rightarrow S$ of normal prae-dilators.
\end{itemize}
The derivative of a normal function is the upper derivative with the smallest possible values. On the level of dilators this is naturaly expressed via the notion of initial object: Consider two upper derivatives $(S^1,\xi^1)$ and $(S^2,\xi^2)$ of a normal prae-dilator~$T$. A morphism $\nu:S^1\Rightarrow S^2$ of normal prae-dilators is called a morphism of upper derivatives if we have $\nu_x\circ\xi^1_X=\xi^2_X\circ T_{\nu_X}$ for every order~$X$. Now an upper derivative $(S,\xi)$ of~$T$ is called a derivative if any other upper derivative $(S',\xi')$ admits a unique morphism $\nu:S\Rightarrow S'$ of upper derivatives. Just as all initial objects, derivatives of normal prae-dilators are unique up to natural isomorphism. If $(S,\xi)$ is the derivative of a normal dilator~$T$, then $\alpha\mapsto\otp(S_\alpha)$ is the usual derivative of the normal function~$\alpha\mapsto\otp(T_\alpha)$, as we will show in Corollary~\ref{cor:deriv-dil-to-fct}. Together with the good categorical properties, this shows that we have found the ``right" definition of derivative on the level of dilators.
\end{summary}

In the rest of this section we make the previous summary precise and formalize it in reverse mathematics. The formalization relies on Girard's observation that dilators are essentially determined by their restrictions to finite orders. Let us fix some terminology: The category of natural numbers consists of the finite orders $n=\{0,\dots,n-1\}$ (ordered as usual) and all embeddings between them. Note that this yields a small category that is equivalent to the category of all finite orders. The equivalence is witnessed by the increasing enumerations $\en_a:|a|\rightarrow a$, where $|a|=\{0,\dots,|a|-1\}$ denotes the cardinality of the finite order~$a$. For each embedding $f:a\rightarrow b$ there is a unique increasing function $|f|:|a|\rightarrow |b|$ with
\begin{equation*}
 \en_b\circ|f|=f\circ\en_a.
\end{equation*}
Thus~$|\cdot|$ and $\en$ become a functor and a natural isomorphism. We continue to use the finite subset functor $[\cdot]^{<\omega}$, even though $[n]^{<\omega}$ coincides with the full power set of~$n=\{0,\dots,n-1\}$. Hereditarily finite sets with the natural numbers as urelements can be coded by natural numbers. It is straightforward to see that basic relations and operations on these sets are primitive recursive in the codes. This allows us to introduce the following notion in the theory~$\rca_0$, as in~\cite{freund-computable}:

\begin{definition}[$\rca_0$]\label{def:coded-prae-dilator}
 A coded prae-dilator consists of
 \begin{enumerate}[label=(\roman*)]
  \item a functor $T$ from the category of natural numbers to the category of linear orders with fields $T_n\subseteq\mathbb N$ and
  \item a natural transformation $\supp^T:T\Rightarrow[\cdot]^{<\omega}$ such that any $\sigma\in T_n$ lies in the range of $T_{\iota_\sigma\circ\en_\sigma}$, where
  \begin{equation*}
   |\supp^T_n(\sigma)|\xrightarrow{\mathmakebox[2em]{\en_\sigma}}\supp^T_n(\sigma)\xhookrightarrow{\mathmakebox[2em]{\iota_\sigma}}n=\{0,\dots,n-1\}
  \end{equation*}
  factors the unique morphism with range $\supp^T_n(\sigma)\subseteq n$.
 \end{enumerate}
\end{definition}

More precisely, the functor~$T$ is represented by the sets
  \begin{align*}
  T^0&=\{\langle 0,n,\sigma\rangle\,|\,\sigma\in T_n\}\cup\{\langle 1,n,\sigma,\tau\rangle\,|\,\sigma<_{T_n}\tau\},\\
  T^1&=\{\langle f,\sigma,\tau\rangle\,|\,T_f(\sigma)=\tau\}
  \end{align*}
  of natural numbers. The natural transformation $\supp^T$ is represented by the set
  \begin{equation*}
  \supp^T=\{\langle n,\sigma,a\rangle\,|\,\supp^T_n(\sigma)=a\}.
  \end{equation*}
  Thus an expression such as $\sigma\in T_n$ is an abbreviation for $\langle 0,n,\sigma\rangle\in T^0$, which is a $\Delta^0_1$-formula in $\rca_0$. The statement that $T$ is a coded prae-dilator is easily seen to be arithmetical in the sets $T^0,T^1,\supp^T\subseteq\mathbb N$.

\begin{example}\label{ex:omega-dilator}
  For any order~$X$ we consider the set
  \begin{equation*}
   \omega^X=\{\langle x_{k-1},\dots,x_0\rangle\,|\,x_0\leq_X\dots\leq_X x_{k-1}\}
  \end{equation*}
  with the lexicographic order (it may help to think of $\langle x_{k-1},\dots,x_0\rangle$ as the formal Cantor normal form $\omega^{x_{k-1}}+\dots+\omega^{x_0}$). To obtain a functor we map each morphism $f:X\rightarrow Y$ to the embedding $\omega^f:\omega^X\rightarrow\omega^Y$ with
  \begin{equation*}
  \omega^f(\langle x_{k-1},\dots,x_0\rangle)=\langle f(x_{k-1}),\dots,f(x_0)\rangle.
  \end{equation*}
  If we define $\supp^\omega_X:\omega^n\rightarrow[X]^{<\omega}$ by 
  \begin{equation*}
  \supp^\omega_X(\langle x_{k-1},\dots,x_0\rangle)=\{x_{k-1},\dots,x_0\},
  \end{equation*}
  then we get a class-sized prae-dilator in the sense of Summary~\ref{sum:class-sized-dilators}. Its restriction to the category of finite orders is a coded prae-dilator in the sense of Definition~\ref{def:coded-prae-dilator}.
\end{example}
  
Let us discuss how the class-sized prae-dilator $X\mapsto\omega^X$ from the previous example can be reconstructed from its coded restriction. The idea is to view an element $\langle n_{k-1},\dots,n_0\rangle\in\omega^n$ as a term. In order to obtain an element of $\omega^X$, the ``variables" $n_i$ in this term are substituted by elements~$x_i\in X$, in increasing order. For example, the pair $\langle\{x_0,x_1\},\langle 1,1,0\rangle\,\rangle$ with $x_0<x_1$ represents the element~$\langle x_1,x_1,x_0\rangle\in\omega^X$. To make the representations unique we require that the variables are as small as possible. Thus $\langle\{x_0,x_1\},\langle 3,3,1\rangle\,\rangle$ would not be a valid representation. In order to formulate this requirement in general we will rely on the observation that we have $\langle 1,1,0\rangle\in\omega^2=\omega^{|\{x_0,x_1\}|}$. One should also demand that all given elements of~$X$ are substituted for a variable. Thus $\langle\{x_0,x_1,x_2\},\langle 1,1,0\rangle\,\rangle$ with $x_0<_X x_1<_X x_2$ would not be a valid representation. This can be expressed via the condition~$\supp^\omega_{|\{x_0,x_1\}|}(\langle 1,1,0\rangle)=\{0,1\}=2=|\{x_0,x_1\}|$. In general, the class-sized extension $D^T$ of a coded prae-dilator~$T$ can be defined as follows (cf.~\cite{freund-computable}):
  
\begin{definition}[$\rca_0$]\label{def:coded-prae-dilator-reconstruct}
 Consider a coded prae-dilator $T=(T,\supp^T)$. For each order~$X$ we define a set $D^T_X$ and a binary relation $<_{D^T_X}$ on $D^T_X$ by
 \begin{gather*}
 D^T_X=\{\langle a,\sigma\rangle\,|\,a\in[X]^{<\omega}\text{ and }\sigma\in T_{|a|}\text{ and }\supp^T_{|a|}(\sigma)=|a|\},\\
 \langle a,\sigma\rangle<_{D^T_X}\langle b,\tau\rangle\Leftrightarrow T_{|\iota_a^{a\cup b}|}(\sigma)<_{T_{|a\cup b|}}T_{|\iota_b^{a\cup b}|}(\tau),
\end{gather*}
where $\iota_a^{a\cup b}:a\hookrightarrow a\cup b$ and $\iota_b^{a\cup b}:b\hookrightarrow a\cup b$ denote the inclusions between suborders of~$X$. Given an embedding $f:X\rightarrow Y$, we define $D^T_f:D^T_X\rightarrow D^T_Y$ by
 \begin{equation*}
 D^T_f(\langle a,\sigma\rangle)=\langle [f]^{<\omega}(a),\sigma\rangle.
\end{equation*}
To define a family of functions $\supp^{D^T}_X:D^T_X\Rightarrow[X]^{<\omega}$ we set
\begin{equation*}
 \supp^{D^T}_X(\langle a,\sigma\rangle)=a
\end{equation*}
for each order~$X$.
\end{definition}

In order to see that $D^T_f(\langle a,\sigma\rangle)$ still satisfies the uniqueness conditions (i.\,e.~that we have $\sigma\in T_{|[f]^{<\omega}(a)|}$ and $\supp^T_{|[f]^{<\omega}(a)|}(\sigma)=|[f]^{<\omega}(a)|$) it suffices to note that $[f]^{<\omega}(a)$ has the same cardinality as~$a$. The following shows that $D^T$ is a class-sized prae-dilator in the sense of Summary~\ref{sum:class-sized-dilators} (in part~(ii) of the proposition one could replace $\iota_{\langle a,\sigma\rangle}$ by $\iota_{\langle a,\sigma\rangle}\circ\en_a$, since $\en_a:|a|\rightarrow a$ is an isomorphism).

\begin{proposition}[$\rca_0$]\label{prop:reconstruct-class-sized-dil}
 If $T$ is a coded prae-dilator, then
 \begin{enumerate}[label=(\roman*)]
  \item the maps $X\mapsto(D^T_X,<_{D^T_X})$ and $f\mapsto D^T_f$ form an endofunctor on the category of linear orders and
  \item the map $X\mapsto\supp^{D^T}_X$ is a natural transformation between $D^T$ and $[\cdot]^{<\omega}$, with the property that any $\langle a,\sigma\rangle\in D^T_X$ lies in the range of $D^T_{\iota_{\langle a,\sigma\rangle}}$, where
  \begin{equation*}
   \iota_{\langle a,\sigma\rangle}:\supp^{D^T}_X(\langle a,\sigma\rangle)=a\hookrightarrow X
  \end{equation*}
  is the inclusion.
 \end{enumerate}
\end{proposition}
\begin{proof}
 In~\cite[Lemma~2.4]{freund-computable} the same has been shown in a stronger base theory (we point out that the uniqueness conditions are crucial for the linearity of $<_{D^T_X}$). It is straightforward to check that the proof goes through in~$\rca_0$.
\end{proof}

While $D^T$ is a class-sized object, its restriction $D^T\!\restriction\!\mathbb N$ to the category of natural numbers can be constructed in~$\rca_0$. The following is similar to~\cite[Proposition~2.5]{freund-computable}. Nevertheless we give a detailed proof, since we want to refer to it later.

\begin{lemma}[$\rca_0$]\label{lem:class-sized-restrict}
 If $T$ is a coded prae-dilator, then so is $D^T\!\restriction\!\mathbb N$. In this case we get a natural isomorphism $\eta^T:D^T\!\restriction\!\mathbb N\Rightarrow T$ by setting
 \begin{equation*}
  \eta^T_n(\langle a,\sigma\rangle)=T_{\iota_a\circ\en_a}(\sigma),
 \end{equation*}
 where $\iota_a:a\hookrightarrow n$ is the inclusion.
\end{lemma}
\begin{proof}
 The previous proposition implies that $D^T\!\restriction\!\mathbb N$ is a coded prae-dilator. To see that $\eta^T_n$ is order preserving we consider an inequality $\langle a,\sigma\rangle<_{D^T_n}\langle b,\tau\rangle$. According to Definition~\ref{def:coded-prae-dilator-reconstruct} this amounts to $T_{|\iota_a^{a\cup b}|}(\sigma)<_{T_{|a\cup b|}}T_{|\iota_b^{a\cup b}|}(\tau)$. Write $\iota_{a\cup b}:a\cup b\hookrightarrow n$ and observe
 \begin{equation*}
  \iota_{a\cup b}\circ\en_{a\cup b}\circ|\iota_a^{a\cup b}|=\iota_{a\cup b}\circ\iota_a^{a\cup b}\circ\en_a=\iota_a\circ\en_a.
 \end{equation*}
 Applying $T_{\iota_{a\cup b}\circ\en_{a\cup b}}$ to both sides of the above inequality we obtain
 \begin{multline*}
  \eta^T_n(\langle a,\sigma\rangle)=T_{\iota_a\circ\en_a}(\sigma)=T_{\iota_{a\cup b}\circ\en_{a\cup b}}\circ T_{|\iota_a^{a\cup b}|}(\sigma)<_{T_n}{}\\
  {}<_{T_n}T_{\iota_{a\cup b}\circ\en_{a\cup b}}\circ T_{|\iota_b^{a\cup b}|}(\tau)=T_{\iota_b\circ\en_b}(\tau)=\eta^T_n(\langle b,\tau\rangle).
 \end{multline*}
 To establish naturality we consider an order preserving function~$f:n\rightarrow m$. Write $\iota_{[f]^{<\omega}(a)}:[f]^{<\omega}(a)\hookrightarrow m$ and observe that we have
 \begin{equation*}
  f\circ\iota_a\circ\en_a=\iota_{[f]^{<\omega}(a)}\circ\en_{[f]^{<\omega}(a)},
 \end{equation*}
 as both sides are order isomorphisms between $|a|=|[f]^{<\omega}(a)|$ and $[f]^{<\omega}(a)\subseteq n$. We can deduce
 \begin{multline*}
  \eta^T_m\circ D^T_f(\langle a,\sigma\rangle)=\eta^T_m(\langle[f]^{<\omega}(a),\sigma\rangle)=T_{\iota_{[f]^{<\omega}(a)}\circ\en_{[f]^{<\omega}(a)}}(\sigma)=\\
  =T_f\circ T_{\iota_a\circ\en_a}(\sigma)=T_f\circ\eta^T_n(\langle a,\sigma\rangle).
 \end{multline*}
 By the definition of coded prae-dilator any $\sigma\in T_n$ can be written as $\sigma=T_{\iota_a\circ\en_a}(\sigma_0)$ with $a=\supp^T_n(\sigma)$ and $\sigma_0\in T_{|a|}$. In view of
 \begin{equation*}
  [\iota_a\circ\en_a]^{<\omega}(\supp^T_{|a|}(\sigma_0))=\supp^T_n(T_{\iota_a\circ\en_a}(\sigma_0))=\supp^T_n(\sigma)=a
 \end{equation*}
 we have $\supp^T_{|a|}(\sigma_0)=|a|$ and hence $\langle a,\sigma_0\rangle\in D^T_n$. Since $\eta^T_n(\langle a,\sigma_0\rangle)=\sigma$ holds by construction we can conclude that $\eta^T_n$ is surjective.
\end{proof}

As indicated in the introduction, the following notion plays a crucial role (there is no ambiguity since the two obvious definitions of well-foundedness are equivalent in $\rca_0$, see e.\,g.~\cite[Lemma~2.3.12]{freund-thesis}):

\begin{definition}[$\rca_0$]\label{def:coded-dilator}
 A coded prae-dilator~$T$ is called a coded dilator if $D^T_X$ is well-founded for every well-order~$X$.
\end{definition}

In a sufficiently expressive meta-theory, we can now discuss the reconstruction of a class-sized prae-dilator~$T$. Assuming that $T$ preserves countability, we may assume $T_n\subseteq\mathbb N$ for every number~$n$. Then the restriction $T\!\restriction\!\mathbb N$ is a coded prae-dilator. The equivalence from Lemma~\ref{lem:class-sized-restrict} is readily extended into a natural isomorphism between $D^{T\restriction\mathbb N}$ and $T$ (see~\cite[Proposition~2.5]{freund-computable}). In view of $D^{T\restriction\mathbb N}_X\cong T_X$ it is immediate that $T\!\restriction\!\mathbb N$ is a coded dilator if~$T$ is a class-sized dilator. The converse is somewhat more subtle, since Definition~\ref{def:coded-dilator} only quantifies over well-orders with field $X\subseteq\mathbb N$. Girard~\cite[Theorem~2.1.15]{girard-pi2} has shown that it suffices to test the preservation of well-foundedness on countable orders. Thus it is true that $D^T$ is a class-sized dilator for any coded dilator~$T$. In second order arithmetic we can consider the orders $T_X$ and the isomorphisms $D^{T\restriction\mathbb N}_X\cong T_X$ when $T$ is a specific class-sized prae-dilator with a computable construction. This can be useful when $T_X$ has a more transparent description than $D^{T\restriction\mathbb N}_X$ (as in the example above, where the term $\langle x_1,x_1,x_0\rangle\in\omega^X$ is more intelligible than the expression $\langle\{x_0,x_1\},\langle 1,1,0\rangle\,\rangle\in D^ \omega_X$). On the other hand, second order arithmetic cannot reason about class-sized prae-dilators in general (i.\,e.~quantify over them). Thus we will mostly be concerned with coded prae-dilators, which are more important on a theoretical level. We will often omit the specification ``coded'' to improve readability.

Arguing in a sufficiently strong set theory, each coded dilator $T$ induces a function $\alpha\mapsto\otp(D^T_\alpha)$ on the ordinals. To see that this function does not need to be normal we consider the coded dilator that maps $n$ to the order
\begin{equation*}
 T_n=\{0,\dots,n-1\}\cup\{\Omega\}
\end{equation*}
with a new biggest element~$\Omega$. Its action on a morphism $f:n\rightarrow m$ and the support functions $\supp^T_n:T_n\rightarrow[n]^{<\omega}$ are given by
\begin{equation*}
 T_f(\sigma)=\begin{cases}
              f(\sigma) & \text{if $\sigma\in\{0,\dots,n-1\}$},\\
              \Omega & \text{if $\sigma=\Omega$},
             \end{cases}\quad
 \supp^T_n(\sigma)=\begin{cases}
              \{\sigma\} & \text{if $\sigma\in\{0,\dots,n-1\}$},\\
              \emptyset & \text{if $\sigma=\Omega$}.
             \end{cases}
\end{equation*}
It is straightforward to check that
\begin{equation*}
 D^T_X=\{\langle\{x\},0\rangle\,|\,x\in X\}\cup\{\langle\emptyset,\Omega\rangle\}\qquad\text{(with $0\in 1\subseteq T_1$ and $\Omega\in T_0$)}
\end{equation*}
is isomorphic to $X\cup\{\Omega\}$ (where $\Omega$ is still the biggest element). Thus we have $\otp(D^T_\alpha)=\alpha+1$, which means that the function induced by $T$ is not continuous at limit stages and does not have any fixed points.

To analyze the given counterexample we observe that the functor~$T$ from the previous paragraph does not preserve initial segments: Given that the range of $f:n\rightarrow m$ is an initial segment of $m$, we cannot infer that the range of $T_f$ is an initial segment of $T_n$ (since it contains the element~$\Omega$). Indeed, Aczel~\cite{aczel-phd,aczel-normal-functors} and Girard~\cite{girard-pi2} have identified preservation of initial segments as the crucial condition that reconciles categorical continuity, i.\,e.~preservation of direct limits, and the usual notion of continuity at limit ordinals (paraphrasing Girard). More precisely, Aczel focuses on initial segments of the form
\begin{equation*}
 X\!\restriction\!x=\{y\in X\,|\,y<_Xx\},
\end{equation*}
where $x$ is an element of the linear order~$X=(X,<_X)$. It will be convenient to have the following notation: For $a,b\in[X]^{<\omega}$ we abbreviate
\begin{equation*}
 a\lef_X b\quad\Leftrightarrow\quad\forall_{x\in a}\exists_{y\in b}\,x<_X y.
\end{equation*}
The relation $\leqf_X$ is defined in the same way, with $\leq_X$ at the place of $<_X$. We omit the subscript when we refer to the usual order on the natural numbers or on the ordinals. In the case of a singleton we write $a\lef_X y$ rather than $a\lef_X\{y\}$. Note that this makes $a\lef_Xx$ equivalent to $a\subseteq X\!\restriction\!x$. The following is fundamental:

\begin{lemma}[$\rca_0$]\label{lem:range-dil-support}
 If $T$ is a coded prae-dilator, then we have
 \begin{equation*}
  \rng(D^T_f)=\{\langle a,\sigma\rangle\in D^T_Y\,|\,a\subseteq\rng(f)\}
 \end{equation*}
 for any order embedding~$f:X\rightarrow Y$.
\end{lemma}
\begin{proof}
 For the inclusion $\subseteq$ it suffices to recall $D^T_f(\langle b,\sigma\rangle)=\langle [f]^{<\omega}(b),\sigma\rangle$. Conversely, induction on the size of $a\subseteq\rng(f)$ yields a finite $b\subseteq X$ with $[f]^{<\omega}(b)=a$. Then $\langle a,\sigma\rangle\in D^T_Y$ is the image of $\langle b,\sigma\rangle\in D^T_X$ (observe $|b|=|a|$).
\end{proof}

Preservation of initial segments can now be characterized as follows:

\begin{corollary}[$\rca_0$]\label{cor:initial-segments-supports}
 Consider a coded prae-dilator $T$ and a linear order $X$. The following are equivalent for any elements $x\in X$ and $\rho\in D^T_X$:
\begin{enumerate}[label=(\roman*)]
 \item We have $\rng(D^T_{\iota_x})=D^T_X\!\restriction\!\rho$, where $\iota_x:X\!\restriction\!x\hookrightarrow X$ is the inclusion.
 \item For any $\langle a,\sigma\rangle\in D^T_X$ we have
 \begin{equation*}
  \langle a,\sigma\rangle<_{D^T_X}\rho\quad\Leftrightarrow\quad a\lef_X x.
 \end{equation*}
\end{enumerate}
\end{corollary}

We will see that a coded dilator with the following property does induce a normal function on the ordinals. 

\begin{definition}[$\rca_0$]\label{def:coded-normal-dil}
 A normal prae-dilator consists of a (coded) prae-dilator~$T$ and a natural family of order embeddings $\mu^T_n:n\rightarrow T_n$ such that we have
 \begin{equation*}
  \sigma<_{T_n}\mu^T_n(m)\quad\Leftrightarrow\quad\supp^T_n(\sigma)\lef m
 \end{equation*}
 for all numbers $m<n$ and all elements $\sigma\in T_n$.
\end{definition}

Note that the family of functions $\mu^T_n$ can be represented by the set
\begin{equation*}
 \mu^T=\{\langle n,m,\rho\rangle\,|\,\mu^T_n(m)=\rho\}
\end{equation*}
of natural numbers. As an example we recall the coded dilator $n\mapsto\omega^n$ considered above. It is straightforward to verify that we obtain a normal dilator by setting
\begin{equation*}
 \mu^\omega_n(m)=\langle m\rangle\in\omega^n
\end{equation*}
for all numbers $m<n$. Recall that $\langle m\rangle$ corresponds to the formal Cantor normal form $\omega^m$. This suggests to think of $\mu^\omega_n$ as the restriction of the normal function $\alpha\mapsto\omega^\alpha$ to the finite ordinal~$n$. A formal version of this idea can be found in the proof of Proposition~\ref{prop:normal-dil-fct} below. Before we can formulate it we must extend $\mu^T$ beyond the category of natural numbers. This relies on the following observation:

\begin{lemma}[$\rca_0$]\label{lem:support-mu}
 If $T=(T,\mu^T)$ is a normal prae-dilator, then we have
 \begin{equation*}
  \supp^T_n(\mu^T_n(m))=\{m\}
 \end{equation*}
 for all numbers $m<n$.
\end{lemma}
\begin{proof}
 Define $\iota:1\rightarrow n$ by $\iota(0)=m$. By the naturality of $\mu^T$ and $\supp^T$ we get
 \begin{equation*}
  \supp^T_n(\mu^T_n(m))=\supp^T_n(\mu^T_n(\iota(0)))=[\iota]^{<\omega}(\supp^T_1(\mu^T_1(0)))\subseteq\rng(\iota)=\{m\}.
 \end{equation*}
 So it remains to show that we cannot have $\supp^T_n(\mu^T_n(m))=\emptyset$. The latter would imply $\supp^T_n(\mu^T_n(m))\lef m$ and hence $\mu^T_n(m)<_{T_n}\mu^T_n(m)$, which is impossible.
\end{proof}

In particular the lemma yields $\supp^T_{|\{x\}|}(\mu^T_1(0))=|\{x\}|$, which secures the uniqueness condition needed for the following construction:

\begin{definition}[$\rca_0$]\label{def:extend-normal-transfos}
 Let $T$ be a normal prae-dilator. For each order~$X$ we define $D^{\mu^T}_X:X\rightarrow D^T_X$ by setting
 \begin{equation*}
  D^{\mu^T}_X(x)=\langle\{x\},\mu^T_1(0)\rangle
 \end{equation*}
 for all elements $x\in X$.
\end{definition}

The reader may have noticed that only the value $\mu^T_1(0)$ was needed in order to extend $\mu^T$ to arbitrary linear orders. To state the equivalence from Definition~\ref{def:coded-normal-dil} for all numbers~$n$ it is nevertheless convenient to consider the entire family of functions $\mu^T_n:n\rightarrow T_n$ as given. The following shows that we have reconstructed the normal prae-dilators from Summary~\ref{sum:class-sized-dilators}.

\begin{proposition}[$\rca_0$]\label{prop:reconstruct-normal-dil}
 If $T$ is a normal prae-dilator, then the functions $D^{\mu^T}_X:X\rightarrow D^T_X$ form a natural family of order embeddings. Furthermore we have
 \begin{equation*}
  \langle a,\sigma\rangle<_{D^T_X}D^{\mu^T}_X(x)\quad\Leftrightarrow\quad a\lef_X x
 \end{equation*}
 for any order~$X$ and any element $\langle a,\sigma\rangle\in D^T_X$.
\end{proposition}
\begin{proof}
 To show that~$D^{\mu^T}_X$ is an embedding we consider $x_0<_Xx_1$. For $j\in\{0,1\}$ we write $\iota_j:\{x_j\}\hookrightarrow\{x_0,x_1\}$. Using the naturality of $\mu^T$ and the fact that $\mu^T_2$ is order preserving we get
 \begin{equation*}
  T_{|\iota_0|}(\mu^T_1(0))=\mu^T_2(|\iota_0|(0))=\mu^T_2(0)<_{T_2}\mu^T_2(1)=\mu^T_2(|\iota_1|(0))=T_{|\iota_1|}(\mu^T_1(0)).
 \end{equation*}
 According to Definition~\ref{def:coded-prae-dilator-reconstruct} this yields
 \begin{equation*}
  D^{\mu^T}_X(x_0)=\langle\{x_0\},\mu^T_1(0)\rangle<_{D^T_X}\langle\{x_1\},\mu^T_1(0)\rangle=D^{\mu^T}_X(x_1),
 \end{equation*}
 as desired. To see that $D^{\mu^T}$ is natural we compute
 \begin{equation*}
  D^T_f(D^{\mu^T}_X(x))=\langle[f]^{<\omega}(\{x\}),\mu^T_1(0)\rangle=\langle\{f(x)\},\mu^T_1(0)\rangle=D^{\mu^T}_Y(f(x)).
 \end{equation*}
 It remains to establish the stated equivalence: First assume that we have
 \begin{equation*}
  \langle a,\sigma\rangle<_{D^T_X}D^{\mu^T}_X(x)=\langle\{x\},\mu^T_1(0)\rangle.
 \end{equation*}
 Write $\iota_0:a\hookrightarrow a\cup\{x\}$ and $\iota_1:\{x\}\hookrightarrow a\cup\{x\}$ for the inclusions. By definition of the order on $D^T_X$ we have
 \begin{equation*}
  T_{|\iota_0|}(\sigma)<_{T_{|a\cup\{x\}|}}T_{|\iota_1|}(\mu^T_1(0))=\mu^T_{|a\cup\{x\}|}(|\iota_1|(0)).
 \end{equation*}
 Using the equivalence from Definition~\ref{def:coded-normal-dil} we can deduce
 \begin{equation*}
  [|\iota_0|]^{<\omega}(|a|)=[|\iota_0|]^{<\omega}(\supp^T_{|a|}(\sigma))=\supp^T_{|a\cup\{x\}|}(T_{|\iota_0|}(\sigma))\lef |\iota_1|(0).
 \end{equation*}
 This implies $a\lef_X x$, as desired. To establish the converse implication one follows the argument backwards, noting that $a\lef_X x$ implies $[|\iota_0|]^{<\omega}(|a|)\lef |\iota_1|(0)$.
\end{proof}

Working in a sufficiently strong set theory, we can now prove that normal dilators do induce normal functions. This result is due to Aczel~\cite[Theorem~2.11]{aczel-phd}.

\begin{proposition}\label{prop:normal-dil-fct}
 Assume that $T$ is a normal dilator. Then $\alpha\mapsto\otp(D^T_\alpha)$ is a normal function on the ordinals.
\end{proposition}
\begin{proof}
 As a preparation we observe the following: Writing $\iota_x:X\!\restriction\!x\hookrightarrow X$ for the inclusion, we can combine Corollary~\ref{cor:initial-segments-supports} and Proposition~\ref{prop:reconstruct-normal-dil} to see that the range of $D^T_{\iota_x}$ is equal to $D^T_X\!\restriction\!D^{\mu^T}_X(x)$. Since $D^T_{\iota_x}$ is an order embedding this yields
 \begin{equation*}
  D^T_{X\restriction x}\cong D^T_X\!\restriction\!D^{\mu^T}_X(x)
 \end{equation*}
 for any order~$X$ and any~$x\in X$. Now we prove that $\alpha\mapsto\otp(D^T_\alpha)$ is strictly increasing: If we have $\alpha<\beta$, then $\alpha$ is isomorphic (and, with the usual set-theoretic definition of ordinals, even equal) to $\beta\!\restriction\!\alpha$. As $D^T$ is functorial (see Proposition~\ref{prop:reconstruct-class-sized-dil}) we get $D^T_\alpha\cong D^T_{\beta\restriction\alpha}$. Together with the above observation this yields
 \begin{equation*}
  \otp(D^T_\alpha)=\otp(D^T_{\beta\restriction\alpha})=\otp(D^T_\beta\!\restriction\!D^{\mu^T}_\beta(\alpha))<\otp(D^T_\beta).
 \end{equation*}
 To conclude that $\alpha\mapsto\otp(D^T_\alpha)$ is a normal function we must establish
 \begin{equation*}
  \otp(D^T_\lambda)\leq\sup\{\otp(D^T_\alpha)\,|\,\alpha<\lambda\}
 \end{equation*}
 when $\lambda$ is a limit ordinal. Given an element $\langle a,\sigma\rangle\in D^T_\lambda$, pick an ordinal $\alpha<\lambda$ with $a\lef\alpha$. Then Lemma~\ref{lem:range-dil-support} tells us that $\langle a,\sigma\rangle$ lies in the range of $D^T_{\iota_\alpha}$. By the above we obtain
 \begin{equation*}
  \otp(D^T_\lambda\!\restriction\!\langle a,\sigma\rangle)<\otp(D^T_\lambda\!\restriction\!D^{\mu^T}_\lambda(\alpha))=\otp(D^T_\alpha).
 \end{equation*}
 Since $\langle a,\sigma\rangle\in D^T_\lambda$ was arbitrary this implies the claim.
\end{proof}

The notion of upper derivative has already been described in Summary~\ref{sum:class-sized-dilators}. It order to make it precise, we need to consider compositions of and natural transformations between coded prae-dilators. Compositions are particularly technical in the coded case, since we cannot form $T(S_n)$ when $T$ is coded and $S_n$ is infinite. The reader may wish to skim through the following considerations (up to and including Lemma~\ref{lem:morph-dilators-extend}) without considering all technical details.

\begin{definition}[$\rca_0$]\label{def:compose-dils}
 Let $T$ and $S$ be coded prae-dilators. For each number~$n$ and each morphism $f:n\rightarrow m$ we put
 \begin{equation*}
  (T\circ S)_n=D^T(S_n),\qquad (T\circ S)_f=D^T(S_f),
 \end{equation*}
 where $D^T(S_n)$ is ordered according to Definition~\ref{def:coded-prae-dilator-reconstruct}. We also define a family of functions $\supp^{T\circ S}_n:(T\circ S)_n\rightarrow[n]^{<\omega}$ by setting
 \begin{equation*}
  \supp^{T\circ S}_n(\langle a,\tau\rangle)=\bigcup_{\sigma\in a}\supp^ S_n(\sigma)
 \end{equation*}
 for each number~$n$.
\end{definition}

It is straightforward to see that $\rca_0$ proves the existence of $T\circ S$. Crucially, the extension $D^{T\circ S}$ recovers the composition of~$D^T$ and $D^S$:

\begin{proposition}[$\rca_0$]\label{prop:compose-dils-rca}
 If $T$ and $S$ are coded (prae-) dilators, then so is $T\circ S$. We get a natural collection of isomorphisms $\zeta^{T,S}_X:D^T(D^S_X)\rightarrow D^{T\circ S}_X$ by setting
 \begin{equation*}
 \zeta^{T,S}_X(\langle\{\langle a_1,\sigma_1\rangle,\dots,\langle a_k,\sigma_k\rangle\},\tau\rangle)=\langle a_1\cup\dots\cup a_k,\langle\{S_{|\iota_1|}(\sigma_1),\dots,S_{|\iota_k|}(\sigma_k)\},\tau\rangle\,\rangle,
\end{equation*}
where $\iota_j:a_j\hookrightarrow a_1\cup\dots\cup a_k$ are the inclusion maps. Furthermore we have
 \begin{equation*}
 \supp^{D^{T\circ S}}_X(\zeta^{T,S}_X(\sigma))=\bigcup\{\supp^{D^S}_X(\rho)\,|\,\rho\in\supp^{D^T}_{D^S_X}(\sigma)\}
 \end{equation*}
 for any element $\sigma\in D^T(D^S_X)$.
\end{proposition}
\begin{proof}
One readily verifies that $T\circ S$ is a functor, using Proposition~\ref{prop:reconstruct-class-sized-dil}. The naturality of $\supp^{T\circ S}$ follows from the naturality of $\supp^S$. To see that the support condition from Definition~\ref{def:coded-prae-dilator} is satisfied we consider an arbitrary $\langle a,\tau\rangle\in(T\circ S)_n$. Abbreviate $c=\supp^{T\circ S}_n(\langle a,\tau\rangle)$ and observe that $\supp^S_n(\sigma)\subseteq c$ holds for any~$\sigma\in a$. Using the support condition for $S$ we get $\sigma\in\rng(S_{\iota_c\circ\en_c})$, where $\iota_c:c\hookrightarrow n$ is the inclusion. Induction on $|a|$ yields a finite set $b\subseteq S_{|c|}$ with $[S_{\iota_c\circ\en_c}]^{<\omega}(b)=a$. Since $S_{\iota_c\circ\en_c}$ is an embedding we have $|b|=|a|$ and hence $\langle b,\tau\rangle\in D^T(S_{|c|})$. In view of
\begin{equation*}
 (T\circ S)_{\iota_c\circ\en_c}(\langle b,\tau\rangle)=D^T(S_{\iota_c\circ\en_c})(\langle b,\tau\rangle)=\langle [S_{\iota_c\circ\en_c}]^{<\omega}(b),\tau\rangle=\langle a,\tau\rangle
\end{equation*}
we learn that $\langle a,\tau\rangle$ lies in the range of~$(T\circ S)_{\iota_c\circ\en_c}$, as required. If $T$ and $S$ are coded dilators, then~$D^T(D^S_X)$ is well-founded for any well-order~$X$. The claim that $T\circ S$ is a coded dilator will follow once we have proved $D^T(D^S_X)\cong D^{T\circ S}_X$. To show that the given equation for $\zeta^{T,S}_X$ defines such an isomorphism we first check that
\begin{equation*}
\sigma=\langle\{\langle a_1,\sigma_1\rangle,\dots,\langle a_k,\sigma_k\rangle\},\tau\rangle\in D^T(D^S_X)
\end{equation*} 
implies
\begin{equation*}
\zeta^{T,S}_X(\sigma)=\langle a_1\cup\dots\cup a_k,\langle\{S_{|\iota_1|}(\sigma_1),\dots,S_{|\iota_k|}(\sigma_k)\},\tau\rangle\,\rangle\in D^{T\circ S}_X.
\end{equation*}
Assuming that the pairs $\langle a_j,\sigma_j\rangle$ are all distinct, we see that $\sigma\in D^T(D^S_X)$ requires $\tau\in T_k$ and $\supp^T_k(\tau)=k$. Definition~\ref{def:coded-prae-dilator-reconstruct} also shows that $\langle a_i,\sigma_i\rangle<_{D^S_X}\langle a_j,\sigma_j\rangle$ implies $S_{|\iota_i|}(\sigma_i)<_{S_{|c|}}S_{|\iota_j|}(\sigma_j)$, where we abbreviate $c=a_1\cup\dots\cup a_k$. Thus the set $\{S_{|\iota_1|}(\sigma_1),\dots,S_{|\iota_k|}(\sigma_k)\}$ is still of cardinality~$k$, which yields
\begin{equation*}
\rho:=\langle\{S_{|\iota_1|}(\sigma_1),\dots,S_{|\iota_k|}(\sigma_k)\},\tau\rangle\in D^T(S_{|c|})=(T\circ S)_{|c|}.
\end{equation*}
To conclude $\zeta^{T,S}_X(\sigma)\in D^{T\circ S}_X$ it remains to establish $\supp^{T\circ S}_{|c|}(\rho)=|c|$. In view of $\sigma\in D^T(D^S_X)$ we must have $\langle a_j,\sigma_j\rangle\in D^S_X$ and hence $\supp^S_{|a_j|}(\sigma_j)=|a_j|$. Together with the naturality of $\supp^S$ we indeed get
\begin{equation*} \supp^{T\circ S}_{|c|}(\rho)=\bigcup_{j=1,\dots,k}\supp^S_{|c|}(S_{|\iota_j|}(\sigma_j))=\bigcup_{j=1,\dots,k}[|\iota_j|]^{<\omega}(\supp^S_{|a_j|}(\sigma_j))=|c|.
\end{equation*}
It is straightforward to check that $\zeta^{T,S}$ is natural, i.\,e.~that we have
\begin{equation*}
 \zeta^{T,S}_Y\circ D^T(D^S_f)=D^{T\circ S}_f\circ\zeta^{T,S}_X
\end{equation*}
for any embedding $f:X\rightarrow Y$. Using naturality, the claim that $\zeta^{T,S}_X$ is order preserving can be reduced to the case where $X=n$ is a natural number. There it follows from the observation that $\zeta^ {T,S}_n$ factors as
\begin{equation*}
 D^T(D^S_n)\xrightarrow{D^T(\eta^S_n)}D^T(S_n)=(T\circ S)_n\xrightarrow{(\eta^{T\circ S}_n)^{-1}} D^{T\circ S}_n,
\end{equation*}
where $\eta^S_n$ and $\eta^{T\circ S}_n$ are the isomorphisms from Lemma~\ref{lem:class-sized-restrict}. To establish that $\zeta^{T,S}_X$ is surjective we consider an arbitrary element $\langle c,\langle\{\rho_1,\dots,\rho_k\},\tau\rangle\rangle\in D^{T\circ S}_X$. Define $a_j=[\en_c]^{<\omega}(\supp^S_{|c|}(\rho_i))$ and write $\iota_j:a_j\hookrightarrow c$ for the inclusions. Using the support condition for $S$ we get an element $\sigma_j\in S_{|a_j|}$ with $\rho_j=S_{|\iota_j|}(\sigma_j)$. In view of
\begin{equation*}
 [\en_c\circ |\iota_j|]^{<\omega}(\supp^S_{|a_j|}(\sigma_j))=[\en_c]^{<\omega}(\supp^S_{|c|}(\rho_j))=a_j
\end{equation*}
we have $\langle a_i,\sigma_i\rangle\in D^S_X$. One can check that $\langle\{\langle a_1,\sigma_1\rangle,\dots,\langle a_k,\sigma_k\rangle\},\tau\rangle\in D^T(D^S_X)$ is the desired preimage under $\zeta^{T,S}_X$. The support formula given in the lemma follows by unravelling definitions.
\end{proof}

We should also consider compositions in the normal case:

\begin{definition}[$\rca_0$]\label{def:compose-normal-dils}
Let $T=(T,\mu^T)$ and $S=(S,\mu^S)$ be normal prae-dilators. We define a family of functions $\mu^{T\circ S}_n:n\rightarrow(T\circ S)_n=D^T(S_n)$ by setting
\begin{equation*}
 \mu^{T\circ S}_n(m)=D^{\mu^T}_{S_n}\circ \mu^S_n(m)=\langle\{\mu^S_n(m)\},\mu^T_1(0)\rangle
\end{equation*}
 for all numbers $m<n$.
\end{definition}

We verify the expected property:

\begin{lemma}[$\rca_0$]\label{lem:compose-normal-dils}
 If $(T,\mu^T)$ and $(S,\mu^S)$ are normal prae-dilators, then so is~$(T\circ S,\mu^{T\circ S})$. Furthermore we have
 \begin{equation*}
  \zeta^{T,S}_X\circ D^{\mu^T}_{D^S_X}\circ D^{\mu^S}_X=D^{\mu^{T\circ S}}_X
 \end{equation*}
 for any order $X$, where $\zeta^{T,S}_X$ is the isomorphism from Proposition~\ref{prop:compose-dils-rca}.
\end{lemma}
\begin{proof}
 Proposition~\ref{prop:compose-dils-rca} tells us that $T\circ S$ is a prae-dilator. The fact that $\mu^{T\circ S}$ is a natural transformation is readily deduced from Proposition~\ref{prop:reconstruct-normal-dil}. To verify the equivalence from Definition~\ref{def:coded-normal-dil} we consider an arbitrary element $\rho=\langle\{\sigma_1,\dots,\sigma_k\},\tau\rangle$ of $(T\circ S)_n=D^T(S_n)$. By Proposition~\ref{prop:reconstruct-normal-dil} and the normality of $S$ we get
 \begin{align*}
  \rho<_{(T\circ S)_n}\mu^{T\circ S}_n(m)=D^{\mu^T}_{S_n}\circ \mu^S_n(m)&\Leftrightarrow\{\sigma_1,\dots,\sigma_k\}\lef_{S_n}\mu^S_n(m)\\
  &\Leftrightarrow\supp^{T\circ S}_n(\rho)=\bigcup_{i=1,\dots,k}\supp^S_n(\sigma)\lef m.
 \end{align*}
 The equality asserted in the lemma can be verified by unravelling definitions.
\end{proof}

Let us now look at natural transformations between coded prae-dilators. To define their extensions beyond the natural numbers we will use the following result of Girard (the given proof is similar to that of~\cite[Proposition~2.3.15]{girard-pi2}):

\begin{lemma}[$\rca_0$]\label{lem:dilator-cartesian}
 Any natural transformation $\nu:T\Rightarrow S$ between coded prae-dilators satisfies ${\supp^S}\circ\nu=\supp^T$.
\end{lemma}
\begin{proof}
 Consider a number~$n$ and an element $\sigma\in T_n$. By the support condition from Definition~\ref{def:coded-prae-dilator} we have $\sigma=T_{\iota_\sigma\circ\en_\sigma}(\sigma_0)$ for some $\sigma_0\in T_m$, with $m=|\supp^T_n(\sigma)|$. Using the naturality of $\nu$ and $\supp^S$ we get
 \begin{multline*}
  \supp^S_n(\nu_n(\sigma))=\supp^S_n(\nu_n\circ T_{\iota_\sigma\circ\en_\sigma}(\sigma_0))=\supp^S_n(S_{\iota_\sigma\circ\en_\sigma}\circ\nu_m(\sigma_0))=\\
  =[\iota_\sigma\circ\en_\sigma]^{<\omega}(\supp^S_m(\nu_m(\sigma_0)))\subseteq\rng(\iota_\sigma)=\supp^T_n(\sigma).
 \end{multline*}
 Aiming at a contradiction, let us now assume that there is a $k\in\supp^T_n(\sigma)$ that does not lie in $\supp^S_n(\nu_n(\sigma))$. Consider the functions $f_1,f_2:n\rightarrow n+1$ with
 \begin{equation*}
  f_1(i)=\begin{cases}
          i & \text{if $i\leq k$},\\
          i+1 & \text{if $i>k$},
         \end{cases}\qquad
  f_2(i)=\begin{cases}
          i & \text{if $i<k$},\\
          i+1 & \text{if $i\geq k$}.
         \end{cases}
 \end{equation*}
 Observe that we have
 \begin{equation*}
  k=f_1(k)\in[f_1]^{<\omega}(\supp^T_n(\sigma))=\supp^T_{n+1}(T_{f_1}(\sigma)),
 \end{equation*}
 as well as
 \begin{equation*}
  k\notin\rng(f_2)\supseteq[f_2]^{<\omega}(\supp^T_n(\sigma))=\supp^T_{n+1}(T_{f_2}(\sigma)).
 \end{equation*}
 Thus $T_{f_1}(\sigma)$ and $T_{f_2}(\sigma)$ are distinguished by their supports. Since $\nu_{n+1}$ is injective we obtain
 \begin{equation*}
  S_{f_1}\circ\nu_n(\sigma)=\nu_{n+1}\circ T_{f_1}(\sigma)\neq\nu_{n+1}\circ T_{f_2}(\sigma)=S_{f_2}\circ\nu_n(\sigma).
 \end{equation*}
 By Definition~\ref{def:coded-prae-dilator} we may write $\nu_n(\sigma)=S_{\iota_{\nu_n(\sigma)}\circ\en_{\nu_n(\sigma)}}(\tau_0)$. Since~$k$ is not contained in $\rng(\iota_{\nu_n(\sigma)})=\supp^S_n(\nu_n(\sigma))$ we have
 \begin{equation*}
  f_1\circ\iota_{\nu_n(\sigma)}\circ\en_{\nu_n(\sigma)}=f_2\circ \iota_{\nu_n(\sigma)}\circ\en_{\nu_n(\sigma)}.
 \end{equation*}
 Thus we get
 \begin{equation*}
  S_{f_1}\circ\nu_n(\sigma)=S_{f_1}\circ S_{\iota_{\nu_n(\sigma)}\circ\en_{\nu_n(\sigma)}}(\tau_0)=S_{f_2}\circ S_{\iota_{\nu_n(\sigma)}\circ\en_{\nu_n(\sigma)}}(\tau_0)=S_{f_2}\circ\nu_n(\sigma),
 \end{equation*}
 which contradicts the inequality established above.
\end{proof}

Let us remark that ${\supp^S}\circ\nu=\supp^T$ is equivalent to the assertion that $\nu$ is Cartesian (i.\,e.~that the naturality squares for~$\nu$ are pullbacks). Thus the latter holds for any natural transformation between prae-dilators, as pointed out by P.~Taylor~\cite{taylor98} (the first author would like to thank Thomas Streicher for this reference and for enlightening explanations). For us the lemma is important since it ensures that the uniqueness condition $\supp^T_{|a|}(\sigma)=|a|$ from Definition~\ref{def:coded-prae-dilator-reconstruct} is preserved under natural transformations, which justifies the definition of $D^\nu$:

\begin{definition}[$\rca_0$]\label{def:morphism-dils}
 Given coded prae-dilators $T$ and $S$, any natural transformation $\nu:T\Rightarrow S$ is called a morphism of coded prae-dilators. If $T=(T,\mu^T)$ and $S=(S,\mu^S)$ are normal and we have $\nu\circ\mu^T=\mu^S$, then $\nu$ is called a morphism of normal prae-dilators. To extend~$\nu$ beyond the category of natural numbers we define $D^\nu_X:D^T_X\rightarrow D^S_X$ by setting
 \begin{equation*}
  D^\nu_X(\langle a,\sigma\rangle)=\langle a,\nu_{|a|}(\sigma)\rangle
 \end{equation*}
 for each linear order~$X$.
\end{definition}

Let us verify that the extension of a morphism has the expected property:

\begin{lemma}[$\rca_0$]\label{lem:morph-dilators-extend}
 If $\nu:T\Rightarrow S$ is a morphism of (normal) prae-dilators, then the maps $D^\nu_X:D^T_X\rightarrow D^S_X$ form a natural transformation (and $D^\nu_X\circ D^{\mu^T}_X=D^{\mu^S}_X$ holds for every order~$X$). Furthermore we have ${\supp^{D^S}_X}\circ D^\nu_X=\supp^{D^T}_X$.
\end{lemma}
\begin{proof}
 To see that each function $D^\nu_X:D^T_X\rightarrow D^S_X$ is order preserving it suffices to observe that $T_{|\iota_a^{a\cup b}|}(\sigma)<_{T_{|a\cup b|}}T_{|\iota_b^{a\cup b}|}(\tau)$ implies
 \begin{equation*}
  S_{|\iota_a^{a\cup b}|}(\nu_{|a|}(\sigma))=\nu_{|a\cup b|}(T_{|\iota_a^{a\cup b}|}(\sigma))<_{S_{|a\cup b|}}\nu_{|a\cup b|}(T_{|\iota_b^{a\cup b}|}(\tau))=S_{|\iota_b^{a\cup b}|}(\nu_{|b|}(\tau)),
 \end{equation*}
 using the naturality of $\nu$ and the fact that $\nu_{|a\cup b|}$ is order preserving. The naturality of $D^\nu$ follows from the fact that we have $|[f]^{<\omega}(a)|=|a|$ for any order preserving function $f$. If $\nu$ is a morphism of normal prae-dilators, then we get
 \begin{equation*}
  D^\nu_X\circ D^{\mu^T}_X(x)=D^\nu_X(\langle\{x\},\mu^T_1(0)\rangle)=\langle\{x\},\nu_1\circ\mu^T_1(0)\rangle=\langle\{x\},\mu^S_1(0)\rangle=D^{\mu^S}_X(x).
 \end{equation*}
 The relation between the supports is immediate in view of Definition~\ref{def:coded-prae-dilator-reconstruct}.
\end{proof}

As suggested by the last line of equations, one can show that the general condition $\nu_n\circ\mu^T_n=\mu^S_n$ follows from the special case $n=1$ (write $m<n$ as $m=\iota(0)$ with $\iota:1\rightarrow n$ and use naturality). In practice it is just as straightforward to verify the condition for arbitrary~$n$.  We now have all ingredients to define upper derivatives on the level of coded prae-dilators:

\begin{definition}[$\rca_0$]\label{def:upper-derivative}
 Let $T$ be a normal prae-dilator. An upper derivative of~$T$ consists of a normal prae-dilator $S$ and a morphism $\xi:T\circ S\Rightarrow S$ of normal prae-dilators.
\end{definition}

With the previous definition we have completed our formalization of statement~(2) from the introduction (where $S$ stands for $(S,\xi)$). Of course we want to know that we have recovered the notion of upper derivative for normal functions. This fact can be established in a sufficiently strong set theory:

\begin{proposition}\label{prop:upper-deriv-real}
 Consider normal dilators $T$ and $S$. If there is a natural transformation $\xi:T\circ S\Rightarrow S$, then the  normal function $\alpha\mapsto\otp(D^S_\alpha)$ is an upper derivative of the normal function $\alpha\mapsto\otp(D^T_\alpha)$.
\end{proposition}
Before we prove the proposition, let us remark that $X\mapsto D^T_X$ automatically preserves well-foundedness if~$X\mapsto D^S_X$ does, since
\begin{equation*}
 D^\xi_X\circ\zeta^{T,S}_X\circ D^T(D^{\mu^S}_X):D^T_X\rightarrow D^S_X
\end{equation*}
is an embedding ($\zeta^{T,S}$ is the natural isomorphism from Proposition~\ref{prop:compose-dils-rca}).
\begin{proof}
In view of Proposition~\ref{prop:normal-dil-fct} it suffices to establish $\otp(D^T_\gamma)\leq\gamma$ for any value~$\gamma=\otp(D^S_\alpha)\cong D^S_\alpha$. The required inequality is witnessed by the embeddings
\begin{equation*}
 D^T_\gamma\cong D^T(D^S_\alpha)\xrightarrow{\mathmakebox[2em]{\zeta^{T,S}_\alpha}}D^{T\circ S}_\alpha\xrightarrow{\mathmakebox[2em]{D^\xi_\alpha}}D^S_\alpha\cong\gamma,
\end{equation*}
where the first isomorphism uses the functoriality of $D^T$ (cf.~Proposition~\ref{prop:reconstruct-class-sized-dil}).
\end{proof}

To conclude the discussion of upper derivatives we record an immediate consequence of Lemmas~\ref{lem:compose-normal-dils}~and~\ref{lem:morph-dilators-extend}. The equality in the corollary has an intuitive meaning, which will become clear in the proof of Theorem~\ref{thm:equalizer-to-derivative}.

\begin{corollary}[$\rca_0$]\label{cor:deriv-normal-witness}
Assume that $(S,\xi)$ is an upper derivative of a normal prae-dilator~$T$. Then we have
\begin{equation*}
D^\xi_X\circ\zeta^{T,S}_X\circ D^{\mu^T}_{D^S_X}\circ D^{\mu^S}_X=D^{\mu^S}_X
\end{equation*}
for any order~$X$.
\end{corollary}

As a final topic of this section we consider derivatives of normal prae-dilators. On the level of normal functions the derivative is the upper derivative with the smallest possible values. In a categorical setting this is naturally expressed via the notion of initial object. To make this precise we need the following construction:

\begin{definition}[$\rca_0$]\label{def:comp-morphs}
 Given coded prae-dilators $T,S^1,S^2$ and a natural transformation $\nu:S^1\Rightarrow S^2$, we define a family of functions $T(\nu)_n:(T\circ S^1)_n\rightarrow (T\circ S^2)_n$ by setting
 \begin{equation*}
  T(\nu)_n=D^T(\nu_n)
 \end{equation*}
 for each number~$n$. 
\end{definition}

We verify the expected properties:

 \begin{lemma}[$\rca_0$]
  Let $T$ be a (normal) prae-dilator. If $\nu:S^1\Rightarrow S^2$ is a morphism of (normal) prae-dilators, then so is $T(\nu):T\circ S^1\Rightarrow T\circ S^2$. Furthermore we have
  \begin{equation*}
   D^{T(\nu)}_X\circ\zeta^{T,S^1}_X=\zeta^{T,S^2}_X\circ D^T(D^\nu_X)
  \end{equation*}
  for each order~$X$, where $\zeta^{T,S^i}$ are the natural isomorphisms from Proposition~\ref{prop:compose-dils-rca}.
 \end{lemma}
 \begin{proof}
 Using Proposition~\ref{prop:reconstruct-class-sized-dil} and the naturality of $\nu$ one readily shows that $T(\nu)$ is a natural family of embeddings. If $\nu$ is a morphism of normal prae-dilators, then we invoke the naturality of $D^{\mu^T}$ (due to Proposition~\ref{prop:reconstruct-normal-dil}) to get
 \begin{equation*}
  T(\nu)_n\circ\mu^{T\circ S^1}_n=D^T(\nu_n)\circ D^{\mu^T}_{S^1_n}\circ\mu^{S^1}_n=D^{\mu^T}_{S^2_n}\circ\nu_n\circ\mu^{S^1}_n=D^{\mu^T}_{S^2_n}\circ\mu^{S^2}_n=\mu^{T\circ S^2}_n,
 \end{equation*}
 which shows that $T(\nu)$ is a morphism of normal prae-dilators. The equality asserted in the lemma can be verified by unravelling definitions.
 \end{proof}

Let us introduce a last ingredient for the definition of derivatives:
 
 \begin{definition}[$\rca_0$]\label{def:morph-upper-derivs}
  Consider a normal prae-dilator $T$ with upper derivatives $(S^1,\xi^1)$ and $(S^2,\xi^2)$. A morphism of upper derivatives is a morphism $\nu:S^1\Rightarrow S^2$ of normal prae-dilators that satisfies $\nu\circ\xi^1=\xi^2\circ T(\nu)$.
 \end{definition}

Finally, we can characterize derivatives on the level of coded prae-dilators:

\begin{definition}[$\rca_0$]\label{def:derivative}
 A derivative of a normal prae-dilator $T$ is an upper derivative $(S,\xi)$ of $T$ that is initial in the following sense: For any upper derivative $(S',\xi')$ of $T$ there is a unique morphism $\nu:S\Rightarrow S'$ of upper derivatives.
\end{definition}

Due to the form of the given definition it is clear that the derivative of a normal prae-dilator is unique up to isomorphism of upper derivatives. Other properties of the derivative are harder to establish. In Sections~\ref{sect:constructin-derivative} and~\ref{sect:bi-deriv-wf} we will show that the assumptions of the following theorem hold when $(S,\xi)$ is a derivative of~$T$. This leads to an unconditional version of the result, which will be stated as Corollary~\ref{cor:deriv-dil-to-fct}.

\begin{theorem}\label{thm:equalizer-to-derivative}
Let $(S,\xi)$ be an upper derivative of a normal dilator $T$. Assume that the maps $\xi_n:(T\circ S)_n\rightarrow S_n$ are surjective (so that $\xi$ is an isomorphism), that
\begin{equation*}
   \begin{tikzcd}
  n\ar{r}{\mu^S_n} &[2em] S_n\arrow[r,shift left,"\id_{S_n}"]\arrow[r,shift right,swap,"\xi_n\circ D^{\mu^T}_{S_n}"]&[5em] S_n
 \end{tikzcd}
 \end{equation*}
 is an equalizer diagram for every~$n$, and that $X\mapsto D^S_X$ preserves well-foundedness. Then $\alpha\mapsto\otp(D^S_\alpha)$ is the derivative of the normal function~$\alpha\mapsto\otp(D^T_\alpha)$.
\end{theorem}
Before we prove the theorem we motivate the equalizer condition: By Definition~\ref{def:compose-normal-dils} and the fact that $\xi$ is a morphism of normal prae-dilators we get
\begin{equation*}
 \xi_n\circ D^{\mu^T}_{S_n}\circ\mu^S_n=\xi_n\circ\mu^{T\circ S}_n=\mu^S_n.
\end{equation*}
So the assumption that the equalizer diagrams commute is automatic. After Definition~\ref{def:coded-normal-dil} we have explained that $\mu^T$ can be seen as an internal version of the function~$\alpha\mapsto\otp(D^T_\alpha)$. Intuitively speaking, the equalizer condition thus demands that any ordinal $\alpha$ with $\otp(D^T_\alpha)=\alpha$ lies in the range of $\alpha\mapsto\otp(D^S_\alpha)$.
\begin{proof}
 As a preparation we lift the assumptions of the theorem to the level of class-sized dilators: In view of Definition~\ref{def:morphism-dils} it is straightforward to show that each function $D^\xi_X:D^{T\circ S}_X\rightarrow D^S_X$ is an isomorphism. From Corollary~\ref{cor:deriv-normal-witness} we know that
 \begin{equation*}
 \begin{tikzcd}
  X\ar{r}{D^{\mu^S}_X} &[2em] D^S_X\arrow[r,shift left,"\id_{D^S_X}"]\arrow[r,shift right,swap,"D^\xi_X\circ\zeta^{T,S}_X\circ D^{\mu^T}_{D^S_X}"]&[5em] D^S_X
 \end{tikzcd}
 \end{equation*}
 is a commutative diagram. To show that it defines an equalizer we consider an arbitrary element $\langle a,\sigma\rangle\in D^S_X$. Invoking Definitions~\ref{def:extend-normal-transfos} and~\ref{def:morphism-dils}, as well as the proof of Proposition~\ref{prop:compose-dils-rca}, we see
 \begin{multline*}
  D^\xi_X\circ\zeta^{T,S}_X\circ D^{\mu^T}_{D^S_X}(\langle a,\sigma\rangle)=D^\xi_X\circ\zeta^{T,S}_X(\langle\{\langle a,\sigma\rangle\},\mu^T_1(0)\rangle)=\\
  =D^\xi_X(\langle a,\langle\{\sigma\},\mu^T_1(0)\rangle\rangle)=\langle a,\xi_{|a|}(\langle\{\sigma\},\mu^T_1(0)\rangle)\rangle=\langle a,\xi_{|a|}\circ D^{\mu^T}_{S_{|a|}}(\sigma)\rangle.
 \end{multline*}
 If this value is equal to $\langle a,\sigma\rangle$, then we have $\xi_{|a|}\circ D^{\mu^T}_{S_{|a|}}(\sigma)=\sigma$. Thus the equalizer condition from the theorem yields $\sigma=\mu^S_{|a|}(m)$ for some $m<|a|$. According to Definition~\ref{def:coded-prae-dilator-reconstruct} and Lemma~\ref{lem:support-mu} we must have
 \begin{equation*}
  |a|=\supp^S_{|a|}(\sigma)=\supp^S_{|a|}(\mu^S_{|a|}(m))=\{m\}.
 \end{equation*}
 This forces $m=0$ and $|a|=1$, say $a=\{x\}$. We can conclude
 \begin{equation*}
  \langle a,\sigma\rangle=\langle\{x\},\mu^S_1(0)\rangle=D^{\mu^S}_X(x)\in\rng(D^{\mu^S}_X),
 \end{equation*}
 as required to make the above an equalizer diagram. Based on these preparations we can now prove the actual claim of the theorem: Write $f$ for the normal function $\alpha\mapsto\otp(D^T_\alpha)$ and $f'$ for its derivative. Proposition~\ref{prop:upper-deriv-real} yields $\otp(D^S_\beta)\geq f'(\beta)$. We may thus define an order embedding $f'_S:\beta\rightarrow D^S_\beta$ by stipulating
 \begin{equation*}
  \otp(D^S_\beta\!\restriction\!f'_S(\alpha))=f'(\alpha).
 \end{equation*}
 To make use of the equalizer diagram from the beginning of the proof we need
 \begin{equation*}
  D^\xi_\beta\circ\zeta^{T,S}_\beta\circ D^{\mu^T}_{D^S_\beta}\circ f'_S(\alpha)=f'_S(\alpha).
 \end{equation*}
 Since $D^\xi_\beta\circ\zeta^{T,S}_\beta:D^T(D^S_\beta)\rightarrow D^S_\beta$ is an isomorphism this reduces to
 \begin{equation*}
  \otp(D^T(D^S_\beta)\!\restriction\!D^{\mu^T}_{D^S_\beta}\circ f'_S(\alpha))=f'(\alpha).
 \end{equation*}
 By the proof of Proposition~\ref{prop:normal-dil-fct} and the functoriality of $D^T$, the left side is indeed equal to
 \begin{equation*}
  \otp(D^T(D^S_\beta\!\restriction\!f'_S(\alpha)))=\otp(D^T_{f'(\alpha)})=f(f'(\alpha))=f'(\alpha).
 \end{equation*}
 Now the universal property of equalizers yields an embedding $g:\beta\rightarrow\beta$ that satisfies $D^{\mu^S}_\beta\circ g=f'_S$. Since $g$ is a strictly increasing function on the ordinals we have $\alpha\leq g(\alpha)$. Thus, again invoking the proof of Proposition~\ref{prop:normal-dil-fct}, we get
 \begin{equation*}
  \otp(D^S_\alpha)=\otp(D^S_\beta\!\restriction\!D^{\mu^S}_\beta(\alpha))\leq\otp(D^S_\beta\!\restriction\!D^{\mu^S}_\beta\circ g(\alpha))=\otp(D^S_\beta\!\restriction\!f'_S(\alpha))=f'(\alpha).
 \end{equation*}
 We have already seen $\otp(D^S_\alpha)\geq f'(\alpha)$. So we can conclude that $\alpha\mapsto\otp(D^S_\alpha)$ coincides with the derivative~$f'$ of the normal function $\alpha\mapsto\otp(D^T_\alpha)$, as desired.
\end{proof}

In Example~\ref{ex:equalizers-without-deriv} we will exhibit an upper derivative $(S,\xi)$ that satisfies the equalizer condition but fails to be a derivative in the sense of Definition~\ref{def:derivative}. This shows that the equalizer condition does not suffice to characterize derivatives on the categorical level. The relevance of Example~\ref{ex:equalizers-without-deriv} is somewhat diminished by the fact that $X\mapsto D^S_X$ does not preserve well-foundedness.

\section{From upper derivative to $\Pi^1_1$-bar induction}\label{sect:upper-deriv-bi}

In this section we prove that bar induction for $\Pi^1_1$-formulas follows from the principle that every normal dilator has an upper derivative that preserves well-foundedness (i.\,e.~that is again a normal dilator). To establish this result we follow the informal argument given at the beginning of the introduction.

The first major goal of the section is to reconstruct the functions~$h$ and~$f$ from the informal argument in terms of dilators. Since the notion of dilator is \mbox{$\Pi^1_2$-complete} (due to Girard) it makes sense that this is possible. Indeed our reconstruction of~$h$ is inspired by Norman's proof of $\Pi^1_2$-completeness (see~\cite[\mbox{Annex 8.E}]{girard-book-part2}). To get a usable result we will have to adapt his argument to the specific form of bar induction. Our reconstruction of~$f$ can be read as a proof that the more restrictive class of normal dilators is $\Pi^1_2$-complete as well.

Let us begin with some terminology: Given a set $Y$, we write $Y^{<\omega}$ for the tree of finite sequences with entries in $Y$. If $Y=(Y,<_Y)$ is a linear order, then the Kleene-Brouwer order (also known as Lusin-Sierpi\'nski order) on $Y^{<\omega}$ is defined by
\begin{equation*}
 \sigma<_{\operatorname{KB}(Y)}\tau\quad\Leftrightarrow\quad\begin{cases}
                                                                   \text{either $\sigma$ is a proper end extension of $\tau$},\\
                                                                   \text{or we have $(\sigma)_j<_Y(\tau)_j$ for the smallest $j$ with $(\sigma)_j\neq(\tau)_j$.}
                                                                  \end{cases}
\end{equation*}
In the second clause $(\sigma)_j$ refers to the $j$-th entry of $\sigma$, for $j<\len(\sigma)$ below the length of $\sigma$ (note that such a $j$ exists when neither sequence is an end extension of the other). If we want to emphasize that $\mathcal T$ is ordered as a subtree of $Y^{<\omega}$, then we say that it carries the Kleene-Brouwer order with respect to~$Y$. The symbol~$\lkb$ will be reserved for the Kleene-Brouwer order with respect to~$\mathbb N$ (ordered as usual). Recall that a branch of $\mathcal T\subseteq Y^{<\omega}$ is given by a function $f:\mathbb N\rightarrow Y$ such that we have $f[n]\in\mathcal T$ for all $n\in\mathbb N$, where the sequence
\begin{equation*}
 f[n]=\langle f(0),\dots,f(n-1)\rangle
\end{equation*}
lists the first $n$ values of $f$. It is well-known that $\aca_0$ proves the characteristic property of the Kleene-Brouwer order: If $Y$ is a well-order, then $\mathcal T$ has no branch if, and only if, the Kleene-Brouwer order with respect to~$Y$ is well-founded on~$\mathcal T$ (to adapt the proof from~\cite[Lemma~V.1.3]{simpson09}, which treats the case $Y=\mathbb N$, one observes that~$\mathbb N$ embeds into any infinite sub\-order~$Y_0\subseteq Y$, e.\,g.~as an initial segment).

Given an order $X=(X,<_X)$, an $X$-indexed family of orders is given as a set
\begin{equation*}
 Y=\{\langle x,y\rangle\,|\,x\in X\text{ and }y\in Y_x\},
\end{equation*}
where $Y_x=(Y_x,<_{Y_x})$ is an order for each $x\in X$. The dependent sum $\Sigma_{x\in X}Y_x$ (or shorter~$\Sigma Y$) is the order with underlying set $Y$ and order relation given by
\begin{equation*}
 \langle x,y\rangle<_{\Sigma Y}\langle x',y'\rangle\quad\Leftrightarrow\quad\begin{cases}
                                                                                   \text{either $x<_X x'$,}\\
                                                                                   \text{or $x=x'$ and $y<_{Y_x}y'$}.
                                                                                  \end{cases}
\end{equation*}
For $x\in X$ we write $\Sigma_{x'<_Xx}Y_{x'}$ (or shorter $\Sigma_x Y$) for the suborder that contains all pairs $\langle x',y\rangle\in\Sigma Y$ with $x'<_Xx$. If $X$ is a well-order, then $\Sigma Y$ is well-founded if, and only if, $Y_x$ is well-founded for every $x\in X$, provably in $\rca_0$ (the first components of a descending sequence in $\Sigma Y$ must become constant with some value~$x\in X$, from which point on the second components form a descending sequence in~$Y_x$). The product $X\times Y$ of two orders is explained as the special case where we have $Y_x=Y$ for all $x\in X$. Let us mention one other construction that will be needed later: Given an order $Y=(Y,<_Y)$, we write
\begin{equation*}
 Y^\bot=\{\bot\}\cup Y
\end{equation*}
for the extension of $Y$ by a new minimal element (i.\,e.~we have $\bot<_{Y^\bot}y<_{Y^\bot}y'$ for any $y,y'\in Y$ with $y<_Yy'$). If $f:Y\rightarrow Z$ is an embedding, then we get an embedding $f^\bot:Y^\bot\rightarrow Z^\bot$ by setting
\begin{equation*}
 f^\bot(y)=\begin{cases}
            f(y) & \text{if $y\in Y\subseteq Y^\bot$,}\\
            \bot & \text{if $y=\bot$.}
           \end{cases}
\end{equation*}
One readily verifies that the construction is functorial (and in fact a dilator), in the sense that we have $(g\circ f)^\bot=g^\bot\circ f^\bot$ for functions $f,g$ of suitable (co-)domain.

We will be particularly interested in dependent sums of the form $\Sigma\mathcal T=\Sigma_{x\in X}\mathcal T_x$, where $X$ is a well-order and each $\mathcal T_x$ is a subtree of $\mathbb N^{<\omega}$, with the usual Kleene-Brouwer order. In this situation we call $\mathcal T$ an $X$-indexed family of $\mathbb N$-trees. As in the informal argument from the introduction, the idea is that the well-foundedness of $\mathcal T_x$ corresponds to the fact that some $\Pi^1_1$-formula $\varphi$ holds at $x\in X$. Above we have seen that $\Sigma_x\mathcal T$ is well-founded if, and only if, $\mathcal T_y$ is well-founded for all~$y<_X x$. Thus it makes sense to call $\mathcal T$ progressive at $x\in X$ if we have
\begin{equation*}
 \text{``$\Sigma_x\mathcal T$ is well-founded''}\rightarrow\text{``$\mathcal T_x$ is well-founded''}.
\end{equation*}
If $\mathcal T$ is progressive at every~$x\in X$, then it is called progressive along~$X$.

Let us now describe our reconstruction of the function $h$: The ordinal~$\alpha$ and the induction formula $\varphi$ that appear in the informal argument from the introduction correspond to a well-order~$X$ and an $X$-indexed family $\mathcal T$ of $\mathbb N$-trees. In order to represent the function~$\delta\mapsto h(\gamma,\delta)$ with $\gamma<\alpha$ we will construct a prae-dilator~$H[x]$ such that $\mathcal T$ is progressive at $x\in X$ if, and only if, $H[x]$ is a dilator (we write $H[x]$ for both the class-sized dilator and its coded restriction, cf.~the discussion after Definition~\ref{def:coded-dilator}). Considering the contra\-positive of the implication from left to right, we see that we should ensure the following: If $H[x]_Z\cong D^{H[x]}_Z$ is ill-founded for some well-order~$Z$, then $\Sigma_x\mathcal T$ must be well-founded while $\mathcal T_x$ is not. Inspired by Norman's proof of \mbox{$\Pi^1_2$-completeness}, the idea is to define $H[x]_Z$ as a tree: Along each potential branch one searches for an embedding of $\Sigma_x\mathcal T$ into~$Z$ and, simultaneously, for a descending sequence in~$\mathcal T_x$. This leads to the following construction:

\begin{definition}[$\rca_0$]\label{def:H-0}
 Consider a well-order~$X$, an $X$-indexed family $\mathcal T$ of \mbox{$\mathbb N$-trees} and an element~$x\in X$. For each order~$Z$ we define $H[x]_Z=H[X,\mathcal T,x]_Z$ as the tree of all sequences
 \begin{equation*}
 \langle\langle z_0,s_0\rangle,\dots,\langle z_{k-1},s_{k-1}\rangle\rangle\in(Z^\bot\times\mathbb N)^{<\omega}
 \end{equation*}
 that satisfy the following:
 \begin{enumerate}[label=(\roman*)]
  \item Whenever $j_1,j_2<k$ code elements $j_i=\langle y_i,\sigma_i\rangle\in\Sigma_x\mathcal T$, we have $z_{j_i}\in Z$ (i.\,e.~$z_{j_i}\neq\bot$) and
  \begin{equation*}
   \langle y_1,\sigma_1\rangle<_{\Sigma\mathcal T}\langle y_2,\sigma_2\rangle\quad\Rightarrow\quad z_{j_1}<_Z z_{j_2}.
  \end{equation*}
  \item We have $\langle s_0,\dots,s_{k-1}\rangle\in\mathcal T_x$.
 \end{enumerate}
 The order on $H[x]_Z$ is the Kleene-Brouwer order with respect to~$Z^\bot\times\mathbb N$. For an embedding $f:Z\rightarrow Y$ we define $H[x]_f:H[x]_Z\rightarrow H[x]_Y$ by
 \begin{equation*}
  H[x]_f(\langle\langle z_0,s_0\rangle,\dots,\langle z_{k-1},s_{k-1}\rangle\rangle)=\langle\langle f^\bot(z_0),s_0\rangle,\dots,\langle f^\bot(z_{k-1}),s_{k-1}\rangle\rangle.
 \end{equation*}
 To define a family of functions $\supp^{H[x]}_Z:H[x]_Z\rightarrow[Z]^{<\omega}$ we set
 \begin{equation*}
  \supp^{H[x]}_Z(\langle\langle z_0,s_0\rangle,\dots,\langle z_{k-1},s_{k-1}\rangle\rangle)=\{z_j\,|\,j<k\text{ and }z_j\neq\bot\}
 \end{equation*}
 for each order~$Z$.
\end{definition}

The following is readily verified:

\begin{lemma}[$\rca_0$]\label{lem:H-0-prae-dil}
 The orders and functions that we have constructed in the previous definition form a prae-dilator $H[x]=H[X,\mathcal T,x]$.
\end{lemma}

Since the previous result is formulated in $\rca_0$, it is officially concerned with the coded restriction of $H[x]$ to the category of natural numbers. The following technical observation shows that the class-sized and coded versions of $H[x]$ can be identified (cf.~the discussion after Definition~\ref{def:coded-dilator}).

\begin{lemma}[$\rca_0$]\label{lem:H-0-alternative}
 For each order~$Z$ we get an isomorphism $D^{H[x]}_Z\cong H[x]_Z$ by stipulating
 \begin{multline*}
  \langle a,\langle\langle n_0,s_0\rangle,\dots,\langle n_{k-1},s_{k-1}\rangle\rangle\,\rangle\mapsto\\
  \langle\langle(\iota_a\circ\en_a)^\bot(n_0),s_0\rangle,\dots,\langle(\iota_a\circ\en_a)^\bot(n_{k-1}),s_{k-1}\rangle\rangle
 \end{multline*}
 where $\iota_a\circ\en_a:|a|\rightarrow Z$ is the unique embedding with range $a\in[Z]^{<\omega}$.
\end{lemma}
\begin{proof}
 To verify the claim one follows the proof of Lemma~\ref{lem:class-sized-restrict}.
\end{proof}

We can now deduce the connection with the premise of $\Pi^1_1$-bar induction. As mentioned before, this part of our argument is similar to Norman's proof that the notion of dilator is \mbox{$\Pi^1_2$-complete} (see~\cite[Annex~8.E]{girard-book-part2}). We choose the base theory $\aca_0$ since we will apply the characteristic property of the Kleene-Brouwer order:

\begin{proposition}[$\aca_0$]\label{prop:H-0-captures-prog-new}
 An $X$-indexed family $\mathcal T$ of $\mathbb N$-trees is progressive at $x\in X$ if, and only if, $H[x]=H[X,\mathcal T,x]$ is a dilator.
\end{proposition}
\begin{proof}
 By Lemma~\ref{lem:H-0-prae-dil} we already know that $H[x]$ is a prae-dilator. Thus we need to show that the implication
 \begin{equation*}
 \text{``$\Sigma_x\mathcal T$ is well-founded''}\rightarrow\text{``$\mathcal T_x$ is well-founded''}
 \end{equation*}
 holds if, and only if, $H[x]_Z\cong D^{H[x]}_Z$ is well-founded for every well-order~$Z$. Aiming at the contrapositive of the first direction, assume that $Z$ is a well-order such that $H[x]_Z$ is ill-founded. By the characteristic property of the Kleene-Brouwer order we get a branch in the tree $H[x]_Z\subseteq(Z^\bot\times\mathbb N)^{<\omega}$. In view of Definition~\ref{def:H-0} the first components of this branch form an embedding of $\Sigma_x\mathcal T$ into the well-order $Z$, which witnesses the premise of the above implication. The second components of our branch form a branch in the tree $\mathcal T_x$, so that the latter is ill-founded. Thus our implication fails and the first direction is established. For the other direction we assume that $H[x]$ is a dilator. Assuming the premise of the above implication, it follows that $H[x]_{\Sigma_x\mathcal T}$ is a well-order. To establish the conclusion of our implication, we construct an embedding of $\mathcal T_x$ into $H[x]_{\Sigma_x\mathcal T}$. Define $e:\mathbb N\to(\Sigma_x\mathcal T)^\bot$ by
\begin{equation*}
e(j)=\begin{cases}
\langle y,\sigma\rangle & \text{if $j$ codes the element $\langle y,\sigma\rangle\in\Sigma_x\mathcal T$},\\
\bot & \text{if $j$ does not code an element of $\Sigma_x\mathcal T$}.
\end{cases}
\end{equation*}
It is straightforward to verify that
\begin{equation*}
\mathcal T_x\ni\langle s_0,\dots,s_{k-1}\rangle\mapsto\langle\langle e(0),s_0\rangle,\dots,\langle e(k-1),s_{k-1}\rangle\rangle\in H[x]_{\Sigma_x\mathcal T}
\end{equation*}
is the required embedding.
\end{proof}

So far we have reconstructed the function~$h$ from the informal argument from the introduction. In the rest of this section we reconstruct the normal function~$f$ and conclude the proof of $\Pi^1_1$-bar induction. Since the formalization in second order arithmetic is somewhat technical, we begin with an informal exposition:

\begin{summary}\label{sum:deduce-Pi11-bi}
In the informal argument from the introduction we have transformed~$h$ into a function $h_0$ with $h_0(\delta)=\sup_{\gamma\leq\delta} h(\gamma,\delta)$. The supremum does not seem to be available on the level of dilators, but it can be bounded by a transfinite sum: For each order~$Z$ we consider the order
\begin{equation*}
H^0_Z=\Sigma_{x\in X}H[x]_Z.
\end{equation*}
Here $H[x]=H[X,\mathcal T,x]$ is constructed with respect to a fixed well-order~$X$ and an $X$-indexed family $\mathcal T$ of $\mathbb N$-trees that is progressive along~$X$. It is straightforward to turn $H^0$ into a prae-dilator, and Proposition~\ref{prop:H-0-captures-prog-new} implies that~$H^0$ is a dilator. The informal argument proceeds with a normal function $f$ with $f(\delta)=\sum_{\gamma<\delta}1+h_0(\gamma)$. If $\delta$ is represented by~$Z$, then $\gamma<\delta$ corresponds to $Z\!\restriction\!z=\{z'\in Z\,|\,z'<_Z z\}$ for some $z\in Z$. Hence~$f$ can be represented by a dilator~$F$ with
\begin{equation*}
F_Z=\Sigma_{z\in Z}(H^0_{Z\restriction z})^\bot=\Sigma_{z\in Z}\left(\Sigma_{x\in X}H[x]_{Z\restriction z}\right)^\bot.
\end{equation*}
Elements of $F_Z$ have the form $\langle z,\bot\rangle$ or $\langle z,x,\sigma\rangle$ (with one pair of parentheses omitted) for $z\in Z$, $x\in X$ and $\sigma\in H[x]_{Z\restriction z}$. One can check that $F$ becomes normal if we define $\mu^F_Z:Z\to F_Z$ by $\mu^F_Z(z)=\langle z,\bot\rangle$ (cf.~Proposition~\ref{prop:F-normal} below). We point out that the definition of $F$ resembles Girard's construction of a flower $\int T$ from a given dilator~$T$ (cf.~\cite[Example~2.4.9]{girard-pi2}). In our setting, the conclusion of $\Pi^1_1$-bar induction amounts to the claim that $\Sigma_{x\in X}\mathcal T_x$ is well-founded. We want to deduce this claim from the assumption that $F$ has an upper derivative
\begin{equation*}
\xi:G\circ F\Rightarrow G
\end{equation*}
such that~$G$ preserves well-foundedness. Since $X$ is a well-order, it suffices to construct an embedding
\begin{equation*}
J:\Sigma_{x\in X}\mathcal T_x\rightarrow G_X.
\end{equation*}
In a sufficiently strong meta theory, the value $J(\langle x,\sigma\rangle)$ can be constructed by recursion on~$x$: The normal dilator $G$ comes with an embedding $\mu^G_X:X\to G_X$. Inductively we assume that we already have values $J(\langle y,\sigma\rangle)\in G_X\!\restriction\!\mu^G_X(x)$ for all~$\langle y,\sigma\rangle\in\Sigma_{y<_X x}\mathcal T_y=\Sigma_x\mathcal T$. We can then define $J_x:\mathbb N\to (G_X\!\restriction\!\mu^G_X(x))^\bot$ by
\begin{equation*}
J_x(j)=\begin{cases}
J(\langle y,\sigma\rangle) & \text{if $j$ codes the element $\langle y,\sigma\rangle\in\Sigma_x\mathcal T$},\\
\bot & \text{if $j$ does not code an element of $\Sigma_x\mathcal T$}.
\end{cases}
\end{equation*}
Assuming that~$J$ is order preserving on~$\Sigma_x\mathcal T$, we have
\begin{equation*}
\langle s_0,\dots,s_{k-1}\rangle\in\mathcal T_x\quad\Rightarrow\quad\langle\langle J_x(0),s_0\rangle,\dots,\langle J_x(k-1),s_{k-1}\rangle\rangle\in H[x]_{G_X\restriction\mu^G_X(x)}.
\end{equation*}
Using the function $\xi_X:F(G_X)\to G_X$, we can now define $J(\langle x,\langle s_0,\dots,s_{k-1}\rangle\rangle)$ as
\begin{equation*}
\xi_X(\langle \mu^G_X(x),x,\langle\langle J_x(0),s_0\rangle,\dots,\langle J_x(k-1),s_{k-1}\rangle\rangle\rangle)\in G_X.
\end{equation*}
For $x<_X x'$ the argument of $\xi_X$ is below $\langle\mu^G_X(x'),\bot\rangle=\mu^F_{G(X)}\circ\mu^G_X(x')=\mu^{F\circ G}_X(x')$. Since $\xi:F\circ G\Rightarrow G$ is a morphism of normal dilators, this implies
\begin{equation*}
 J(\langle x,\langle s_0,\dots,s_{k-1}\rangle\rangle)<_{G(X)}\xi_X\circ\mu^{F\circ G}_X(x')=\mu^G_X(x'),
\end{equation*}
as we have assumed in the recursive construction. In the proof of Theorem~\ref{thm:embed-Sigma-T} we will see that $J$ can be constructed by primitive recursion, since each value $J(\langle x,\sigma\rangle)$ does only depend on a finite (and effectively bounded) collection of previous values.
\end{summary}

In the rest of this section we provide a detailed formalization of the argument from Summary~\ref{sum:deduce-Pi11-bi}. The first step is to define~$F$ as a coded dilator. As explained in the previous section, this means that $F$ acts on finite orders of the form~\mbox{$n=\{0,\dots,n-1\}$}. It will be very convenient to have a more uniform presentation of the orders $(\Sigma_{x\in X}H[x]_n)^\bot$. For this purpose we extend the $X$-indexed family of prae-dilators $H[x]$ to a family indexed by $X^\bot$: Define \mbox{$H[\bot]=H[X,\mathcal T,\bot]$} as the constant prae-dilator with values
\begin{equation*}
H[\bot]_n=\{\star\},
\end{equation*}
where $\star$ is some new symbol. Its action on morphisms and the support functions are given by $H[\bot]_f(\star)=\star$ and $\supp^{H[\bot]}_n(\star)=\emptyset$. Then we have
\begin{equation*}
(\Sigma_{x\in X} H[x]_n)^\bot\cong\Sigma_{x\in X^\bot}H[x]_n,
\end{equation*}
where the isomorphism sends $\bot$ to $\langle\bot,\star\rangle$ and leaves $\langle x,\tau\rangle$ with $x\in X$ unchanged. The point is that all elements of the right side are pairs, which will save us many case distinctions. Invoking Definition~\ref{def:coded-prae-dilator-reconstruct} we see that $D^{H[\bot]}_Z$ consists of the single element $\langle\emptyset,\star\rangle$. Thus $H[\bot]$ is a dilator and we have
\begin{equation*}
\left(\Sigma_{x\in X}D^{H[x]}_Z\right)^\bot\cong\Sigma_{x\in X^\bot}D^{H[x]}_Z.
\end{equation*}
We now define the coded prae-dilator $F$ that reconstructs the function~$f$ from the informal argument. The crucial point is that~$F$ is normal, as we shall see below.

\begin{definition}[$\rca_0$]\label{def:dil-F}
Consider a well-order~$X$ and an $X$-indexed family $\mathcal T$ of $\mathbb N$-trees. For each number $n$ we define
\begin{equation*}
F_n=F[X,\mathcal T]_n=\Sigma_{N\in n}\Sigma_{x\in X^\bot}H[X,\mathcal T,x]_N.
\end{equation*}
Omitting one pair of parentheses, we write the elements of $F_n$ in the form $\langle N,x,\sigma\rangle$ with $N\in n=\{0,\dots,n-1\}$, $x\in X^\bot$ and $\sigma\in H[x]_N$. The order on $F_n$ is the usual order on a dependent sum, which coincides with the lexicographic order on the triples $\langle N,x,\sigma\rangle$. Given an embedding~$f:n\rightarrow m$, we write $f\!\restriction\!N:N\rightarrow f(N)$ for the restriction of $f$. Then we define $F_f:F_n\rightarrow F_m$ by
\begin{equation*}
F_f(\langle N,x,\tau\rangle)=\langle f(N),x,H[x]_{f\restriction N}(\tau)\rangle.
\end{equation*}
The functions $\supp^F_n:F_n\rightarrow[n]^{<\omega}$ are given as
\begin{equation*}
\supp^F_n(\langle N,x,\tau\rangle)=\{N\}\cup\supp^{H[x]}_N(\tau).
\end{equation*}
Finally we define a family of functions $\mu^F_n:n\rightarrow F_n$ by setting
\begin{equation*}
\mu^F_n(N)=\langle N,\bot,\star\rangle
\end{equation*}
for all numbers~$N<n$.
\end{definition}

Let us verify the following:

\begin{proposition}[$\rca_0$]\label{prop:F-normal}
The orders and functions from the previous definition form a normal prae-dilator $F=F[X,\mathcal T]$.
\end{proposition}
\begin{proof}
Using Lemma~\ref{lem:H-0-prae-dil} it is straightforward to show that $F$ is a functor and that $\supp^F$ is a natural transformation. To conclude that $F$ is a prae-dilator we must verify the support condition from clause~(ii) of Definition~\ref{def:coded-prae-dilator}. To do so we consider an arbitrary $\sigma=\langle N,x,\tau\rangle\in F_n$. We abbreviate $a:=\supp^{H[x]}_N(\tau)\subseteq\{0,\dots,N-1\}$ and write $\iota_\sigma\circ\en_\sigma:|a|+1\rightarrow n$ for the embedding with range $\supp^F_n(\sigma)=a\cup\{N\}$. Since the restriction $(\iota_\sigma\circ\en_\sigma)\!\restriction\!|a|:|a|\rightarrow N$ has range~$a$, the support condition for $H[x]$ yields a $\tau_0\in H[x]_{|a|}$ with $\tau=H[x]_{(\iota_\sigma\circ\en_\sigma)\restriction|a|}(\tau_0)$. By construction we have $\langle |a|,x,\tau_0\rangle\in F_{|a\cup\{N\}|}$ and $\sigma=F_{\iota_\sigma\circ\en_\sigma}(\langle |a|,x,\tau_0\rangle)$, which completes the proof of the support condition for $F$. One readily verifies that $\mu^F$ is a natural family of embeddings (for naturality we recall $H[\bot]_{f\restriction N}(\star)=\star$). In view of Definition~\ref{def:coded-normal-dil} it remains to establish
 \begin{equation*}
  \sigma<_{F_n}\mu^F_n(N)\quad\Leftrightarrow\quad\supp^F_n(\sigma)\lef N
 \end{equation*}
 for arbitrary $\sigma\in F_n$ and $N<n$. For the first direction we recall that $\langle\bot,\star\rangle$ is the smallest element of $\Sigma_{x\in X^\bot}H[x]_N$. Thus any $\sigma<_{F_n}\mu^F_n(N)=\langle N,\bot,\star\rangle$ must have the form $\sigma=\langle N',x,\tau\rangle$ with $N'<N$. In view of $\supp^{H[x]}_{N'}(\tau)\in[N']^{<\omega}$ we get
 \begin{equation*}
  \supp^F_n(\sigma)=\{N'\}\cup\supp^{H[x]}_{N'}(\tau)\lef N.
 \end{equation*}
 For the converse we also write $\sigma=\langle N',x,\tau\rangle$. In view of $N'\in\supp^F_n(\sigma)$ the right side of the desired equivalence implies $N'<N$ and thus $\sigma<_{F_n}\mu^F_n(N)$.
\end{proof}

Lemma~\ref{lem:H-0-alternative} provides a transparent description of the orders $D^{H[x]}_Z$ (recall that the latter consist of pairs $\langle a,\tau\rangle$ with $a\in[Z]^{<\omega}$ and $\tau\in H[x]_{|a|}$, see Definition~\ref{def:coded-prae-dilator-reconstruct}). We now describe $D^F_Z$ relative to these orders: 

\begin{definition}[$\rca_0$]\label{def:F-alternative}
Given an order $Z$, we put
\begin{equation*}
F_Z=\{\langle z,x,\langle a,\tau\rangle\rangle\in Z\times\Sigma_{x\in X^\bot}D^{H[x]}_Z\,|\,a\lef_Z z\}.
\end{equation*}
The order on $F_Z$ is the indicated product order, which coincides with the lexicographic order on the triples $\langle z,x,\langle a,\tau\rangle\rangle$ (we again omit a pair of angle parentheses).
\end{definition}

Note that the previous definition of $F_Z$ is similar but not quite identical to the informal construction in Summary~\ref{sum:deduce-Pi11-bi}. The following proof consists in a technical verification, which can be skipped at first reading.

\begin{lemma}[$\rca_0$]\label{lem:F-alternative}
For each order~$Z$ the clause
\begin{equation*}
\chi^F_X(\langle z,x,\langle a,\tau\rangle\rangle)=\langle a\cup\{z\},\langle |a|,x,\tau\rangle\rangle
\end{equation*}
defines an isomorphism $\chi^F_Z:F_Z\rightarrow D^F_Z$.
\end{lemma}
\begin{proof}
To see that $\chi^F_Z$ has values in $D^F_Z$ we consider an arbitrary $\langle z,x,\langle a,\tau\rangle\rangle\in F_Z$. In view of Definition~\ref{def:coded-prae-dilator-reconstruct} we have $\tau\in H[x]_{|a|}$ and $\supp^{H[x]}_{|a|}(\tau)=|a|$. This yields
\begin{equation*}
\langle |a|,x,\tau\rangle\in F_{|a|+1}\quad\text{and}\quad\supp^F_{|a|+1}(\langle |a|,x,\tau\rangle)=\{|a|\}\cup\supp^{H[x]}_{|a|}(\tau)=|a|+1.
\end{equation*}
The condition $a\lef_Z z$ ensures $|a\cup\{z\}|=|a|+1$. Thus we indeed have
\begin{equation*}
\chi^F_Z(\langle z,x,\langle a,\tau\rangle\rangle)=\langle a\cup\{z\},\langle |a|,x,\tau\rangle\rangle\in D^F_Z.
\end{equation*}
To prove that $\chi^F_Z$ is order preserving we consider an inequality
\begin{equation*}
\langle z_0,x_0,\langle a_0,\tau_0\rangle\rangle<_{F_Z}\langle z_1,x_1,\langle a_1,\tau_1\rangle\rangle.
\end{equation*}
We write $\iota_j:a_j\cup\{z_j\}\hookrightarrow a_0\cup a_1\cup\{z_0,z_1\}=:c$ with $j\in\{0,1\}$ for the inclusions. Furthermore, let $\en_c:|c|\rightarrow c$ and $\en_j:|a_j\cup\{z_j\}|\rightarrow a_j\cup\{z_j\}$ denote the increasing enumerations. As in the previous section, the function $|\iota_j|:|a_j\cup\{z_j\}|\rightarrow|c|$ is determined by the property that it is order preserving and makes the following diagram commute:
\begin{equation*}
\begin{tikzcd}
{|a_j\cup\{z_j\}|} \ar{r}{|\iota_j|}\ar{d}{\en_j} & {|c|} \ar{d}{\en_c}\\
a_j\cup\{z_j\} \ar{r}{\iota_j} & c.
\end{tikzcd}
\end{equation*}
According to Definition~\ref{def:coded-prae-dilator-reconstruct} the desired inequality between the values of $\chi^F_Z$ is equivalent to
\begin{equation*}
F_{|\iota_0|}(\langle|a_0|,x_0,\tau_0\rangle)<_{F_{|c|}}F_{|\iota_1|}(\langle|a_1|,x_1,\tau_1\rangle).
\end{equation*}
By the definition of $F$ the latter amounts to
\begin{equation*}
\langle|\iota_0|(|a_0|),x_0,H[x_0]_{|\iota_0|\restriction|a_0|}(\tau_0)\rangle<_{F_{|c|}}\langle|\iota_1|(|a_1|),x_1,H[x_1]_{|\iota_1|\restriction|a_1|}(\tau_1)\rangle.
\end{equation*}
To establish this inequality we first assume that the given inequality between the arguments of $\chi^F_Z$ holds because of $z_0<_Zz_1$. In view of the condition $a_j\lef_Z z_j$ we have $\en_j(|a_j|)=z_j$ and thus
\begin{equation*}
\en_c\circ|\iota_0|(|a_0|)=\iota_0\circ\en_0(|a_0|)=z_0<_Zz_1=\iota_1\circ\en_1(|a_1|)=\en_c\circ|\iota_1|(|a_1|).
\end{equation*}
This implies $|\iota_0|(|a_0|)<|\iota_1|(|a_1|)$, so that the required inequality holds. A similar argument shows that $z_0=z_1$ implies $|\iota_0|(|a_0|)=|\iota_1|(|a_1|)$. The case where we have $z_0=z_1$ and $x_0<_Xx_1$ is now immediate. Finally assume $z_0=z_1$, $x_0=x_1=:x$ and
\begin{equation*}
\langle a_0,\tau_0\rangle<_{D^{H[x]}_Z}\langle a_1,\tau_1\rangle.
\end{equation*}
It is straightforward to check that the restriction $|\iota_j|\!\restriction\!|a_j|$ makes the following diagram commute, where the vertical arrows are the increasing enumerations:
\begin{equation*}
\begin{tikzcd}
{|a_j|} \ar{r}{|\iota_j|\restriction|a_j|}\ar{d} &[2em] {|a_0\cup a_1|} \ar{d}\\
a_j \arrow[r,hook] & a_0\cup a_1.
\end{tikzcd}
\end{equation*}
In view of Definition~\ref{def:coded-prae-dilator-reconstruct} we can conclude
\begin{equation*}
H[x]_{|\iota_0|\restriction|a_0|}(\tau_0)<_{H_{|a_0\cup a_1|}}H[x]_{|\iota_1|\restriction|a_1|}(\tau_1),
\end{equation*}
which implies the required inequality. To show that $\chi^F_Z$ is surjective we consider an arbitrary $\langle b,\langle N,x,\tau\rangle\rangle\in D^F_Z$. According to Definition~\ref{def:coded-prae-dilator-reconstruct} we must have
\begin{equation*}
|b|=\supp^F_{|b|}(\langle N,x,\tau\rangle)=\{N\}\cup\supp^{H[x]}_N(\tau),
\end{equation*}
which forces $|b|=N+1$ and $\supp^{H[x]}_N(\tau)=N$. In particular $b$ is non-empty, so that we can write $b=a\cup\{z\}$ with $a\lef_Z z$. In view of $|a|=N$ we get $\langle a,\tau\rangle\in D^{H[x]}_Z$ and then $\langle z,x,\langle a,\tau\rangle\rangle\in F_Z$, as well as $\chi^F_Z(\langle z,x,\langle a,\tau\rangle\rangle)=\langle b,\langle N,x,\tau\rangle\rangle$.
\end{proof}

By combining previous results we obtain the following:

\begin{corollary}[$\aca_0$]\label{cor:F-dil}
Consider a well-order~$X$. An $X$-indexed family $\mathcal T$ of $\mathbb N$-trees is progressive along $X$ if, and only if, $F[X,\mathcal T]$ is a normal dilator.
\end{corollary}
\begin{proof}
In view of Proposition~\ref{prop:H-0-captures-prog-new}, Proposition~\ref{prop:F-normal} and the previous lemma it suffices to show that $Z\mapsto F_Z$ preserves well-foundedness if, and only if, $Z\mapsto D^{H[x]}_Z$ preserves well-foundedness for all $x\in X$. If the latter holds, then
\begin{equation*}
Z\times\Sigma_{x\in X^\bot}D^{H[x]}_Z
\end{equation*}
is well-founded for any well-order~$Z$ (recall that $D^{H[\bot]}_Z=\{\langle\emptyset,\star\rangle\}$ consists of a single element, so that it is well-founded in any case). Since $F_Z$ is contained in that order it must be well-founded itself. To establish the other direction we show the following: For any $x\in X$ the order $D^{H[x]}_Z$ can be embedded into $F_{Z^\top}$, where $Z^\top=Z\cup\{\top\}$ extends $Z$ by a new maximal element. Let us write $\iota:Z\hookrightarrow Z^\top$ for the inclusion. Definition~\ref{def:coded-prae-dilator-reconstruct} and Proposition~\ref{prop:reconstruct-class-sized-dil} tell us that $D^{H[x]}_\iota(\langle a,\tau\rangle)=\langle[\iota]^{<\omega}(a),\tau\rangle=\langle a,\tau\rangle$ defines an embedding of $D^{H[x]}_Z$ into $D^{H[x]}_{Z^\top}$. Since any $a\in[Z]^{<\omega}$ satisfies $a\lef_{Z^\top}\top$ we see that $\langle a,\tau\rangle\mapsto\langle\top,x,\langle a,\tau\rangle\rangle$ is the desired embedding of $D^{H[x]}_Z$ into $F_{Z^\top}$.
\end{proof}

With the previous result we have completed our reconstruction of the functions $h$ and $f$ that appear in the informal argument from the introduction. The latter proceeds by considering an upper derivative $g$ of $f$. It then invokes induction on~$\gamma$ to show that each tree $\mathcal T_\gamma$ can be embedded into the ordinal~$g(\gamma+1)\leq g(\alpha)$. To avoid the use of transfinite induction (or recursion) we now give a particularly careful construction of embeddings. This construction will involve the finite orders
\begin{equation*}
 \Sigma_x\mathcal T\cap k=\{j<k\,|\,\text{$j$ is (the code of) an element of $\Sigma_x\mathcal T$}\},
\end{equation*}
with the same order relation as on $\Sigma_x\mathcal T$. Let us also recall that $\mathcal T_x$ carries the Kleene-Brouwer order $\lkb$ with respect to~$\mathbb N$.

\begin{proposition}[$\rca_0$]\label{prop:emd-finite}
 Consider a well-order~$X$ and an $X$-indexed family $\mathcal T$ of $\mathbb N$-trees. There is a function $E:\Sigma\mathcal T\rightarrow\mathbb N$ such that the following holds for any element $x\in X$, any $\sigma,\sigma_1,\sigma_2\in\mathcal T_x$ and any order~$Z$:
 \begin{enumerate}[label=(\roman*)]
  \item Given a (finite) embedding $e:\Sigma_x\mathcal T\cap\len(\sigma)\rightarrow Z$, we have
  \begin{equation*}
   \langle\rng(e),E(\langle x,\sigma\rangle)\rangle\in D^{H[x]}_Z.
  \end{equation*}
  \item If $e_1:\Sigma_x\mathcal T\cap\len(\sigma_1)\rightarrow Z$ and $e_2:\Sigma_x\mathcal T\cap\len(\sigma_2)\rightarrow Z$ coincide on the intersection of their domains, then we have
  \begin{equation*}
   \sigma_1\lkb\sigma_2\quad\Rightarrow\quad \langle\rng(e_1),E(\langle x,\sigma_1\rangle)\rangle<_{D^{H[x]}_Z}\langle\rng(e_2),E(\langle x,\sigma_2\rangle)\rangle.
  \end{equation*}
 \end{enumerate}
\end{proposition}
\begin{proof}
 We begin by defining $E(\langle x,\sigma\rangle)$ for given $x\in X$ and $\sigma=\langle s_0,\dots,s_{k-1}\rangle\in\mathcal T_x$. For $j\in\Sigma_x\mathcal T\cap k$ we define $n_j<|\Sigma_x\mathcal T\cap k|$ by stipulating that $j$ is the $n_j$-th element of $\Sigma_x\mathcal T\cap k$. For all $j<k$ outside of $\Sigma_x\mathcal T$ we set $n_j=\bot$. Now we put
 \begin{equation*}
  E(\langle x,\sigma\rangle)=\langle\langle n_0,s_0\rangle,\dots,\langle n_{k-1},s_{k-1}\rangle\rangle.
 \end{equation*}
 To establish~(i) we first need $E(\langle x,\sigma\rangle)\in H[x]_{|\rng(e)|}$. Since $e$ is injective its range has the same cardinality as $\Sigma_x\mathcal T\cap k$ (continuing the notation from above, so that~\mbox{$k=\len(\sigma)$}). Thus we do have $n_j\in|\rng(e)|^\bot$ for all~$j<k$. Clause~(i) of Definition~\ref{def:H-0} is satisfied by construction. Clause~(ii) says nothing but $\sigma\in\mathcal T_x$. To complete the verification of~(i) we need
 \begin{equation*}
  \supp^{H[x]}_{|\rng(e)|}(E(\langle x,\sigma\rangle))=|\rng(e)|.
 \end{equation*}
 For the crucial inclusion $\supseteq$ it suffices to observe that any $n\in|\rng(e)|=|\Sigma_x\mathcal T\cap k|$ is the position of some $j\in\Sigma_x\mathcal T\cap k$. To prove property~(ii) we compose with the order isomorphism from Lemma~\ref{lem:H-0-alternative}. If $j$ is the $n_j$-th element of~$\Sigma_x\mathcal T\cap k$, then $e(j)$ is the $n_j$-th element of $\rng(e)$. Thus (still with the same notation as above) we see that $\langle\rng(e),E(\langle x,\sigma\rangle)\rangle$ corresponds to
 \begin{equation*}
  \langle\langle e_\bot(0),s_0\rangle,\dots,\langle e_\bot(k-1),s_{k-1}\rangle\rangle\in H[x]_Z,
 \end{equation*}
 where $e_\bot:k\rightarrow Z^\bot$ extends $e$ by the values $e_\bot(j)=\bot$ for $j\notin\Sigma_x\mathcal T$. With this description it is straightforward to check property~(ii): Assume that we have
 \begin{equation*}
  \sigma_1=\langle s_0,\dots,s_{k-1}\rangle\lkb\langle s'_0,\dots,s'_{l-1}\rangle=\sigma_2
 \end{equation*}
 and that $e_1:\Sigma_x\mathcal T\cap k\rightarrow Z$ and $e_2:\Sigma_x\mathcal T\cap l\rightarrow Z$ coincide below $\min\{k,l\}$. Since $H[x]_Z$ carries the Kleene-Brouwer order with respect to $Z^\bot\times\mathbb N$ we get
 \begin{multline*}
  \langle\langle (e_1)_\bot(0),s_0\rangle,\dots,\langle (e_1)_\bot(k-1),s_{k-1}\rangle\rangle<_{H[x]_Z}\\
  \langle\langle (e_2)_\bot(0),s'_0\rangle,\dots,\langle (e_2)_\bot(l-1),s'_{l-1}\rangle.
 \end{multline*}
 Up to the isomorphism $H[x]_Z\cong D^{H[x]}_Z$ this is just as required.
\end{proof}

Using the previous result, we now show that the embedding $J$ from Summary~\ref{sum:deduce-Pi11-bi} can be constructed in a weak meta-theory.

\begin{theorem}[$\rca_0$]\label{thm:embed-Sigma-T}
Consider a well-order~$X$ and an $X$-indexed family $\mathcal T$ of $\mathbb N$-trees. Assume that $G$ and $\xi:F\circ G\Rightarrow G$ form an upper derivative of the normal prae-dilator $F=F[X,\mathcal T]$. Then $\Sigma\mathcal T$ can be embedded into the order $D^G_X$.
\end{theorem}
\begin{proof}
As a preparation we specify two functions that are implicit in the given data: By combining Proposition~\ref{prop:compose-dils-rca}, Lemma~\ref{lem:morph-dilators-extend} and Lemma~\ref{lem:F-alternative} we get an embedding
\begin{equation*}
\xi^F:=D^\xi_X\circ\zeta^{F,G}_X\circ\chi^F_{D^G_X}:F_{D^G_X}\rightarrow D^G_X.
\end{equation*}
From Definitions~\ref{def:coded-normal-dil},~\ref{def:extend-normal-transfos} and~\ref{def:upper-derivative} we know that $G=(G,\mu^G)$ must be a normal prae-dilator and does, as such, give rise to an order preserving function
\begin{equation*}
 D^{\mu^G}_X:X\rightarrow D^G_X.
\end{equation*}
We now show that the desired embedding
\begin{equation*}
J:\Sigma\mathcal T\rightarrow D^G_X
\end{equation*}
can be constructed by a finitary course-of-values recursion. For this purpose we assume that the code of any $\langle x,\sigma\rangle\in\Sigma\mathcal T$ is at least as big as the length of the sequence $\sigma$. The value $J(\langle x,\sigma\rangle)$ may then depend on the finite function
\begin{equation*}
 e_{x,\sigma}:=J\!\restriction\!(\Sigma_x\mathcal T\cap\len(\sigma)):\Sigma_x\mathcal T\cap\len(\sigma)\rightarrow D^G_X.
\end{equation*}
After describing the construction of $J$ we will set up an induction which ensures that $e_{x,\sigma}$ is an embedding. When this is the case Proposition~\ref{prop:emd-finite} yields an element
\begin{equation*}
 J^0(\langle x,\sigma\rangle):=\langle\rng(e_{x,\sigma}),E(\langle x,\sigma\rangle)\rangle\in D^{H[x]}(D^G_X).
\end{equation*}
In view of Definition~\ref{def:F-alternative} we can now state the recursive clause for $J$ as
\begin{equation*}
 J(\langle x,\sigma\rangle)=\xi^F(\langle D^{\mu^G}_X(x),x,J^0(\langle x,\sigma\rangle)\rangle).
\end{equation*}
To conclude the proof we show the following by simultaneous induction on $j$:
\begin{enumerate}[label=(\roman*)]
\item If $j$ codes an element of $\Sigma\mathcal T$, then we have $J(j)\in D^G_X$.
\item If $j_1,j_2\leq j$ code elements of $\Sigma\mathcal T$, then we have
\begin{equation*}
 j_1<_{\Sigma\mathcal T}j_2\quad\Rightarrow\quad J(j_1)<_{D^G_X}J(j_2).
 \end{equation*}
\item If $j$ codes an element of $\Sigma_x\mathcal T$, then we have $J(j)<_{D^G_X} D^{\mu^G}_X(x)$.
\end{enumerate}
To verify the induction step for~(i) we write $j=\langle x,\sigma\rangle$. Parts~(i) and~(ii) of the induction hypothesis guarantee that $e_{x,\sigma}$ is an embedding with values in~$D^G_X$, as promised above. We have seen that this yields $J^0(\langle x,\sigma\rangle)\in D^{H[x]}(D^G_X)$. In view of Definition~\ref{def:F-alternative} we also need
\begin{equation*}
 \rng(e_{x,\sigma})\lef_{D^G_X} D^{\mu^G}_X(x).
\end{equation*}
This holds by part~(iii) of the induction hypothesis. To establish the induction step for~(ii) we consider an inequality
\begin{equation*}
 j_1=\langle x_1,\sigma_1\rangle<_{\Sigma\mathcal T}\langle x_2,\sigma_2\rangle=j_2.
\end{equation*}
If we have $x_1<_Xx_2$, then we get $D^{\mu^G}_X(x_1)<_{D^G_X}D^{\mu^G}_X(x_2)$ and thus
\begin{equation*}
 \langle D^{\mu^G}_X(x_1),x_1,J^0(\langle x_1,\sigma_1\rangle)\rangle<_{F_{D^G_X}}\langle D^{\mu^G}_X(x_2),x_2,J^0(\langle x_2,\sigma_2\rangle)\rangle.
\end{equation*}
As $\xi^F$ is order preserving this implies $J(j_1)<_{D^G_X}J(j_2)$. Now assume that $j_1<_{\Sigma\mathcal T}j_2$ holds because we have $x_1=x_2=:x$ and $\sigma_1\lkb\sigma_2$ (recall that $\mathcal T_x$ carries the usual Kleene-Brouwer order). Since $e_{x,\sigma_1}$ and $e_{x,\sigma_2}$ are restrictions of the same function, they coincide on the intersection of their domains. Thus Proposition~\ref{prop:emd-finite} yields
\begin{equation*}
 J^0(\langle x_1,\sigma_1\rangle)<_{D^{H[x]}(D^G_X)}J^0(\langle x_2,\sigma_2\rangle),
\end{equation*}
which again implies $J(j_1)<_{D^G_X}J(j_2)$. Finally we prove the induction step for~(iii). As a preparation we recall that $D^{H[\bot]}(D^G_X)$ consists of the single element $\langle\emptyset,\star\rangle$. In view of Lemma~\ref{lem:F-alternative}, Definition~\ref{def:dil-F} and Definition~\ref{def:extend-normal-transfos} we have
\begin{multline*}
 \chi^F_{D^G_X}(\langle D^{\mu^G}_X(x),\bot,\langle\emptyset,\star\rangle\rangle)=\langle\{D^{\mu^G}_X(x)\},\langle 0,\bot,\star\rangle\rangle=\\
 =\langle\{D^{\mu^G}_X(x)\},\mu^F_1(0)\rangle=D^{\mu^F}_{D^G_X}\circ D^{\mu^G}_X(x).
\end{multline*}
Together with Corollary~\ref{cor:deriv-normal-witness} we get
\begin{equation*}
 \xi^F(\langle D^{\mu^G}_X(x),\bot,\langle\emptyset,\star\rangle\rangle)=D^\xi_X\circ\zeta^{F,G}_X\circ D^{\mu^F}_{D^G_X}\circ D^{\mu^G}_X(x)=D^{\mu^G}_X(x).
\end{equation*}
To deduce~(iii) we observe that any $j\in\Sigma_x\mathcal T$ has the form $j=\langle y,\sigma\rangle$ with $y<_Xx$. The latter implies
\begin{equation*}
 \langle D^{\mu^G}_X(y),y,J^0(\langle y,\sigma\rangle)\rangle<_{F_{D^G_X}}\langle D^{\mu^G}_X(x),\bot,\langle\emptyset,\star\rangle\rangle.
\end{equation*}
By the above this yields
\begin{equation*}
 J(j)<_{D^G_X}\xi^F(\langle D^{\mu^G}_X(x),\bot,\langle\emptyset,\star\rangle\rangle)=D^{\mu^G}_X(x),
\end{equation*}
as required.
\end{proof}

The following result completes our reconstruction of the informal argument and establishes the implication (2)$\Rightarrow$(3) that we have discussed in the introduction. The notions that appear in statement~(2) have been made precise in Section~\ref{sect:normal-dils-so}.

\begin{corollary}\label{cor:deriv-implies-BI}
 For each $\Pi^1_1$-formula $\varphi(x)$ (possibly with further free variables) the following is provable in $\aca_0$: If every normal dilator $F$ has an upper derivative $(G,\xi)$ such that $G$ is a dilator, then $\varphi$ satisfies bar induction, i.\,e.~we have
 \begin{equation*}
  \forall_{x\in X}(\forall_{y<_Xx}\varphi(y)\rightarrow\varphi(x))\rightarrow\forall_{x\in X}\varphi(x)
 \end{equation*}
 for any well-order $X=(X,<_X)$.
\end{corollary}
\begin{proof}
 By the Kleene normal form theorem (see~\cite[Lemma~V.1.4]{simpson09}) there is a bounded arithmetical formula $\theta(\sigma,x)$ such that $\aca_0$ proves
 \begin{equation*}
  \varphi(x)\leftrightarrow\forall_f\exists_n\theta(f[n],x).
 \end{equation*}
 Here the universal quantifier ranges over all functions $f:\mathbb N\rightarrow\mathbb N$. Let us recall that $f[n]=\langle f(0),\dots,f(n-1)\rangle$ denotes the sequence that contains the first $n$ values of such a function. Given a sequence $\sigma=\langle \sigma_0,\dots,\sigma_{\len(\sigma)-1}\rangle$ and a number $n\leq\len(\sigma)$, we similarly write $\sigma[n]=\langle \sigma_0,\dots,\sigma_{n-1}\rangle$. We can now define an $X$-indexed family $\mathcal T=\{\langle x,\sigma\rangle\,|\,x\in X\text{ and }\sigma\in\mathcal T_x\}$ of $\mathbb N$-trees by setting
 \begin{equation*}
  \mathcal T_x=\{\sigma\in\mathbb N^{<\omega}\,|\,\forall_{n\leq\len(\sigma)}\neg\theta(\sigma[n],x)\}
 \end{equation*}
 for every $x\in X$. Observe that $\forall_n\neg\theta(f[n],x)$ is equivalent to the assertion that $f$ is a branch in $\mathcal T_x$. Thus we have
 \begin{equation*}
  \varphi(x)\leftrightarrow\text{``$\mathcal T_x$ is well-founded''},
 \end{equation*}
 where $\mathcal T_x$ carries the usual Kleene-Brouwer order with respect to~$\mathbb N$. Let us now assume that the premise of the desired induction statement holds. Then $\mathcal T$ is progressive along $X$, using the terminology that we have introduced at the beginning of this section. Consider the prae-dilator $F=F[X,\mathcal T]$ that is constructed according to~Definition~\ref{def:dil-F}. From Corollary~\ref{cor:F-dil} we learn that $F$ is a normal dilator. By the assumption of the present corollary we get a dilator $G$ and a natural transformation $\xi:F\circ G\rightarrow G$ that form an upper derivative of $F$. Theorem~\ref{thm:embed-Sigma-T} tells us that $\Sigma\mathcal T$ can be embedded into the order $D^G_X$. The latter is well-founded, because $X$ is a well-order and $G$ is a dilator. Hence $\Sigma\mathcal T$ is well-founded as well. It follows that $\mathcal T_x$ is well-founded for every $x\in X$, which yields the conclusion $\forall_{x\in X}\varphi(x)$ of the desired induction statement.
\end{proof}

\section{Constructing the derivative}\label{sect:constructin-derivative}

In the present section we show how to construct a derivative $\partial T$ of a given normal prae-dilator $T$. We will see that $\rca_0$ proves the existence of $\partial T$, as well as the fact that it is a derivative. As a consequence we obtain the implication (1)$\Rightarrow$(2) from the introduction. What $\rca_0$ cannot show is that~$\partial T$ is a dilator (i.\,e.~that $X\mapsto D^{\partial T}_X$ preserves well-foundedness) whenever $T$ is one: Due to Corollary~\ref{cor:deriv-implies-BI} this statement implies $\Pi^1_1$-bar induction. The converse implication, which amounts to (3)$\Rightarrow$(1) from the introduction, will be established in Section~\ref{sect:bi-deriv-wf}.

The construction of $\partial T$ can also be exploited to establish general results about derivatives. This relies on the fact that derivatives are essentially unique, as observed after Definition~\ref{def:derivative}. We will use this approach to show that the assumptions of Theorem~\ref{thm:equalizer-to-derivative} are automatic when $(S,\xi)$ is a derivative of $T$.

As mentioned in the introduction, a categorical construction of derivatives has already been given by Aczel~\cite{aczel-phd,aczel-normal-functors}. In the following we give a rather informal presentation of his approach in the terminology of the present paper (in particular we are rather liberal about the distinction between coded and class-sized dilators, cf.~the discussion after Definition~\ref{def:coded-dilator}). Given a normal dilator $T=(T,\mu^T)$ and an order~$X$, Aczel's idea was to define the value $\partial T_X$ as the direct limit of the diagram
  \begin{equation*}
  \begin{tikzcd}[column sep = large]
 X\arrow[r,"\mu^T_X"] & T_X\arrow[r,"T(\mu^T_X)"] & T^2_X:=T(T_X)\arrow[r,"T^2(\mu^T_X):=T(T(\mu^T_X))"] &[2cm] \cdots.
 \end{tikzcd}
\end{equation*}
As a direct limit, $\partial T_X$ comes with compatible embeddings $j^n_X:T^n_X\rightarrow\partial T_X$. By the universal property the functions
\begin{equation*}
 T(j^n_X)\circ T^n(\mu^T_X):T^n_X\rightarrow T(\partial T_X)
\end{equation*}
glue to an embedding of $\partial T_X$ into $T(\partial T_X)$. The latter is an isomorphism since $T$ preserves direct limits. Thus $\partial T_X$ is a fixed-point of $T$, as one would expect if $\partial T$ is to be a derivative. Furthermore, Aczel could show that $\partial T$ preserves well-foundedness if $T$ does. This is a non-trivial matter, since well-foundedness is not preserved under direct limits in general. The proof that it is preserved under the specific limit constructed above makes crucial use of the assumption that $T$ preserves initial segments (cf.~the discussion before Lemma~\ref{lem:range-dil-support}). Finally, Aczel has shown that $\alpha\mapsto\otp(\partial T_\alpha)$ is the derivative of the normal function $\alpha\mapsto\otp(T_\alpha)$. Let us mention that he did not give an explicit characterization of derivatives on the level of functors, i.\,e.~he did not formulate an analogue of Definition~\ref{def:derivative}.

In order to recover Aczel's construction in $\rca_0$ we need to approach the direct limit in a particularly finitistic way. Our idea is to represent $\partial T_X$ by a system of~terms. To~see how this works, recall that we want to ensure the existence of an iso\-morphism~$\xi_X:T(\partial T_X)\rightarrow\partial T_X$. In view of Lemma~\ref{lem:class-sized-restrict} (and the discussion after Definition~\ref{def:coded-dilator}) any element of $T(\partial T_X)$ corresponds to a pair $\langle a,\sigma\rangle\in D^T(\partial T_X)$, where $a\subseteq\partial T_X$ is finite and $\sigma\in T_{|a|}$ satisfies $\supp^T_{|a|}(\sigma)=|a|$. Assuming that the elements of $a$ are already represented by terms, we can add a term $\xi\langle a,\sigma\rangle\in\partial T_X$ that represents the value of $\langle a,\sigma\rangle$ under $\xi_X$. To make this idea precise we switch back to the rigorous framework of coded prae-dilators, as introduced in Section~\ref{sect:normal-dils-so}. In particular we want to construct $\partial T$ as a coded prae-dilator, which leads us to focus on the values $\partial T_n$ for the finite orders $n=\{0,\dots,n-1\}$. Let us first specify the underlying set of the order $\partial T_n$:

\begin{definition}[$\rca_0$]\label{def:derivative-term-system}
 Consider a normal prae-dilator $T=(T,\supp^T,\mu^T)$. For each number $n$ we define a term system $\partial T_n$ by the following inductive clauses:
 \begin{enumerate}[label=(\roman*)]
  \item We have a term $\mu_m\in\partial T_n$ for every number $m<n$.
  \item Given a finite set $a\subseteq\partial T_n$ of terms and a $\sigma\in T_{|a|}$ with $\supp^T_{|a|}(\sigma)=|a|$, we get a term $\xi\langle a,\sigma\rangle\in\partial T_n$, provided that the following holds: If we have $a=\{\mu_m\}$ for some $m<n$, then $\sigma$ must be different from $\mu^T_1(0)\in T_1$.
 \end{enumerate}
\end{definition}

Note that the term systems $\partial T_n$ are uniformly computable (with respect to~$n$), so that $\rca_0$ proves the existence of the set
\begin{equation*}
 \{\langle n,s\rangle\,|\,s\in\partial T_n\}.
\end{equation*}
This is crucial if we want to extend $\partial T$ into a coded prae-dilator (cf.~the discussion after Definition~\ref{def:coded-prae-dilator}). In order to understand the proviso in clause~(ii) one should think of $\mu_m$ as a notation for~$f'(m)$, where $f$ is the normal function induced by $T$ and $f'$ is its derivative. In the proof of Proposition~\ref{prop:normal-dil-fct} we have seen that $D^{\mu^T}$ amounts to an internal version of the function $f$. Together with Definition~\ref{def:extend-normal-transfos} we see that $\langle\{\mu_m\},\mu^T_1(0)\rangle=D^{\mu^T}_{\partial T_n}(\mu_m)$ corresponds to $f(f'(m))$. Since the latter is equal to $f'(m)$ the terms $\xi\langle\{\mu_m\},\mu^T_1(0)\rangle$ and $\mu_m$ would represent the same ordinal. To keep our notations unique, the first of these terms has been excluded in clause~(ii). A formal version of this intuitive explanation will play a role in the proof of Theorem~\ref{thm:partial-derivative}. The following notion of term length will be used to define the order on $\partial T_n$:

\begin{definition}[$\rca_0$]\label{def:deriv-term-length}
 For each $n$ we define a length function $L^{\partial T}_n:\partial T_n\rightarrow\mathbb N$ by recursion over the build-up of terms, setting
 \begin{equation*}
 L^{\partial T}_n(s)=\begin{cases}
                      \ulcorner s\urcorner & \text{if $s=\mu_m$},\\
                      \max\{\ulcorner s\urcorner,1+\textstyle\sum_{t\in a}2\cdot L^{\partial T}_n(t)\} & \text{if $s=\xi\langle a,\sigma\rangle$},
                     \end{cases}
 \end{equation*}
 where $\ulcorner s\urcorner$ denotes the code (G\"odel number) of the term $s$.
\end{definition}

Note that $\ulcorner s\urcorner$ coincides with $s$ if Definition~\ref{def:derivative-term-system} is already arithmetized. The role of the G\"odel numbers is to justify certain applications of induction and recursion over the length of terms: In view of $\ulcorner s\urcorner\leq L^{\partial T}_n(s)$ a quantifier of the form
\begin{equation*}
 \forall_{s\in\partial T_n}(L^{\partial T}_n(s)\leq l\rightarrow\dots)
\end{equation*}
is bounded. Thus such a quantifier does not lead out of the $\Sigma^0_1$-formulas, for which induction is available in $\rca_0$. To construct a binary relation $<_{\partial T_n}$ on $\partial T_n$, the following definition decides $s<_{\partial T_n}t$ by recursion on $L^{\partial T}_n(s)+L^{\partial T}_n(t)$. In case we have $s=\xi\langle a,\sigma\rangle$ and $t=\xi\langle b,\tau\rangle$ we can assume that the restriction of $<_{\partial T_n}$ to $a\cup b$ is already determined (note that $r\in a$ yields $2\cdot L^{\partial T}_n(r)<L^{\partial T}_n(s)$, so that we can decide $r<_{\partial T_n}r$). In particular we can check whether $<_{\partial T_n}$ is a linear order on the finite set $a\cup b$. If it is, then we may refer to the functions $|\iota_a^{a\cup b}|$ and $|\iota_b^{a\cup b}|$ from Definition~\ref{def:coded-prae-dilator-reconstruct}. More explicitly, we can write $\en_a:|a|\rightarrow a$ and $\en_{a\cup b}:|a\cup b|\rightarrow a\cup b$ for the unique increasing enumerations with respect to~$<_{\partial T_n}$. Then the function $|\iota_a^{a\cup b}|:|a|\rightarrow|a\cup b|$ is characterized by the fact that it satisfies $\en_{a\cup b}\circ|\iota_a^{a\cup b}|=\iota_a^{a\cup b}\circ\en_a$, where $\iota_a^{a\cup b}:a\hookrightarrow a\cup b$ is the inclusion. Before Lemma~\ref{lem:range-dil-support} we have seen that a linear order $<_X$ on a set $X$ induces relations $\lef_X$ and $\leqf_X$ between finite subsets of $X$. In the following we use $\lef_{\partial T_n}$ and $\leqf_{\partial T_n}$ as abbreviations, without assuming that $<_{\partial T_n}$ is linear on the relevant parts of $X$. In particular we have
\begin{equation*}
 s\leqf_{\partial T_n} b\qquad\Leftrightarrow\qquad\exists_{r\in b}\,s\leq_{\partial T_n}r
\end{equation*}
for any element $s\in\partial T_n$ and any finite subset $b\subseteq\partial T_n$. Note that $s\leq_{\partial T_n}r$ abbreviates ${s<_{\partial T_n}r}\lor {s=r}$, where the second disjunct refers to the equality of terms. We can now state the definition of the desired order relation:

\begin{definition}[$\rca_0$]\label{def:derivative-order}
 For each $n$ we define a binary relation $<_{\partial T_n}$ on~$\partial T_n$. Invoking recursion on $L^{\partial T}_n(s)+L^{\partial T}_n(t)$, we stipulate that $s<_{\partial T_n}t$ holds if, and only if, one of the following is satisfied:
 \begin{enumerate}[label=(\roman*)]
  \item We have $s=\mu_m$ and
  \begin{itemize}
   \item either $t=\mu_k$ and $m<k$,
   \item or $t=\xi\langle b,\tau\rangle$ and $s\leqf_{\partial T_n} b$.
  \end{itemize}
  \item We have $s=\xi\langle a,\sigma\rangle$ and
  \begin{itemize}
   \item either $t=\mu_k$ and $a\lef_{\partial T_n}t$,
   \item or we have $t=\xi\langle b,\tau\rangle$, the restriction of $<_{\partial T_n}$ to $a\cup b$ is linear, and we have $T_{|\iota_a^{a\cup b}|}(\sigma)<_{T_{|a\cup b|}}T_{|\iota_b^{a\cup b}|}(\tau)$.
  \end{itemize}
 \end{enumerate}
\end{definition}

To show that $<_{\partial T_n}$ is a linear order we will need the following auxiliary result:

\begin{lemma}[$\rca_0$]\label{lem:supports-order-normal}
 If $T$ is a normal prae-dilator, then we have
 \begin{equation*}
  \sigma\leq_{T_k}\tau\quad\Rightarrow\quad\supp^T_k(\sigma)\leqf\supp^T_k(\tau)
 \end{equation*}
 for any number $k$ and arbitrary elements $\sigma,\tau\in T_k$.
\end{lemma}
\begin{proof}
 If the conclusion of the implication is false, then we have $\supp^T_k(\tau)\lef m$ for some $m\in\supp^T_k(\sigma)$. Note that $\supp^T_k(\sigma)\lef m$ must fail. In view of Definition~\ref{def:coded-normal-dil} we obtain $\tau<_{T_k}\mu^T_k(m)\leq_{T_k}\sigma$, so that the premise of our implication is false.
\end{proof}

We can now establish the expected fact:

\begin{lemma}[$\rca_0$]\label{lem:deriv-linear}
 Given a normal prae-dilator $T$ and any number $n$, the relation $<_{\partial T_n}$ is a linear order on the term system $\partial T_n$.
\end{lemma}
\begin{proof}
It is straightforward to see that $s<_{\partial T_n}s$ must fail for every $s$, based on the fact that the linear order $<_{T_k}$ is antisymmetric for any number~$k$. We now show
\begin{gather*}
 s<_{\partial T_n}t\lor s=t\lor t<_{\partial T_n} s,\\
 r<_{\partial T_n}s\land s<_{\partial T_n}t\rightarrow r<_{\partial T_n}t
\end{gather*}
by simultaneous induction on $L^{\partial T}_n(s)+L^{\partial T}_n(t)$ resp.~$L^{\partial T}_n(r)+L^{\partial T}_n(s)+L^{\partial T}_n(t)$. Trichotomy is immediate if we have $s=\mu_m$ and $t=\mu_k$ with $m,k<n$. If we have $s=\mu_m$ and $t=\xi\langle b,\tau\rangle$, then the induction hypothesis yields $s\leqf_{\partial T_n}b$ or~$b\lef_{\partial T_n}s$. In the first case we get $s<_{\partial T_n}t$ while the second leads to $t<_{\partial T_n}s$. Now assume that we have $s=\xi\langle a,\sigma\rangle$ and $t=\xi\langle b,\tau\rangle$. The simultaneous induction hypothesis ensures that $<_{\partial T_n}$ is linear on $a\cup b$ (note that $s'<_{\partial T_n}t'\land t'<_{\partial T_n}s'\rightarrow s'<_{\partial T_n}s'$ is covered, due to the factor $2$ in the definition of $L^{\partial T}_n$). It is easy to conclude unless we have $T_{|\iota_a^{a\cup b}|}(\sigma)=T_{|\iota_b^{a\cup b}|}(\tau)$. In this case the naturality of $\supp^T$ yields
\begin{alignat*}{3}
 [|\iota_a^{a\cup b}|]^{<\omega}(|a|)&=[|\iota_a^{a\cup b}|]^{<\omega}(\supp^T_{|a|}(\sigma))&&=\supp^T_{|a\cup b|}(T_{|\iota_a^{a\cup b}|}(\sigma))&&={}\\
 {}&=\supp^T_{|a\cup b|}(T_{|\iota_b^{a\cup b}|}(\tau))&&=[|\iota_b^{a\cup b}|]^{<\omega}(\supp^T_{|b|}(\tau))&&=[|\iota_b^{a\cup b}|]^{<\omega}(|b|).
\end{alignat*}
Composing both sides with $[\en_{a\cup b}]^{<\omega}$ we get $a=b$. Then $|\iota_a^{a\cup b}|$ and $|\iota_b^{a\cup b}|$ must be the identity on $|a|=|a\cup b|=|b|$. As $T$ is functorial we get $\sigma=\tau$ and hence~$s=t$. To establish transitivity one needs to distinguish several cases according to the form of the terms $r,s$ and $t$. In the first interesting case we have $r=\mu_m$, $s=\xi\langle a,\sigma\rangle$ and $t=\xi\langle b,\tau\rangle$. Invoking the previous lemma we see that $s<_{\partial T_n}t$ implies
\begin{equation*}
 [|\iota_a^{a\cup b}|]^{<\omega}(|a|)=\supp^T_{|a\cup b|}(T_{|\iota_a^{a\cup b}|}(\sigma))\leqf\supp^T_{|a\cup b|}(T_{|\iota_b^{a\cup b}|}(\tau))=[|\iota_b^{a\cup b}|]^{<\omega}(|b|).
\end{equation*}
Again we compose both sides with $[\en_{a\cup b}]^{<\omega}$, to get $a\leqf_{\partial T_n}b$. In view of $r<_{\partial T_n}s$ we have $\mu_m\leqf_{\partial T_n}a$. Using the induction hypothesis we can infer $\mu_m\leqf_{\partial T_n}b$ and thus $r<_{\partial T_n}t$. Let us now consider $r=\xi\langle a,\sigma\rangle$, $s=\mu_m$ and~$t=\xi\langle b,\tau\rangle$. In this situation $r<_{\partial T_n}s<_{\partial T_n}t$ amounts to $a\lef_{\partial T_n}s\leqf_{\partial T_n}b$, which implies that $b\leqf_{\partial T_n}a$ must~fail. Similarly to the previous case we can conclude
\begin{equation*}
 \supp^T_{|a\cup b|}(T_{|\iota_b^{a\cup b}|}(\tau))\not\leqf\supp^T_{|a\cup b|}(T_{|\iota_a^{a\cup b}|}(\sigma)).
\end{equation*}
Note that we can refer to $|\iota_a^{a\cup b}|$ and $|\iota_b^{a\cup b}|$, since the simultaneous induction hypothesis ensures that $<_{\partial T_n}$ is linear on $a\cup b$. Using the previous lemma and trichotomy for $<_{T_{|a\cup b|}}$ we obtain $T_{|\iota_a^{a\cup b}|}(\sigma)<_{T_{|a\cup b|}}T_{|\iota_b^{a\cup b}|}(\tau)$ and hence $r<_{\partial T_n}t$. To establish transitivity for $r=\xi\langle a,\sigma\rangle$, $s=\xi\langle b,\tau\rangle$ and $t=\xi\langle c,\rho\rangle$ it suffices to considers the inclusions into $a\cup b\cup c$ and to use transitivity for $<_{T_{|a\cup b\cup c|}}$.
\end{proof}

We will see that the following turns $n\mapsto\partial T_n$ into a functor:

\begin{definition}[$\rca_0$]\label{def:deriv-functor}
 Given a strictly increasing function $f:n\rightarrow l$, we define a function $\partial T_f:\partial T_n\rightarrow\partial T_l$ by recursion over the build-up of terms, setting
 \begin{align*}
  \partial T_f(\mu_m)&=\mu_{f(m)},\\
  \partial T_f(\xi\langle a,\sigma\rangle)&=\xi\langle [\partial T_f]^{<\omega}(a),\sigma\rangle.
 \end{align*}
\end{definition}

The fact that $\partial T_f$ has values in $\partial T_l$ is established as part of the following proof:

\begin{lemma}[$\rca_0$]\label{lem:deriv-embedding}
 If $f:n\rightarrow l$ is strictly increasing, then $\partial T_f:\partial T_n\rightarrow\partial T_l$ is an order embedding.
\end{lemma}
\begin{proof}
 By simultaneous induction on $L^{\partial T}_n(r)$ resp.~$L^{\partial T}_n(s)+L^{\partial T}_n(t)$ one can verify
 \begin{gather*}
  r\in\partial T_n\rightarrow\partial T_f(r)\in\partial T_l,\\
  s<_{\partial T_n}t\rightarrow\partial T_f(s)<_{\partial T_l}\partial T_f(t).
 \end{gather*}
 Let us consider the first claim for $r=\xi\langle a,\sigma\rangle$: The simultaneous induction hypothesis implies that $\partial T_f$ is order preserving and hence injective on $a$. In particular we have $|[\partial T_f]^{<\omega}(a)|=|a|$. Furthermore, it is easy to see that $[\partial T_f]^{<\omega}(a)=\{\mu_k\}$ implies $a=\{\mu_m\}$ with $k=f(m)$. Invoking Definition~\ref{def:derivative-term-system} we can now conclude that $r\in\partial T_n$ implies $\partial T_f(r)=\xi\langle [\partial T_f]^{<\omega}(a),\sigma\rangle\in\partial T_l$. To verify that $\partial T_f$ is order preserving we distinguish cases according to the form of $s$ and $t$. In the first interesting case we have $s=\mu_m<_{\partial T_n}\xi\langle b,\tau\rangle=t$ because of $s\leqf_{\partial T_n}b$. By the induction hypothesis we obtain $\partial T_f(s)\leqf_{\partial T_k}[\partial T_f]^{<\omega}(b)$ and hence
 \begin{equation*}
  \partial T_f(s)=\mu_{f(m)}<_{\partial T_k}\xi\langle [\partial T_f]^{<\omega}(b),\tau\rangle=\partial T_f(t).
 \end{equation*}
 Let us also consider the case where $s=\xi\langle a,\sigma\rangle<_{\partial T_k}\xi\langle b,\tau\rangle=t$ holds because we have $T_{|\iota_a^{a\cup b}|}(\sigma)<_{T_{|a\cup b|}}T_{|\iota_b^{a\cup b}|}(\tau)$. To infer $\partial T_f(s)<_{\partial T_k}\partial T_f(t)$ it suffices to show
 \begin{equation*}
  \left|\iota_{[\partial T_f]^{<\omega}(a)}^{[\partial T_f]^{<\omega}(a\cup b)}\right|=|\iota_a^{a\cup b}|\qquad\text{and}\qquad\left|\iota_{[\partial T_f]^{<\omega}(b)}^{[\partial T_f]^{<\omega}(a\cup b)}\right|=|\iota_b^{a\cup b}|.
 \end{equation*}
 By the definition of the functor $|\cdot|$ (cf.~the paragraph after Summary~\ref{sum:class-sized-dilators}), the first of these equations reduces to
 \begin{equation*}
  \en_{[\partial T_f]^{<\omega}(a\cup b)}\circ|\iota_a^{a\cup b}|=\iota_{[\partial T_f]^{<\omega}(a)}^{[\partial T_f]^{<\omega}(a\cup b)}\circ\en_{[\partial T_f]^{<\omega}(a)}.
 \end{equation*}
 The induction hypothesis tells us that $\partial T_f$ is order preserving on $a\cup b$. Since the increasing enumeration of a finite order is uniquely determined this yields
 \begin{equation*}
  \en_{[\partial T_f]^{<\omega}(a\cup b)}=\partial T_f\circ\en_{a\cup b}:|[\partial T_f]^{<\omega}(a\cup b)|=|a\cup b|\rightarrow[\partial T_f]^{<\omega}(a\cup b).
 \end{equation*}
 Together with $\en_{a\cup b}\circ|\iota_a^{a\cup b}|=\iota_a^{a\cup b}\circ\en_a$ we indeed get
 \begin{align*}
  \en_{[\partial T_f]^{<\omega}(a\cup b)}\circ|\iota_a^{a\cup b}|&=\partial T_f\circ\en_{a\cup b}\circ|\iota_a^{a\cup b}|=\partial T_f\circ\iota_a^{a\cup b}\circ\en_a=\\
  {}&=\iota_{[\partial T_f]^{<\omega}(a)}^{[\partial T_f]^{<\omega}(a\cup b)}\circ\partial T_f\circ\en_a=\iota_{[\partial T_f]^{<\omega}(a)}^{[\partial T_f]^{<\omega}(a\cup b)}\circ\en_{[\partial T_f]^{<\omega}(a)}.
 \end{align*}
 The equation with $b$ at the place of $a$ is established in the same way.
\end{proof}

To get a normal prae-dilator (cf.~Definitions~\ref{def:coded-prae-dilator} and~\ref{def:coded-normal-dil}) we need the following:

\begin{definition}[$\rca_0$]\label{def:deriv-support}
 For each $n$ we define a function $\supp^{\partial T}_n:\partial T_n\rightarrow[n]^{<\omega}$ by induction on the build-up of terms, setting
 \begin{align*}
  \supp^{\partial T}_n(\mu_m)&=\{m\},\\
  \supp^{\partial T}_n(\xi\langle a,\sigma\rangle)&=\textstyle\bigcup_{t\in a}\supp^{\partial T}_n(t).
 \end{align*}
 To define a family of functions $\mu^{\partial T}_n:n\rightarrow\partial T_n$ we put
 \begin{equation*}
  \mu^{\partial T}_n(m)=\mu_m
 \end{equation*}
 for all numbers $m<n$.
\end{definition}

Let us verify that $\partial T=(\partial T,\supp^{\partial T},\mu^{\partial T})$ has the expected property:

\begin{proposition}[$\rca_0$]\label{prop:deriv-normal-prae-dilator}
 If $T$ is a normal prae-dilator, then so is $\partial T$.
\end{proposition}
\begin{proof}
 From Lemma~\ref{lem:deriv-linear} we know that $\partial T_n$ is a linear order for any number~$n$. Lemma~\ref{lem:deriv-embedding} tells us that $f\mapsto\partial T_f$ maps morphisms to morphisms. A straightforward induction over the build-up of terms establishes the functoriality of $\partial T$ and the naturality of $\supp^{\partial T}$. By another induction one can prove the implication
 \begin{equation*}
  \supp^{\partial T}_n(s)\subseteq a\rightarrow s\in\rng(\partial T_{\iota_a\circ\en_a}),
 \end{equation*}
 where $\iota_a:a\hookrightarrow n$ denotes the inclusion of a given $a\subseteq n$. For $a=\supp^{\partial T}_n(s)$ this amounts to the support condition from clause~(ii) of Definition~\ref{def:coded-prae-dilator}. We have thus established that $\partial T$ is a coded prae-dilator. The functions $\mu^{\partial T}_n:n\rightarrow\partial T_n$ clearly form a natural family of embeddings. In view of Definitions~\ref{def:derivative-order} and~\ref{def:deriv-support} a~straightforward induction on the build-up of $s$ shows
 \begin{equation*}
  s<_{\partial T_n}\mu^T_n(m)\qquad\Leftrightarrow\qquad\supp^{\partial T}_n(s)\lef m.
 \end{equation*}
 According to Definition~\ref{def:coded-normal-dil} this means that $\partial T$ is normal.
\end{proof}

Our next goal is to turn $\partial T$ into an upper derivative of $T$. According to Definition~\ref{def:upper-derivative} we need to construct a morphism $\xi^T:T\circ\partial T\Rightarrow T$ of normal prae-dilators. Concerning the notion of composition, Definitions~\ref{def:compose-dils} and~\ref{def:coded-prae-dilator-reconstruct} tell us that any element of $(T\circ\partial T)_n=D^T(\partial T_n)$ has the form $\langle a,\sigma\rangle$, where $a\subseteq\partial T_n$ is finite and $\sigma\in T_{|a|}$ satisfies $\supp^T_{|a|}(\sigma)=|a|$. In view of Definition~\ref{def:derivative-term-system} this justifies the following construction:

\begin{definition}[$\rca_0$]\label{def:deriv-xiT}
For each $n$ we define a function $\xi^T_n:(T\circ\partial T)_n\rightarrow\partial T_n$ by setting
\begin{equation*}
 \xi^T_n(\langle a,\sigma\rangle)=\begin{cases}
                                   \mu_m & \text{if $\langle a,\sigma\rangle=\langle\{\mu_m\},\mu^T_1(0)\rangle$ with $m<n$},\\
                                   \xi\langle a,\sigma\rangle & \text{if $\langle a,\sigma\rangle$ has a different form}.
                                  \end{cases}
\end{equation*}
\end{definition}

The following result is important as it implies the implication (1)$\Rightarrow$(2) from the introduction of this paper.

\begin{proposition}[$\rca_0$]\label{prop:partial-upper-deriv}
 If $T$ is a normal prae-dilator, then $(\partial T,\xi^T)$ is an upper derivative of $T$.
\end{proposition}
\begin{proof}
 In view of Definition~\ref{def:upper-derivative} we must establish that $\xi^T:T\circ\partial T\Rightarrow\partial T$ is a morphism of normal prae-dilators, as characterized by Definition~\ref{def:morphism-dils}. Let us first show that each function $\xi^T_n:(T\circ\partial T)_n\rightarrow\partial T_n$ is an embedding. To make the results from Section~\ref{sect:normal-dils-so} applicable we observe that Definition~\ref{def:extend-normal-transfos} yields
 \begin{equation*}
  \langle\{\mu_m\},\mu^T_1(0)\rangle=D^{\mu^T}_{\partial T_n}(\mu_m).
 \end{equation*}
 To see that $s<_{D^T(\partial T_n)}t$ implies $\xi^T_n(s)<_{\partial T_n}\xi^T_n(t)$ we now distinguish cases according to the form of $s$ and $t$. First assume that we have
 \begin{equation*}
  s=\langle\{\mu_m\},\mu^T_1(0)\rangle=D^{\mu^T}_{\partial T_n}(\mu_m)<_{D^T(\partial T_n)}D^{\mu^T}_{\partial T_n}(\mu_k)=\langle\{\mu_k\},\mu^T_1(0)\rangle=t.
 \end{equation*}
  From Proposition~\ref{prop:reconstruct-normal-dil} we know that $D^{\mu^T}_{\partial T_n}$ is an embedding. Thus we indeed get
  \begin{equation*}
   \xi^T_n(s)=\mu_m<_{\partial T_n}\mu_k=\xi^T_n(t).
  \end{equation*}
  Now consider the case
  \begin{equation*}
   s=\langle a,\sigma\rangle<_{D^T(\partial T_n)}D^{\mu^T}_{\partial T_n}(\mu_k)=\langle\{\mu_k\},\mu^T_1(0)\rangle=t,
  \end{equation*}
  where $\langle a,\sigma\rangle$ is not of the form $\langle\{\mu_m\},\mu^T_1(0)\rangle$. By Proposition~\ref{prop:reconstruct-normal-dil} we get $a\lef_{\partial T_n}\mu_k$. Invoking Definition~\ref{def:derivative-order} we can conclude
  \begin{equation*}
   \xi^T_n(s)=\xi\langle a,\sigma\rangle<_{\partial T_n}\mu_k=\xi^T_n(t).
  \end{equation*}
  The case where we have $s=\langle\{\mu_m\},\mu^T_1(0)\rangle$ and $t=\langle b,\tau\rangle$ is of a different form is treated analogously (infer $\mu_m\leqf_{\partial T_n}b$ from the fact that $b\lef_{\partial T_n}\mu_m$ must fail). Finally we consider the case where we have
  \begin{equation*}
   s=\langle a,\sigma\rangle<_{D^T(\partial T_n)}\langle b,\tau\rangle=t
  \end{equation*}
  and neither $s$ nor $t$ is of the form $\langle\{\mu_m\},\mu^T_1(0)\rangle$ with $m<n$. By Definition~\ref{def:coded-prae-dilator-reconstruct} we get $T_{|\iota_a^{a\cup b}|}(\sigma)<_{T_{|a\cup b|}}T_{|\iota_b^{a\cup b}|}(\tau)$. In view of Definition~\ref{def:derivative-order} this yields
  \begin{equation*}
   \xi^T_n(s)=\xi\langle a,\sigma\rangle<_{\partial T_n}\xi\langle b,\tau\rangle=\xi^T_n(t).
  \end{equation*}
  Let us now show that $\xi^T$ is natural: Given a strictly increasing function~$f:n\rightarrow k$, we invoke Definitions~\ref{def:compose-dils} and~\ref{def:coded-prae-dilator-reconstruct} to obtain
  \begin{equation*}
   \xi^T_k\circ(T\circ\partial T)_f(\langle a,\sigma\rangle)=\xi^T_k(\langle[\partial T_f]^{<\omega}(a),\sigma\rangle).
  \end{equation*}
  First assume $\langle a,\sigma\rangle=\langle\{\mu_m\},\mu^T_1(0)\rangle$. Using Definition~\ref{def:deriv-functor} we compute
  \begin{multline*}
   \xi^T_k\circ(T\circ\partial T)_f(\langle a,\sigma\rangle)=\xi^T_k(\langle\{\mu_{f(m)}\},\mu^T_1(0)\rangle)=\mu_{f(m)}=\\
   =\partial T_f(\mu_m)=\partial T_f\circ\xi^T_n(\langle a,\sigma\rangle).
  \end{multline*}
  Now assume that $\langle a,\sigma\rangle$ does not have the form $\langle\{\mu_m\},\mu^T_1(0)\rangle$. Then $\langle[\partial T_f]^{<\omega}(a),\sigma\rangle$ does not have this form either. Again by Definition~\ref{def:deriv-functor} we get
  \begin{equation*}
   \xi^T_k\circ(T\circ\partial T)_f(\langle a,\sigma\rangle)=\xi\langle[\partial T_f]^{<\omega}(a),\sigma\rangle=\partial T_f(\xi\langle a,\sigma\rangle)=\partial T_f\circ\xi^T_n(\langle a,\sigma\rangle).
  \end{equation*}
  To conclude that $\xi^T:T\circ\partial T\Rightarrow T$ is a morphism of normal prae-dilators we must show $\xi^T\circ\mu^{T\circ\partial T}=\mu^{\partial T}$ (cf.~Definition~\ref{def:morphism-dils}). By Definition~\ref{def:compose-normal-dils} we indeed get
  \begin{equation*}
   \xi^T_n\circ\mu^{T\circ\partial T}_n(m)=\xi^T_n(\langle\{\mu^{\partial T}_n(m)\},\mu^T_1(0)\rangle)=\xi^T_n(\langle\{\mu_m\},\mu^T_1(0)\rangle)=\mu_m=\mu^{\partial T}_n(m)
  \end{equation*}
  for all numbers $m<n$.
\end{proof}

In the construction of $\partial T$ we have only added terms that were needed as values of $\xi:T\circ\partial T\Rightarrow\partial T$. The resulting minimality of $\partial T$ leads to the following:

\begin{theorem}[$\rca_0$]\label{thm:terms-derivative}
Assume that $T$ is a normal prae-dilator. Then $(\partial T,\xi^T)$ is a derivative of $T$.
\end{theorem}
\begin{proof}
The previous proposition tells us that $(\partial T,\xi^T)$ is an upper derivative. In view of Definition~\ref{def:derivative} we assume that $S$ and $\xi':T\circ S\Rightarrow S$ form an upper derivative of $T$ as well. We must show that there is a unique morphism $\nu:\partial T\Rightarrow S$ of upper derivatives. Let us begin with existence: Note that the normality of $S$ is witnessed by a natural family of embeddings $\mu^S_n:n\rightarrow S_n$. Also recall that $(T\circ S)_n=D^T(S_n)$ consists of pairs~$\langle b,\sigma\rangle$, where $b\subseteq S_n$ is finite and $\sigma\in T_{|b|}$ satisfies $\supp^T_{|b|}(\sigma)=|b|$. For each $n$ we define $\nu_n:\partial T_n\rightarrow S_n$ by recursion over the build-up of terms, setting
\begin{align*}
\nu_n(\mu_m)&=\mu^S_n(m),\\
\nu_n(\xi\langle a,\sigma\rangle)&=\xi'_n(\langle[\nu_n]^{<\omega}(a),\sigma\rangle).
\end{align*}
It is not immediately clear that the second clause produces values in~$S_n$: To see that $\xi\langle a,\sigma\rangle$ implies $\langle[\nu_n]^{<\omega}(a),\sigma\rangle\in D^T(S_n)$ we need $|[\nu_n]^{<\omega}(a)|=|a|$, which relies on the fact that $\nu_n$ is order preserving and hence injective. This suggests to verify
 \begin{gather*}
  r\in\partial T_n\rightarrow\nu_n(r)\in S_n,\\
  s<_{\partial T_n}t\rightarrow\nu_n(s)<_{S_n}\nu_n(t)
 \end{gather*}
 by simultaneous induction on $L^{\partial T}_n(r)$ resp.~$L^{\partial T}_n(s)+L^{\partial T}_n(t)$. To establish that $\nu_n$ is order preserving one needs to consider different possibilities for the form of $s$ and~$t$. The first interesting case is
 \begin{equation*}
 s=\xi\langle a,\sigma\rangle<_{\partial T_n}\mu_m=t.
 \end{equation*}
 According to Definition~\ref{def:derivative-order} we have $a\lef_{\partial T_n}\mu_m$, so that the induction hypothesis yields $[\nu_n]^{<\omega}(a)\lef_{S_n}\mu^S_n(m)$. By Definition~\ref{def:coded-normal-dil} we get $\supp^S_n(\nu_n(r))\lef m$ for all~$r\in a$. Using Lemma~\ref{lem:dilator-cartesian} and Definition~\ref{def:compose-dils} we obtain
 \begin{equation*}
 \supp^S_n(\xi'_n(\langle[\nu_n]^{<\omega}(a),\sigma\rangle))=\supp^{T\circ S}_n(\langle[\nu_n]^{<\omega}(a),\sigma\rangle)=\bigcup_{r\in a}\supp^S_n(\nu_n(r))\lef m.
 \end{equation*}
 By the other direction of Definition~\ref{def:coded-normal-dil} this implies
 \begin{equation*}
 \nu_n(s)=\xi'_n(\langle[\nu_n]^{<\omega}(a),\sigma\rangle)<_{S_n}\mu^S_n(m)=\nu_n(t),
 \end{equation*}
 as desired. Let us next consider the case where
 \begin{equation*}
 s=\mu_m<_{\partial T_n}\xi\langle b,\tau\rangle=t
 \end{equation*}
 holds because of $\mu_m\leqf_{\partial T_n} b$. Similarly to the above one can deduce that the statements $\supp^S_n(\nu_n(t))\lef m$ and $\nu_n(t)<_{S_n}\mu^S_n(m)=\nu_n(s)$ must fail. In order to conclude $\nu_n(s)<_{S_n}\nu_n(t)$ we shall now establish $\nu_n(s)\neq\nu_n(t)$. According to Definition~\ref{def:compose-normal-dils} the normality of $T\circ S$ is witnessed by the functions
 \begin{equation*}
 m\mapsto\mu^{T\circ S}_n(m)=\langle\{\mu^S_n(m)\},\mu^T_1(0)\rangle.
 \end{equation*}
 Since $\xi'$ is a morphism of normal prae-dilators we get
 \begin{equation*}
 \nu_n(s)=\mu^S_n(m)=\xi'_n\circ\mu^{T\circ S}_n(m)=\xi'_n(\langle\{\mu^S_n(m)\},\mu^T_1(0)\rangle).
 \end{equation*}
 Invoking the injectivity of the embedding $\xi'_n$ we learn that $\nu_n(s)=\nu_n(t)$ would imply $\langle\{\nu_n(\mu_m)\},\mu^T_1(0)\rangle=\langle[\nu_n]^{<\omega}(b),\tau\rangle$. By induction hypothesis $\nu_n$ is injective on $b\cup\{\mu_m\}$. Hence $\nu_n(s)=\nu_n(t)$ would even yield $t=\xi\langle\{\mu_m\},\mu^T_1(0)\rangle$. This possibility, however, has been excluded in Definition~\ref{def:derivative-term-system}. Finally we assume that
 \begin{equation*}
 s=\xi\langle a,\sigma\rangle<_{\partial T_n}\xi\langle b,\tau\rangle=t.
 \end{equation*}
 holds because of $T_{|\iota_a^{a\cup b}|}(\sigma)<_{T_{|a\cup b|}}T_{|\iota_b^{a\cup b}|}(\tau)$. The induction hypothesis ensures that~$\nu_n$ is order preserving on $a\cup b$. As in the proof of Lemma~\ref{lem:deriv-embedding} one can show
  \begin{equation*}
  \left|\iota_{[\nu_n]^{<\omega}(a)}^{[\nu_n]^{<\omega}(a\cup b)}\right|=|\iota_a^{a\cup b}|\qquad\text{and}\qquad\left|\iota_{[\nu_n]^{<\omega}(b)}^{[\nu_n]^{<\omega}(a\cup b)}\right|=|\iota_b^{a\cup b}|.
 \end{equation*}
Invoking Definition~\ref{def:coded-prae-dilator-reconstruct} one then obtains $\langle [\nu_n]^{<\omega}(a),\sigma\rangle<_{D^T(S_n)}\langle [\nu_n]^{<\omega}(b),\tau\rangle$. Since $\xi'_n$ is an embedding of $(T\circ S)_n=D^T(S_n)$ into $S_n$ this implies
 \begin{equation*}
 \nu_n(s)=\xi'_n(\langle [\nu_n]^{<\omega}(a),\sigma\rangle)<_{S_n}\xi'_n(\langle [\nu_n]^{<\omega}(b),\tau\rangle)=\nu_n(t).
 \end{equation*}
 So far we have established that each function $\nu_n$ is an embedding of $\partial T_n$ into $S_n$. To conclude that these embeddings form a morphism of prae-dilators we must show that they are natural: Given a strictly increasing function $f:n\rightarrow k$, we establish $\nu_k\circ\partial T_f(s)=S_f\circ\nu_n(s)$ by induction on the build-up of $s$. In the case of $s=\mu_m$ we invoke the naturality of $\mu^S$ to get
 \begin{equation*}
 \nu_k\circ\partial T_f(\mu_m)=\nu_k(\mu_{f(m)})=\mu^S_k(f(m))=S_f(\mu^S_n(m))=S_f\circ\nu_n(m).
 \end{equation*}
 Let us now establish the induction step for $s=\xi\langle a,\sigma\rangle$. In view of Definitions~\ref{def:compose-dils} and~\ref{def:coded-prae-dilator-reconstruct} the induction hypothesis yields
 \begin{equation*}
 (T\circ S)_f(\langle[\nu_n]^{<\omega}(a),\sigma\rangle)=\langle[S_f]^{<\omega}\circ[\nu_n]^{<\omega}(a),\sigma\rangle=\langle[\nu_k]^{<\omega}\circ[\partial T_f]^{<\omega}(a),\sigma\rangle.
 \end{equation*}
 Together with the naturality of $\xi':T\circ S\Rightarrow S$ we get
 \begin{multline*}
 \nu_k\circ\partial T_f(\xi\langle a,\sigma\rangle)=\nu_k(\xi\langle[\partial T_f]^{<\omega}(a),\sigma\rangle)=\xi'_k(\langle[\nu_k]^{<\omega}\circ[\partial T_f]^{<\omega}(a),\sigma\rangle)=\\
 =\xi'_k((T\circ S)_f(\langle[\nu_n]^{<\omega}(a),\sigma\rangle))=S_f(\xi'_n(\langle[\nu_n]^{<\omega}(a),\sigma\rangle))=S_f\circ\nu_n(\xi\langle a,\sigma\rangle),
 \end{multline*}
 as required. Next we observe
 \begin{equation*}
 \nu_n\circ\mu^{\partial T}_n(m)=\nu_n(\mu_m)=\mu^S_n(m),
 \end{equation*}
 which shows that $\nu$ is a morphism of normal prae-dilators. To conclude that we have a morphism of upper derivatives we need to establish $\nu\circ\xi^T=\xi'\circ T(\nu)$. First observe that Definitions~\ref{def:comp-morphs} and~\ref{def:coded-prae-dilator-reconstruct} yield
 \begin{equation*}
 T(\nu)_n(\langle a,\sigma\rangle)=D^T(\nu_n)(\langle a,\sigma\rangle)=\langle[\nu_n]^{<\omega}(a),\sigma\rangle.
 \end{equation*}
 If $\langle a,\sigma\rangle\in(T\circ\partial T)_n$ is not of the form $\langle\{\mu_m\},\mu^T_1(0)\rangle$, then we obtain
 \begin{equation*}
 \nu_n\circ\xi^T_n(\langle a,\sigma\rangle)=\nu_n(\xi\langle a,\sigma\rangle)=\xi'_n(\langle[\nu_n]^{<\omega}(a),\sigma\rangle)=\xi'_n\circ T(\nu)_n(\langle a,\sigma\rangle).
 \end{equation*}
 In the remaining case Definition~\ref{def:deriv-xiT} yields
 \begin{equation*}
 \nu_n\circ\xi^T_n(\langle\{\mu_m\},\mu^T_1(0)\rangle)=\nu_n(\mu_m)=\mu^S_n(m).
 \end{equation*}
 Invoking Definition~\ref{def:compose-normal-dils} and the fact that $\xi':T\circ S\Rightarrow S$ is a morphism of normal prae-dilators we also get
 \begin{alignat*}{3}
 \xi'_n\circ T(\nu)_n(\langle\{\mu_m\},\mu^T_1(0)\rangle)&=\xi'_n(\langle\{\nu_n(\mu_m)\},\mu^T_1(0)\rangle)&&={}\\
 {}&=\xi'_n(\langle\{\mu^S_n(m)\},\mu^T_1(0)\rangle)&&=\xi'_n\circ\mu^{T\circ S}_n(m)=\mu^S_n(m).
 \end{alignat*}
 To complete the proof we must show that $\nu$ is unique: Given an arbitrary morphism $\nu':\partial T\Rightarrow S$ of upper derivatives, we establish $\nu'(s)=\nu(s)$ by induction on the build-up of $s$. In the case of $s=\mu_m$ we invoke Definition~\ref{def:deriv-support} and the assumption that $\nu'$ is a morphism of normal prae-dilators to get
 \begin{equation*}
 \nu'_n(\mu_m)=\nu'_n\circ\mu^{\partial T}_n(m)=\mu^S_n(m)=\nu_n(\mu_m).
 \end{equation*}
Given a term $s=\xi\langle a,\sigma\rangle$, we observe that the induction hypothesis implies
\begin{equation*}
T(\nu')_n(\langle a,\sigma\rangle)=D^T(\nu'_n)(\langle a,\sigma\rangle)=\langle[\nu'_n]^{<\omega}(a),\sigma\rangle=\langle[\nu_n]^{<\omega}(a),\sigma\rangle.
\end{equation*}
Together with the assumption that $\nu'$ is a morphism of upper derivatives we obtain
\begin{equation*}
\nu'_n(\xi\langle a,\sigma\rangle)=\nu'_n\circ\xi^T_n(\langle a,\sigma\rangle)=\xi'_n\circ T(\nu')_n(\langle a,\sigma\rangle)=\xi'_n(\langle[\nu_n]^{<\omega}(a),\sigma\rangle)=\nu_n(\xi\langle a,\sigma\rangle),
\end{equation*}
as required.
\end{proof}

We have described a construction that yields a derivative $\partial T$ of a given normal prae-dilator~$T$. Since derivatives are essentially unique, the construction of $\partial T$ can be exploited to prove general properties of derivatives. The following result establishes some of the assumptions from Theorem~\ref{thm:equalizer-to-derivative}. The remaining assumption, which states that $X\mapsto D^S_X$ preserves well-foundedness, will be considered in the next section (in view of Corollary~\ref{cor:deriv-implies-BI} this will require a stronger base theory).

\begin{theorem}[$\rca_0$]\label{thm:partial-derivative}
Consider a normal prae-dilator $T$. If $(S,\xi)$ is a derivative of $T$, then $\xi:T\circ S\Rightarrow S$ is a natural isomorphism. Furthermore
\begin{equation*}
   \begin{tikzcd}
  n\ar{r}{\mu^S_n} &[2em] S_n\arrow[r,shift left,"\id_{S_n}"]\arrow[r,shift right,swap,"\xi_n\circ D^{\mu^T}_{S_n}"]&[5em] S_n
 \end{tikzcd}
 \end{equation*}
is an equalizer diagram for every number~$n$.
\end{theorem}
\begin{proof}
 The definition of derivative ensures that $\xi$ is a natural transformation. To conclude that it is a natural isomorphism we show that $\xi_n:(T\circ S)_n\rightarrow S_n$ is surjective for each~$n$. Since both $(S,\xi)$ and $(\partial T,\xi^T)$ are derivatives of $T$, there is an isomorphism $\nu:S\Rightarrow\partial T$ of upper derivatives (cf.~the remark after Definition~\ref{def:derivative}). In view of Definitions~\ref{def:morph-upper-derivs} and~\ref{def:comp-morphs} we have
 \begin{equation*}
  \nu_n\circ\xi_n=\xi^T_n\circ T(\nu)_n=\xi^T_n\circ D^T(\nu_n).
 \end{equation*}
 Now $\nu_n$ is bijective, and it is straightforward to infer that the same holds for $D^T(\nu_n)$. So it suffices to show that $\xi^T_n$ is surjective. Aiming at the latter, we first observe that Definitions~\ref{def:deriv-support} and~\ref{def:morphism-dils} yield
 \begin{equation*}
  \mu_m=\mu^{\partial T}_n(m)=\xi^T_n\circ\mu^{T\circ\partial T}_n(m)\in\rng(\xi^T_n).
 \end{equation*}
 It remains to consider an element $\xi\langle a,\sigma\rangle\in \partial T_n$. In view of Definitions~\ref{def:derivative-term-system},~\ref{def:coded-prae-dilator-reconstruct} and~\ref{def:compose-dils} we have $\langle a,\sigma\rangle\in D^T(\partial T_n)=(T\circ\partial T)_n$. Thus we get
 \begin{equation*}
 \xi\langle a,\sigma\rangle=\xi^T_n(\langle a,\sigma\rangle)\in\rng(\xi^T_n),
 \end{equation*}
 which completes the proof that $\xi$ is a natural isomorphism. After the statement of Theorem~\ref{thm:equalizer-to-derivative} above we have observed that the given equalizer diagram is automatically commutative. To establish that $\mu^S_n$ is an equalizer of $\xi_n\circ D^{\mu^T}_{S_n}$ and the identity we must show that
 \begin{equation*}
  \xi_n\circ D^{\mu^T}_{S_n}(s)=s\quad\Rightarrow\quad s\in\rng(\mu^S_n)
 \end{equation*}
 holds for any element $s\in S_n$. To reduce the claim to the special case with $(\partial T,\xi^T)$ at the place of $(S,\xi)$ we apply $\nu_n$ to both sides of the antecedent. Using the naturality of~$D^{\mu^T}$, which is provided by Lemma~\ref{prop:reconstruct-normal-dil}, this yields
 \begin{equation*}
  \nu_n(s)=\nu_n\circ\xi_n\circ D^{\mu^T}_{S_n}(s)=\xi^T_n\circ D^T(\nu_n)\circ D^{\mu^T}_{S_n}(s)=\xi^T_n\circ D^{\mu^T}_{\partial T_n}\circ\nu_n(s).
 \end{equation*}
 Assuming the special case of the desired implication, we obtain $\nu_n(s)\in\rng(\mu^{\partial T}_n)$, say $\nu_n(s)=\mu^{\partial T}_n(m)$. Since $\nu$ is a morphism of normal prae-dilators, this implies
 \begin{equation*}
  \nu_n\circ\mu^S_n(m)=\mu^{\partial T}_n(m)=\nu_n(s).
 \end{equation*}
 Invoking the injectivity of $\nu_n$ we see $s=\mu^S_n(m)\in\rng(\mu^S_n)$, which is the conclusion of the general case. It remains to establish the special case for $\partial T$. Aiming at the contrapositive of the desired implication, let us assume that $s\in\partial T_n$ is not of the form $\mu^S_n(m)=\mu_m$. Then Definitions~\ref{def:extend-normal-transfos} and~\ref{def:deriv-xiT} yield
 \begin{equation*}
  \xi^T_n\circ D^{\mu^T}_{\partial T_n}(s)=\xi^T_n(\langle\{s\},\mu^T_1(0)\rangle)=\xi\langle\{s\},\mu^T_1(0)\rangle.
 \end{equation*}
 The term on the right cannot be equal to $s$, which it contains as a proper subterm (one can also appeal to the fact that $s$ is shorter in the sense of Definition~\ref{def:deriv-term-length}).
\end{proof}

To conclude this section we show that the conditions from the previous theorem do not suffice to characterize derivatives on the categorical level:

\begin{example}\label{ex:equalizers-without-deriv}
 Define a normal dilator $T$ by setting $T_n=\{0,\dots,n-1\}$, $T_f=f$, $\supp^T_n(m)=\{m\}$ and $\mu^T_n(m)=m$. Furthermore, consider the sets
 \begin{equation*}
  S_n=\mathbb Z+n=\{\hat p\,|\,p\in\mathbb Z\}\cup\{m\,|\,0\leq m<n\}
 \end{equation*}
 with the expected ordering (i.\,e.~such that $\hat p<_{S_n}\hat q<_{S_n}m<_{S_n}k$ holds for all $p<q$ from~$\mathbb Z$ and all $m<k$ from $\{0,\dots,n-1\}$). To turn $S$ into a prae-dilator we set
 \begin{align*}
  S_f(\sigma)&=\begin{cases}
               f(m) & \text{if $\sigma=m\in\{0,\dots,n-1\}$,}\\
               \sigma & \text{if $\sigma=\hat p$ with $p\in\mathbb Z$,}
              \end{cases}\\
  \supp^S_n(\sigma)&=\begin{cases}
                     \{m\} & \text{if $\sigma=m\in\{0,\dots,n-1\}$,}\\
                     \emptyset & \text{if $\sigma=\hat p$ with $p\in\mathbb Z$.}
                    \end{cases}
 \end{align*}
 Let us point out that $S$ is not a dilator, since $D^S_n\cong S_n$ is ill-founded (cf.~Lemma~\ref{lem:class-sized-restrict}). Be that as it may, we obtain a normal prae-dilator by setting
 \begin{equation*}
  \mu^S_n(m)=m\in\{0,\dots,n-1\}\subseteq S_n.
 \end{equation*}
 Since all supports with respect to $T$ are singletons we have
 \begin{equation*}
  (T\circ S)_n=D^T(S_n)=\{\langle\{\sigma\},0\rangle\,|\,\sigma\in S_n\}.
 \end{equation*}
 Thus we can define $\xi:T\circ S\Rightarrow S$ by setting
 \begin{equation*}
  \xi_n(\langle\{\sigma\},0\rangle)=\begin{cases}
                                   m & \text{if $\sigma=m\in\{0,\dots,n-1\}$,}\\
                                   \widehat{p+1} & \text{if $\sigma=\hat p$ with $p\in\mathbb Z$.}
                                  \end{cases}
 \end{equation*}
 One can check that $(S,\xi)$ is an upper derivative of $T$. It is easy to see that~$\xi$ is an isomorphism, as $p\mapsto p+1$ is an automorphism of $\mathbb Z$. Furthermore, the diagram from Theorem~\ref{thm:partial-derivative} defines an equalizer: Assume that we have
 \begin{equation*}
  \sigma=\xi_n\circ D^{\mu^T}_{S_n}(\sigma)=\xi_n(\langle\{\sigma\},\mu^T_1(0)\rangle)=\xi_n(\langle\{\sigma\},0\rangle).
 \end{equation*}
 In view of $p\neq p+1$ we cannot have $\sigma=\hat p$ with $p\in\mathbb Z$. Thus we must have $\sigma=m$ for some number $m\in\{0,\dots,n-1\}$. It follows that
 \begin{equation*}
  \sigma=m=\mu^S_n(m)
 \end{equation*}
 lies in the range of $\mu^S_n$, as required for the equalizer condition. Thus $S$ and $\xi$ satisfy the conclusion of the previous theorem. Nevertheless they do not form a derivative of $T$. Otherwise we would get a morphism $S\Rightarrow\partial T$ of upper derivatives. This is impossible since $S_n=\mathbb Z+n$ is infinite while $\partial T_n$ is finite: In view of $T_1=\{\mu^T_1(0)\}$ Definition~\ref{def:derivative-term-system} yields $\partial T_n=\{\mu_m\,|\,0\leq m<n\}$.
\end{example}

\section{From $\Pi^1_1$-bar induction to preservation of well-foundedness}\label{sect:bi-deriv-wf}

In this section we use $\Pi^1_1$-bar induction to prove the following: If~$T$ is a normal dilator, then $X\mapsto D^{\partial T}_X$ preserves well-foundedness, so that $\partial T$ is a normal dilator as well. This establishes the implication (3)$\Rightarrow$(1) from the introduction. Together with the results of the previous sections we learn that (1), (2) and~(3) are equivalent over $\aca_0$. Invoking Theorems~\ref{thm:equalizer-to-derivative} and~\ref{thm:partial-derivative} we will also be able to conclude the following: If $(S,\xi)$ is a derivative of a normal dilator $T$, then $\alpha\mapsto\otp(D^S_\alpha)$ is the derivative (in the usual sense) of the normal function $\alpha\mapsto\otp(D^T_\alpha)$.

The construction from the previous section yields a derivative~$\partial T$ of a given normal prae-dilator~$T$. To assess whether $\partial T$ is a dilator we must consider the orders $D^{\partial T}_X$ (cf.~Definitions~\ref{def:coded-prae-dilator-reconstruct} and~\ref{def:coded-dilator}). These will be approximated as follows:

\begin{definition}[$\rca_0$]\label{def:partial-up-to-x}
Consider a normal prae-dilator $T$, as well as a linear order $X=(X,<_X)$. We set
\begin{equation*}
\partial T^x_X=\{\langle a,s\rangle\in D^{\partial T}_X\,|\,a\lef_X x\}
\end{equation*}
for any element $x\in X$.
\end{definition}

To distinguish the expressions $\partial T_X^x$ and $\partial T_n$ (cf.~Definition~\ref{def:derivative-term-system}) it suffices to observe that the latter has no superscript (note that we have $\partial T_n^m\subseteq D^{\partial T}_n$ rather than $\partial T_n^m\subseteq\partial T_n$ in case $X=n=\{0,\dots,n-1\}$). We will argue by induction on~$x$ to show that the suborders $\partial T^x_X\subseteq D^{\partial T}_X$ are well-founded. Assuming that $X$ is non-empty and has no maximal element, we clearly have
\begin{equation*}
D^{\partial T}_X=\bigcup_{x\in X}\partial T^x_X.
\end{equation*}
In general, the union (or direct limit) of compatible well-orders does not need to be well-founded itself. On the other hand it is straightforward to see that an order is well-founded if it is the union of well-founded initial segments. In the present situation we can combine Propositions~\ref{prop:deriv-normal-prae-dilator} and~\ref{prop:reconstruct-normal-dil} to get the following:

\begin{lemma}[$\rca_0$]\label{lem:partialTx-initial}
Consider a normal prae-dilator $T$ and a linear order $X$. For any $x\in X$ we have
\begin{equation*}
 \partial T^x_X=D^{\partial T}_X\!\restriction\!D^{\mu^{\partial T}}_X(x)=\{\sigma\in D^{\partial T}_X\,|\,\sigma<_{D^{\partial T}_X}D^{\mu^{\partial T}}_X(x)\}.
\end{equation*}
In particular $\partial T^x_X$ is an initial segment of $D^{\partial T}_X$.
\end{lemma}

The assumption that $T$ and hence $\partial T$ is normal is crucial for the previous lemma and for many of the following results (cf.~the remarks before Lemma~\ref{lem:range-dil-support}, as well as the discussion of Aczel's construction at the beginning of Section~\ref{sect:constructin-derivative}). Assuming that $\partial T^y_X$ is well-founded for every $y<_X x$, the lemma allows us to conclude that $\bigcup_{y<_Xx}\partial T_X^y$ is well-founded. To complete the induction step one needs to deduce the well-foundedness of $\partial T^x_X$. For this purpose we approximate~$\partial T^x_X$ by distinguishing terms of different height (cf.~Definition~\ref{def:derivative-term-system}):

\begin{definition}[$\rca_0$]\label{def:approximations-height}
Let $T$ be a normal prae-dilator. We define a family of functions $\hth^{\partial T}_n:\partial T_n\rightarrow\mathbb N$ by induction over the build-up of terms, setting
\begin{align*}
\hth^{\partial T}_n(\mu_m)&=0,\\
\hth^{\partial T}_n(\xi\langle a, \sigma\rangle)&=\begin{cases}
                                                   \hth^{\partial T}_n(s)+1 & \text{if $s$ is the $<_{\partial T_n}$-maximal element of $a$},\\
                                                   1 & \text{if $a=\emptyset$}.
                                                  \end{cases}
\end{align*}
Given an order $X$ and an element $x\in X$, we put
\begin{equation*}
\partial T^{x,k}_X=\bigcup_{y<_X x}\partial T^y_X\cup\{\langle a,s\rangle\in\partial T^x_X\,|\,\hth^{\partial T}_{|a|}(s)\leq k\}
\end{equation*}
for every number~$k$.
\end{definition}

According to Definition~\ref{def:deriv-functor} and Lemma~\ref{lem:deriv-embedding}, any strictly increasing function $f:n\rightarrow k$ yields an embedding $\partial T_f:\partial T_n\rightarrow\partial T_k$. We will need to know that these embeddings respect our height functions:

\begin{lemma}[$\rca_0$]\label{lem:heights-functorial}
Consider a normal prae-dilator $T$ and a strictly increasing function $f:n\rightarrow k$. We have
 \begin{equation*}
  \hth^{\partial T}_k(\partial T_f(s))=\hth^{\partial T}_n(s)
 \end{equation*}
for any element $s\in\partial T_n$.
\end{lemma}
\begin{proof}
The claim can be verified by a straightforward induction on the build-up of~$s$. Concerning the case $s=\xi\langle a,\sigma\rangle$, we point out that $\partial T_f(s')$ is $<_{\partial T_k}$-maximal in $[\partial T_f]^{<\omega}(a)$ if $s'$ is $<_{\partial T_n}$-maximal in $a$.
\end{proof}

Yet again, it will be crucial that Definition~\ref{def:approximations-height} provides an approximation by initial segments. To show that this is the case we need a partial converse to Lemma~\ref{lem:supports-order-normal}:

\begin{lemma}[$\rca_0$]\label{lem:normal-height-ineq}
If $T$ is a normal prae-dilator, then we have
\begin{equation*}
\supp^{\partial T}_n(s)\leqf\supp^{\partial T}_n(t)\text{ and }\hth^{\partial T}_n(s)<\hth^{\partial T}_n(t)\quad\Rightarrow\quad s<_{\partial T_n} t
\end{equation*}
for any number $n$ and arbitrary elements $s,t\in\partial T_n$.
\end{lemma}
\begin{proof}
We establish the claim by induction on $L^{\partial T}_n(s)+L^{\partial T}_n(t)$, relying on the length function from Definition~\ref{def:deriv-term-length}. To prove the induction step we distinguish cases according to the form of $s$ and $t$. In any case we assume $\supp^{\partial T}_n(s)\leqf\supp^{\partial T}_n(t)$ and~$\hth^{\partial T}_n(s)<\hth^{\partial T}_n(t)$. Let us first consider terms $s=\mu_m$ and~$t=\xi\langle b,\tau\rangle$. In this case we need neither the induction hypothesis nor the assumption about heights: In view of Definition~\ref{def:deriv-support} we have $\{m\}\leqf\supp^{\partial T}_n(t)$, so that $\supp^{\partial T}_n(t)\lef m$ must fail. Invoking Definition~\ref{def:coded-normal-dil} in conjunction with Proposition~\ref{prop:deriv-normal-prae-dilator} we obtain
\begin{equation*}
s=\mu_m=\mu^{\partial T}_n(m)\leq_{\partial T_n}t.
\end{equation*}
Since $s$ and $t$ are different terms we can conclude $s<_{\partial T_n}t$. Now consider $s=\xi\langle a,\sigma\rangle$ and $t=\mu_k$. Definition~\ref{def:derivative-order} tells us that $s<_{\partial T_n}t$ is equivalent to $a\lef_{\partial T_n}t$. The latter is trivial if $a$ is empty. Otherwise we consider the maximal element $s'\in a$. In view of Definition~\ref{def:deriv-support} we get $\supp^{\partial T}_n(s')\leqf\supp^{\partial T}_n(t)$. Clearly we also have $\hth^{\partial T}_n(s')<\hth^{\partial T}_n(t)$ and~$L^{\partial T}_n(s')<L^{\partial T}_n(s)$. Thus we obtain $s'<_{\partial T_n}t$ by induction hypothesis. Since $s'\in a$ is maximal this establishes $a\lef_{\partial T_n}t$, as needed. Finally we consider $s=\xi\langle a,\sigma\rangle$ and $t=\xi\langle b,\tau\rangle$. In the proof of Lemma~\ref{lem:deriv-linear} we have seen that $s<_{\partial T_n}t$ holds if $b\leqf_{\partial T_n} a$ fails. In order to refute $b\leqf_{\partial T_n} a$ we observe that $\hth^{\partial T}_n(s)<\hth^{\partial T}_n(t)$ implies $b\neq\emptyset$. Let $t'\in b$ be maximal with respect to~$<_{\partial T_n}$. To complete the proof it suffices to establish $a\lef_{\partial T_n}t'$. Yet again this is trivial if $a$ is empty. Otherwise the claim reduces to $s'<_{\partial T_n}t'$, where $s'\in a$ is maximal. The maximality of $t'$ and Lemma~\ref{lem:supports-order-normal} ensure that $\supp^{\partial T}_n(r)\leqf\supp^{\partial T}_n(t')$ holds for all $r\in b$. In view of Definition~\ref{def:deriv-support} this yields
\begin{equation*}
 \supp^{\partial T}_n(s')\leqf\supp^{\partial T}_n(s)\leqf\supp^{\partial T}_n(t)\leqf\supp^{\partial T}_n(t').
\end{equation*}
As we also have $\hth^{\partial T}_n(s')<\hth^{\partial T}_n(t')$, the induction hypothesis yields $s'<_{\partial T_n} t'$.
\end{proof}

For our approximations of $\partial T^x_X$ we get the following:

\begin{proposition}[$\rca_0$]\label{prop:partial-xk-initial}
 Consider a normal prae-dilator $T$, an order $X$ and an element $x\in X$. For any number $k$ the order $\partial T^{x,k}_X$ is an initial segment of $\partial T^x_X$.
\end{proposition}
\begin{proof}
 Given $\langle a,s\rangle\in\partial T^{x,k}_X$ and $\langle b,t\rangle\in\partial T^x_X$, we must show that $\langle b,t\rangle\leq_{D^{\partial T}_X}\langle a,s\rangle$ implies $\langle b,t\rangle\in\partial T^{x,k}_X$. If we have $\langle a,s\rangle\in\partial T^y_X$ for some $y<_X x$, then we can conclude by Lemma~\ref{lem:partialTx-initial}. So we may assume $\hth^{\partial T}_{|a|}(s)\leq k$. Aiming at the contra\-positive of the desired implication, let us assume that $\langle b,t\rangle\in\partial T^{x,k}_X$ fails. Then we have $\hth^{\partial T}_{|b|}(t)>k$, and $b\lef_X y$ must fail for all $y<_X x$. In view of $a\lef_X x$ we get
 \begin{equation*}
  a\leqf_X b\qquad\text{and}\qquad\hth^{\partial T}_{|a|}(s)<\hth^{\partial T}_{|b|}(t).
 \end{equation*}
 To complete the proof of the contra\-positive we must show $\langle a,s\rangle<_{D^{\partial T}_X}\langle b,t\rangle$. In view of Definition~\ref{def:coded-prae-dilator-reconstruct} this amounts to
 \begin{equation*}
  \partial T_{|\iota_a^{a\cup b}|}(s)<_{\partial T_{|a\cup b|}}\partial T_{|\iota_a^{a\cup b}|}(t).
 \end{equation*}
 In order to show this inequality it suffices to establish the assumptions of Lemma~\ref{lem:normal-height-ineq}, which we shall do in the rest of the proof. Recall that $\supp^{\partial T}$ is natural, that we have $\en_{a\cup b}\circ|\iota_a^{a\cup b}|=\iota_a^{a\cup b}\circ\en_a$ (see the beginning of Section~\ref{sect:normal-dils-so}), and that $\langle a,s\rangle\in D^{\partial T}_X$ requires $\supp^{\partial T}_{|a|}(s)=|a|$ (see Definition~\ref{def:coded-prae-dilator-reconstruct}). Combining these facts we obtain
 \begin{multline*}
  [\en_{a\cup b}]^{<\omega}(\supp^{\partial T}_{|a\cup b|}(\partial T_{|\iota_a^{a\cup b}|}(s)))=[\en_{a\cup b}]^{<\omega}\circ[|\iota_a^{a\cup b}|]^{<\omega}(\supp^{\partial T}_{|a|}(s))=\\
  =[\iota_a^{a\cup b}]^{<\omega}\circ[\en_a]^{<\omega}(|a|)=a.
 \end{multline*}
 In the same way one can establish
 \begin{equation*}
  [\en_{a\cup b}]^{<\omega}(\supp^{\partial T}_{|a\cup b|}(\partial T_{|\iota_b^{a\cup b}|}(t)))=b.
 \end{equation*}
 Above we have shown $a\leqf_X b$. Since $\en_{a\cup b}$ is order preserving we can now conclude
 \begin{equation*}
  \supp^{\partial T}_{|a\cup b|}(\partial T_{|\iota_a^{a\cup b}|}(s))\leqf\supp^{\partial T}_{|a\cup b|}(\partial T_{|\iota_b^{a\cup b}|}(t)),
 \end{equation*}
 which is the first of the assumptions needed for Lemma~\ref{lem:normal-height-ineq}. Above we have also seen $\hth^{\partial T}_{|a|}(s)<\hth^{\partial T}_{|b|}(t)$. Together with Lemma~\ref{lem:heights-functorial} we now get
 \begin{equation*}
  \hth^{\partial T}_{|a\cup b|}(\partial T_{|\iota_a^{a\cup b}|}(s))<\hth^{\partial T}_{|a\cup b|}(\partial T_{|\iota_b^{a\cup b}|}(t)).
 \end{equation*}
 This establishes the second assumption of Lemma~\ref{lem:normal-height-ineq} and completes the proof.
\end{proof}

The crucial induction step relies on the assumption that $T$ is a dilator:

\begin{proposition}[$\rca_0$]\label{prop:wf-height-progressive}
 Consider a normal dilator $T$, a linear order $X$, an element $x\in X$, and a number $k$. If $\partial T^{x,k}_X$ is well-founded, then so is $\partial T^{x,k+1}_X$. 
\end{proposition}
\begin{proof}
 Assume that $\partial T^{x,k}_X$ is a well-order. As $T$ is a dilator it follows that $D^T(\partial T^{x,k}_X)$ is a well-order as well. In order to conclude that $\partial T^{x,k+1}_X$ is well-founded it suffices to show that this order can be embedded into~$D^T(\partial T^{x,k}_X)$. First observe that
 \begin{equation*}
  D^T(\partial T^{x,k}_X)=\{\langle a,\sigma\rangle\in D^T(D^{\partial T}_X)\,|\,a\subseteq\partial T^{x,k}_X\}
 \end{equation*}
 is a suborder of $D^T(D^{\partial T}_X)$ (apply Lemma~\ref{lem:range-dil-support} to the inclusion map $\partial T^{x,k}_X\hookrightarrow D^{\partial T}_X$). Also recall that $\partial T$ comes with a natural transformation $\xi^T:T\circ\partial T\Rightarrow\partial T$, which is an isomorphism by the proof of Theorem~\ref{thm:partial-derivative}. In view of Definition~\ref{def:morphism-dils} and Proposition~\ref{prop:compose-dils-rca} we obtain an isomorphism
 \begin{equation*}
  D^{\xi^T}_X\circ\zeta^{T,\partial T}_X:D^T(D^{\partial T}_X)\rightarrow D^{\partial T}_X.
 \end{equation*}
 It suffices to show that $\partial T^{x,k+1}_X$ is contained in the image of $D^T(\partial T^{x,k}_X)$ under this isomorphism, which is equivalent to the assertion that
 \begin{equation*}
 D^{\xi^T}_X\circ\zeta^{T,\partial T}_X(\sigma)\in\partial T^{x,k+1}_X\quad\Rightarrow\quad\sigma\in D^T(\partial T^{x,k}_X)
 \end{equation*}
 holds for any element $\sigma\in D^T(D^{\partial T}_X)$. To establish this fact we write
 \begin{equation*}
 \sigma=\langle\{\langle a_1,s_1\rangle,\dots,\langle a_n,s_n\rangle\},\tau\rangle,
 \end{equation*}
 such that the pairs $\langle a_j,s_j\rangle$ are displayed in increasing order. If we have $n=0$, then $\sigma\in D^T(\partial T^{x,k}_X)$ is immediate. Thus we assume $n>0$ for the rest of the proof. Under this assumption, Proposition~\ref{prop:partial-xk-initial} and Lemma~\ref{lem:partialTx-initial} imply that $\sigma\in D^T(\partial T^{x,k}_X)$ is equivalent to $\langle a_n,s_n\rangle\in\partial T^{x,k}_X$. By Proposition~\ref{prop:compose-dils-rca} and Definition~\ref{def:morphism-dils} we have
 \begin{equation*}
 D^{\xi^T}_X\circ\zeta^{T,\partial T}_X(\sigma)=\langle c,\xi^T_{|c|}(\langle\{\partial T_{|\iota_1|}(s_1),\dots,\partial T_{|\iota_n|}(s_n)\},\tau\rangle)\rangle,
 \end{equation*}
 where $\iota_j:a_j\hookrightarrow a_1\cup\dots\cup a_n=:c$ are the inclusions. In view of $a_n\subseteq c$ we learn that $D^{\xi^T}_X\circ\zeta^{T,\partial T}_X(\sigma)\in\partial T^y_X$ implies $\langle a_n,s_n\rangle\in\partial T^y_X$, for any $y\in X$. To complete the proof it suffices to establish the implication
 \begin{equation*}
  \hth^{\partial T}_{|c|}(\xi^T_{|c|}(\langle\{\partial T_{|\iota_1|}(s_1),\dots,\partial T_{|\iota_n|}(s_n)\},\tau\rangle))\leq k+1\quad\Rightarrow\quad\hth^{\partial T}_{|a_n|}(s_n)\leq k.
 \end{equation*}
 In view of Definition~\ref{def:deriv-xiT} we distinguish two cases: First assume that we have
 \begin{equation*}
  \xi^T_{|c|}(\langle\{\partial T_{|\iota_1|}(s_1),\dots,\partial T_{|\iota_n|}(s_n)\},\tau\rangle)=\xi\langle\{\partial T_{|\iota_1|}(s_1),\dots,\partial T_{|\iota_n|}(s_n)\},\tau\rangle.
 \end{equation*}
 Invoking Definition~\ref{def:coded-prae-dilator-reconstruct} we see that the map $\langle a_j,s_j\rangle\mapsto\partial T_{|\iota_j|}(s_j)$ is order preserving. Thus the values $\partial T_{|\iota_j|}(s_j)$ are displayed in increasing order as well. By Lemma~\ref{lem:heights-functorial} and our definition of heights we get
 \begin{equation*}
 \hth^{\partial T}_{|a_n|}(s_n)=\hth^{\partial T}_{|c|}(\partial T_{|\iota_n|}(s_n))<\hth^{\partial T}_{|c|}(\xi^T_{|c|}(\langle\{\partial T_{|\iota_1|}(s_1),\dots,\partial T_{|\iota_n|}(s_n)\},\tau\rangle)),
 \end{equation*} 
 which yields the desired implication. Now assume that we have
 \begin{equation*}
  \xi^T_{|c|}(\langle\{\partial T_{|\iota_1|}(s_1),\dots,\partial T_{|\iota_n|}(s_n)\},\tau\rangle)=\mu_m
 \end{equation*}
 for some $m<|c|$. This can only happen if we have $\partial T_{|\iota_n|}(s_n)=\mu_m$. We then get
 \begin{equation*}
  \hth^{\partial T}_{|a_n|}(s_n)=\hth^{\partial T}_{|c|}(\partial T_{|\iota_n|}(s_n))=0\leq k,
 \end{equation*}
 which is the conclusion of the desired implication.
\end{proof}

To deduce the main result of this section we extend the base theory by the principle of bar induction for $\Pi^1_1$-formulas (abbreviated $\mathbf{\Pi^1_1}\textbf{-BI}$).

\begin{theorem}[$\rca_0+\mathbf{\Pi^1_1}\textbf{-BI}$]\label{thm:BI-yields-deriv}
If $T$ is a normal dilator, then so is $\partial T$.
\end{theorem}
\begin{proof}
From Proposition~\ref{prop:deriv-normal-prae-dilator} we know that $\partial T$ is a normal prae-dilator. In view of Definition~\ref{def:coded-dilator} it remains to establish that $D^{\partial T}_X$ is well-founded for any well-order~$X$. It suffices to consider the case where $X$ is a limit order, i.\,e.~a non-empty order without a maximal element: If $X$ itself does not have this property, then we replace it by the order~$X+\omega$, in which the initial segment $X$ is followed by a copy of the natural numbers. By Proposition~\ref{prop:reconstruct-class-sized-dil} the inclusion $X\hookrightarrow X+\omega$ yields an embedding of $D^{\partial T}_X$ into $D^{\partial T}_{X+\omega}$, so that the former order is well-founded if the latter~is. For the rest of this proof we assume that $X$ is a well-founded limit order. In view of Definition~\ref{def:partial-up-to-x} we then have
\begin{equation*}
D^{\partial T}_X=\bigcup_{x\in X}\partial T_X^x.
\end{equation*}
Let us argue that $D^{\partial T}_X$ is well-founded if $\partial T_X^x$ is well-founded for every $x\in X$: To find a minimal element of a non-empty set $Y\subseteq D^{\partial T}_X$, pick an element $x\in X$ such that $Y\cap\partial T_X^x$ is non-empty. The well-foundedness of $\partial T_X^x$ provides a minimal element $\sigma\in Y\cap\partial T_X^x$. From Lemma~\ref{lem:partialTx-initial} we know that $\partial T_X^x$ is an initial segment of~$D^{\partial T}_X$. It follows that $\sigma$ is minimal in the entire set $Y$, as required. Invoking the principle of $\Pi^1_1$-bar induction, we shall now establish the well-foundedness of $\partial T_X^x$ by induction on $x\in X$. In the induction step we argue that Definition~\ref{def:approximations-height} and Proposition~\ref{prop:partial-xk-initial} allow us to write
\begin{equation*}
\partial T^x_X=\bigcup_{k\in\mathbb N}\partial T^{x,k}_X
\end{equation*}
as a union of initial segments. Once again it follows that $\partial T^x_X$ is well-founded if $\partial T^{x,k}_X$ is well-founded for every number $k$. We argue by side induction on $k$ to show that the latter is the case. Note that induction over the natural numbers is available as a particular instance of bar induction (alternatively one could combine the main and side induction into a single induction over $X\times\omega$). The side induction step is provided by Proposition~\ref{prop:wf-height-progressive}. To complete the proof it is thus enough to establish the base of the side induction. As a preparation we consider an element $\langle a,s\rangle\in\partial T_X^x$ with~$\hth^{\partial T}_{|a|}(s)=0$. In view of Definition~\ref{def:approximations-height} we must have $s=\mu_m$ for some number~$m<|a|$. Together with Definitions~\ref{def:coded-prae-dilator-reconstruct} and~\ref{def:deriv-support} we obtain
\begin{equation*}
|a|=\supp^{\partial T}_{|a|}(\mu_m)=\{m\}.
\end{equation*}
This forces $m=0$ and $|a|=1$, say $a=\{z\}$ with $z\in X$. Altogether we get
\begin{equation*}
\langle a,s\rangle=\langle\{z\},\mu_0\rangle,
\end{equation*}
where $\langle a,s\rangle\in\partial T_X^x$ ensures $z<_X x$. Let us now distinguish two cases: First assume that $x\in X$ is a limit or zero (i.\,e.~for every $z<_X x$ there is a $y<_X x$ with $z<_X y$). Then Definition~\ref{def:approximations-height} yields
\begin{equation*}
\partial T^{x,0}_X=\bigcup_{y<_X x}\partial T^y_X.
\end{equation*}
By Lemma~\ref{lem:partialTx-initial} this is a union of initial segments. The main induction hypothesis ensures that $\partial T^y_X$ is well-founded for every $y<_X x$. It follows that $\partial T^{x,0}_X$ is well-founded, as required. Now assume that $x$ is the successor of an element~$z\in X$, so that $y<_X x$ is equivalent to $y\leq_X z$. In view of the above we obtain
\begin{equation*}
\partial T^{x,0}_X=\partial T^z_X\cup\{\langle\{z\},\mu_0\rangle\}.
\end{equation*}
Once again the main induction hypothesis tells us that $\partial T^z_X$ is well-founded. Since $\partial T^{x,0}_X$ is a finite extension of this order, it must be well-founded itself. We have thus established the base of the side induction, which completes the proof.
\end{proof}

To shed further light on the previous proof we point out that we have
\begin{equation*}
\partial T^z_X\cup\{\langle\{z\},\mu_0\rangle\}=\{\sigma\in D^{\partial T}_X\,|\,\sigma\leq_{D^{\partial T}_X} D^{\mu^{\partial T}}_X(z)\},
\end{equation*}
due to Definitions~\ref{def:extend-normal-transfos} and~\ref{def:deriv-support} as well as Lemma~\ref{lem:partialTx-initial}. Together with the conclusions of the previous sections we obtain the main result of this paper:

\begin{theorem}[$\aca_0$]\label{thm:main-result}
The following are equivalent:
\begin{enumerate}[label=(\arabic*)]
\item If $T$ is a normal dilator, then $D^{\partial T}_X$ is well-founded for any well-order~$X$.
\item Any normal dilator $T$ has an upper derivative $(S,\xi)$ such that $X\mapsto D^S_X$ preserves well-foundedness (which means that $S$ is again a normal dilator).
\item The principle of $\Pi^1_1$-bar induction is valid.
\end{enumerate}
\end{theorem}
Note that statements~(1) and~(2) are each expressed by a single formula, relying on the formalization of dilators in Section~\ref{sect:normal-dils-so}. To express~(3) by a single formula one uses a truth definition for $\Pi^1_1$-sentences.
\begin{proof}
The implication (1)$\Rightarrow$(2) follows from Proposition~\ref{prop:partial-upper-deriv}, which asserts that $(\partial T,\xi^T)$ is an upper derivative of $T$ (in fact we have a derivative, by Theorem~\ref{thm:terms-derivative}). By Corollary~\ref{cor:deriv-implies-BI} we get (2)$\Rightarrow$(3) (note that the proof uses arithmetical comprehension, via the Kleene normal form theorem and the characteristic property of the Kleene-Brouwer order). The implication (3)$\Rightarrow$(1) holds by Theorem~\ref{thm:BI-yields-deriv}.
\end{proof}

As in the previous section, results about $\partial T$ transfer to arbitrary derivatives:

\begin{corollary}[$\rca_0+\mathbf{\Pi^1_1}\textbf{-BI}$]\label{cor:deriv-is-dilator}
Consider a normal dilator~$T$. If $(S,\xi)$ is a derivative of $T$, then $X\mapsto D^S_X$ preserves well-foundedness.
\end{corollary}
\begin{proof}
By Definition~\ref{def:derivative} and Proposition~\ref{prop:partial-upper-deriv} there is a morphism $\nu:S\Rightarrow\partial T$ of upper derivatives. According to Lemma~\ref{lem:morph-dilators-extend} we get an embedding
\begin{equation*}
D^\nu_X:D^S_X\rightarrow D^{\partial T}_X
\end{equation*}
for each linear order~$X$. Together with Theorem~\ref{thm:BI-yields-deriv} it follows that $D^S_X$ is well-founded whenever $X$ is a well-order.
\end{proof}

Working in a sufficiently strong set theory, one can deduce the following unconditional version of Theorem~\ref{thm:equalizer-to-derivative}. This result provides further justification for our categorical definition of derivatives:

\begin{corollary}\label{cor:deriv-dil-to-fct}
Let $T$ be a normal dilator. If $(S,\xi)$ is a derivative of $T$, then the function $\alpha\mapsto\otp(D^S_\alpha)$ is the derivative of the normal function $\alpha\mapsto\otp(D^T_\alpha)$.
\end{corollary}
\begin{proof}
According to Theorem~\ref{thm:partial-derivative} and Corollary~\ref{cor:deriv-is-dilator} the assumptions of Theorem~\ref{thm:equalizer-to-derivative} are satisfied whenever $(S,\xi)$ is a derivative of $T$.
\end{proof}

\bibliographystyle{amsplain}
\bibliography{Derivatives_RM}

\end{document}